\documentclass[a4paper]{article}
\usepackage{amsmath,amsthm,amsfonts,amssymb}
\usepackage{thmtools}
\usepackage[T1]{fontenc}
\usepackage{placeins}
\usepackage{graphicx}
\usepackage{caption, color}
\usepackage{subcaption}
\usepackage{euscript}
\usepackage{units}
\usepackage{setspace}
\usepackage{enumitem}

\usepackage{float}
\usepackage{wrapfig}
\usepackage{esint}
\usepackage{multirow,multicol,enumerate}

\newtheorem{theorem}{Theorem}[section]
\newtheorem{lemma}[theorem]{Lemma}
\newtheorem{corollary}[theorem]{Corollary}

\theoremstyle{definition}
\newtheorem{definition}[theorem]{Definition}

\newtheorem{remark}[theorem]{Remark}
\newtheorem*{remark*}{Remark}

\newtheorem{algorithm}[theorem]{Algorithm}
\numberwithin{equation}{section}
\numberwithin{figure}{section}
\numberwithin{table}{section}

\newcommand{\dx}{\, dx}

\newcommand{\dt}{\, dt}
\newcommand{\bbT}{\mathbb T}
\newcommand{\bbR}{\mathbb R}
\newcommand{\bbS}{\mathbb S}
\newcommand{\bbZ}{\mathbb Z}
\newcommand{\bbN}{\mathbb N}
\renewcommand{\div}{{\rm div}}
\newcommand{\curl}{{\rm curl}}

\newcommand{\wto}{\rightharpoonup}
\newcommand{\wsto}{\overset{*}{\wto}}
\newcommand{\BP}{{\large {\bf P}}}
\newcommand{\BQ}{{\large {\bf Q}}}
\newcommand{\BI}{{\large {\bf I}}}
\newcommand{\BJ}{{\large {\bf J}}}
\newcommand{\cL}{{\mathcal L}}
\begin{document}

\title{Computation of measure-valued solutions for the incompressible Euler equations}
\author{Samuel Lanthaler \thanks{
 Department of Mathematics, ETH Z\"urich,  Z\"urich -8092,
  Switzerland}, \, Siddhartha Mishra \thanks{
Seminar for Applied Mathematics (SAM)  
Department of Mathematics, ETH Z\"urich, 
HG G 57.2, Z\"urich -8092,
  Switzerland (and) 
Center of Mathematics for Applications (CMA),University of Oslo, 
 P.O.Box -1053, Blindern, Oslo-0316, Norway} {}\thanks{S.M. was supported in part by ERC STG. N 306279, SPARCCLE.}} 
\date{\today}

\maketitle

\begin{abstract}
We combine the spectral (viscosity) method and ensemble averaging to propose an algorithm that computes admissible measure valued solutions of the incompressible Euler equations. The resulting approximate young measures are proved to converge (with increasing numerical resolution) to a measure valued solution. We present numerical experiments demonstrating the robustness and efficiency of the proposed algorithm, as well as the appropriateness of measure valued solutions as a solution framework for the Euler equations. Furthermore, we report an extensive computational study of the two dimensional vortex sheet, which indicates that the computed measure valued solution is non-atomic and implies possible non-uniqueness of weak solutions constructed by Delort. 
\end{abstract}

\section{Introduction}
Many interesting fluid flows are characterized by very low Mach numbers and very high Reynolds numbers \cite{CM1}. It is customary to model these flows by the incompressible Euler equations:
\begin{gather}\label{EQ}
\left\{
	\begin{aligned}
		\partial_t v + \div(v\otimes v) + \nabla p &= 0, \\
		\div(v) &= 0.
	\end{aligned}
	\right.
\end{gather}
Here, the velocity field is denoted by $v$. The pressure $p$ acts as a Lagrange multiplier to impose the divergence constraint and can be eliminated by projecting \eqref{EQ} to divergence free velocity fields \cite{bertozzi}. The system of equations is augmented with initial and boundary conditions. We will only consider the case of periodic boundary conditions in this paper.
\subsection{Theoretical results} 
Although short time existence and uniqueness results for smooth solutions of the incompressible Euler equations are classical \cite{bertozzi}, there are no rigorous \emph{global} wellposedness results for admissible (finite kinetic energy) weak solutions of \eqref{EQ} in three space dimensions. In fact, DeLellis and Szekelyhidi \cite{CDL1,CDL2} were recently able to construct infinitely many admissible weak solutions for the 3D Euler equations. See also \cite{Sch1,Sn1,CDL3} for related non-uniqueness results. 

A major obstacle for the development of a well-posedness theory in three space dimensions lies in the behavior of the vorticity $\eta := {\rm curl} v$, which satisfies (formally),
\begin{equation}
\label{eq:vor}
\eta_t + (v\cdot \nabla) \eta = -(\eta\cdot \nabla) v.
\end{equation}
The right hand side of the vorticity equation \eqref{eq:vor} represents \emph{vortex stretching} and may lead to an unbounded growth of vorticity if the velocity is not smooth. 

The existence and uniqueness theory in two space dimensions is much better developed \cite{CM1,bertozzi}. This is largely due to the fact that the vorticity is a scalar and satisfies \eqref{eq:vor} without the vortex stretching term (right-hand-side). Hence, the vorticity stays bounded (in $L^{\infty}$) as long as we consider initial velocity fields with bounded vorticities. 

However, many interesting flows do not start with initially bounded vorticities. A prototypical example is provided by vortex sheets \cite{bertozzi}, which arise frequently in shear flows. Here, the velocity field is only piecewise smooth, with a discontinuity across a one-dimensional interface. Hence, the resulting vorticity is merely a bounded measure and the wellposedness results of \cite{CM1,bertozzi} are not valid.

The first proof of existence of such vortex sheets was provided in a celebrated paper of Delort \cite{Del1}. The author showed that if the initial data is a bounded non-negative measure (i.e $\in \mathcal{BM}_+$) and lies in $H^{-1}$, then one can construct an weak solution of \eqref{EQ} with the resulting vorticity in $H^{-1}\cap \mathcal{BM}_+$. The question of uniqueness of such solutions remains open.

More recently, Szekelyhidi \cite{Sz1} was able to construct infinitely many admissible weak solutions of \eqref{EQ} with vortex sheet initial data i.e in $H^{-1}\cap \mathcal{BM}_+$. However, the resulting solutions are highly oscillatory and it is unclear if they belong to the Delort class \cite{Sz1}.

\subsection{Numerical methods}
A large variety of numerical methods have been developed to discretize the incompressible Euler equations. Spectral methods, based on an approximation in Fourier space, are widely employed, particularly in the simulation of flows in periodic domains  \cite{Oz1}. Adding spectrally small numerical diffusion results in the spectral viscosity method \cite{Tad1,BT1} that can approximate sharp gradients robustly.

Finite difference, element and volume projection methods {\cite{Cho1,Cho2} are predictor-corrector methods that are very popular in the approximation of flows in domains with boundaries. In this method, a hyperbolic advection step is corrected with an elliptic solve to impose the divergence constraint. Vortex methods, based on discretizing a Lagrangian version of the vorticity equation \eqref{eq:vor} are frequently used to discretize the Euler equations, particularly in two space-dimensions \cite{Cho3,Kras1,Kras2,SL1,Maj1,bertozzi}. 

Although the aforementioned methods are extensively used, rigorous convergence results for them are only available when the underlying solution is smooth \cite{Cho2,Maj1,BT1}. Hence, it is unclear if these methods will converge (with increasing numerical resolution) to an appropriate solution of the Euler equations \eqref{EQ}, either in the three dimensional case or for two-dimensional flows with rough initial data. 
\subsection{Two contrasting numerical examples}
Given the lack of rigorous convergence results, we resort to empirically studying the convergence of numerical methods for two-dimensional flows. To this end, we use a spectral method (see section \ref{sec:spec} for details of the method) to approximate the Euler equations \eqref{EQ} in the two-dimensional periodic box $[0,2\pi]^2$.

First, we consider a perturbed rotating vortex patch \eqref{eq:vp1} as initial data and present the computed vorticity in figure \ref{fig:1}. The figure clearly shows convergence of the vorticity as the number of Fourier modes is increased. This convergence is further verified in figure \ref{fig:2} (left) where we plot the difference in $L^2$ of approximate velocity fields at successive resolutions \eqref{eq:diff}. The figure show that this difference decreases as the resolution is increased and indicates convergence. Furthermore, we also infer stability of the computed solutions, with respect to perturbations of initial data, from figure \ref{fig:2} (right), where we have plotted the $L^2$ difference in computed solutions for successively smaller values of the perturbation parameter.

Next, we consider a perturbed version of the flat vortex sheet, specified by the initial data \eqref{eq:vs} and visualized in figure \ref{initdata}. We visualize the resulting flow by presenting a passive tracer (advected by the computed velocity field) in figure \ref{fig:5}. We observe that the initial perturbation is magnified by Kelvin-Helmholtz instabilities and the tracer seems to be mixed at smaller and smaller scales as the numerical resolution is increased, indicating a possible lack of convergence of the numerical method. This is further verified in figure \ref{fig:6} (left) where we plot the $L^2$ differences at successive resolutions \eqref{eq:diff}. In complete contrast to the rotating vortex patch, the plot clearly shows that the approximations do not even form a Cauchy sequence, let alone converge. Furthermore, the lack of stability of this flow with respect to perturbations is seen from figure \ref{fig:6} (right). In particular, small scale features continue to persist and proliferate further when even the perturbation amplitude is reduced.

Summarizing, the above numerical experiment strongly suggests that numerical approximations of flows, corresponding to non-smooth initial data in two dimensions, may not necessarily converge to a weak solution, at least for realistic numerical resolutions. Although the above results are obtained with a spectral method, a finite difference projection method yielded exactly the same behavior, see the forthcoming paper \cite{LeoSid1}. We strongly suspect that all available numerical methods will not be convergent for this numerical example.
\subsection{Aims and scope of the current paper}
This observed lack of convergence of numerical methods, approximating vortex sheets, was largely on account of the appearance of structures at finer and finer scales as the numerical resolution was increased. This phenomenon was also noticed in the case of the compressible Euler equations in a recent paper \cite{FKMT1}. Given the presence of structures at infinitesimally small scales, the authors of \cite{FKMT1} proposed that measure valued solutions are an appropriate framework to study the question of convergence of numerical approximations. We will follow this approach here. 

Measure valued solutions, introduced by DiPerna and Majda in \cite{DM1,DM2}, are \emph{Young measures} i.e, space-time parametrized probability measures that satisfy the incompressible Euler equations in a weak sense. Given that finite kinetic energy ($L^2$) bounds are the only available global a priori estimate for \eqref{EQ}, Di Perna and Majda proposed that both fine scale oscillations as well as concentrations can prevent strong convergence of approximation schemes. Hence, they proposed a notion of \emph{generalized young measures} that encode both effects. In \cite{DM1}, the authors proved existence of measure valued solutions by showing that Leray-Hopf weak solutions of the Navier-Stokes equations converge to a measure valued solution of \eqref{EQ} as the viscosity tends to zero.

However, the task of computing measure valued solutions is extremely challenging as a large number of refinements in resolution have to be made in order to approximate statistics with respect to the measure \cite{FKMT1}. This necessitates the use of ultra-fine grids (in physical or fourier space) and is enormously computationally expensive. An alternative strategy was proposed in \cite{FKMT1} in the context of compressible flows and is based on exploiting the relationship between young measures and random fields. The resulting algorithm generates an ensemble of numerical simulations and approximates statistics by Monte Carlo sampling, thus computing the underlying measure valued solution efficiently.

The \emph{first aim} of the current paper is to compute measure valued solutions of the incompressible Euler equations \eqref{EQ} efficiently. To this end, we will extend the ensemble based algorithm of \cite{FKMT1} to the incompressible Euler equations. In contrast with \cite{FKMT1}, where high-resolution finite difference (volume) methods were employed, we will use a spectral (viscosity) method and show that the resulting approximations  
 \emph{converge} (upto a subsequence) to an admissible measure valued solution. 

The \emph{second aim} of the current paper is to utilize the proposed algorithm in order to perform an extensive numerical case study of the flat vortex sheet, In particular, we will demonstrate that the proposed algorithm will converge to a \emph{stable} measure valued solution that is \emph{non-atomic}. Consequently, we show that this non-atomicity suggests \emph{non-uniqueness} of weak solutions of the Euler equations in the class considered by Delort in \cite{Del1}

The rest of the paper is organized as follows -- in section \ref{sec:mvs}, we recall the notion of admissible measure valued solutions of the incompressible Euler equations. The spectral method and ensemble averaging based numerical algorithm is presented in section \ref{sec:spec} and numerical experiments are presented in section \ref{sec:numex}. An extensive case study of the flat vortex sheet is described in section \ref{sec:vs}.
\section{Measure valued solutions}
\label{sec:mvs}
We start by recapitulating the definition of admissible measure valued solutions in the sense of DiPerna and Majda \cite{DM1}. In \cite{T1}, Tartar described possible fine scale oscillations of a sequence of functions, bounded uniformly in $L^{\infty}$, in terms of a young measure. However, given that only $L^2$ bounds are available for the solutions of \eqref{EQ}, we follow \cite{DM1,AB1,BDS1,FSID3} and introduce \emph{generalized young measures}, that account for both fine scale oscillations as well as concentration in the (composite) weak limit. 

\begin{definition}
\label{def:gym}
Given a sequence of functions $u_n$ such that $u_n$ is uniformly bounded in $L^2_{\mathrm{loc}}([0,\infty) \times \bbT^n;\bbR^n)$, the triple $(\nu,\lambda,\nu^\infty)$ consisting of
\begin{itemize}
\item the oscillation measure $\nu \in L^\infty([0,+\infty)\times \bbT^n;\mathcal{P}(\bbR^n))$, which is a probability measure on phase space $\bbR^n$ parametrized by $x,t$, and accounts for the persistence of oscillations in the sequence $u_n$,
\item the concentration measure $\lambda = \lambda_t \otimes \dt$, which is a measure on physical space-time $\bbT^n \times [0,+\infty)$ and is singular with respect to Lebesgue measure $\dx \dt$,
\item the concentration-angle measure $\nu^\infty \in L^\infty([0,+\infty)\times \bbT^n;\mathcal{P}(\bbS^{n-1}))$, a probability measure on $\bbS^{n-1}$ parametrized by $x,t$.
\end{itemize}
is termed as the (generalized) Young measure associated with the sequence $u_n$ provided that (after extracting a subsequence, still labeled by $n$), the following holds,
\begin{align}\label{Younglimit}
f(u_n) \dx \dt \wsto \left(\int_{\bbR^n} f \, d\nu_{x,t}\right) \dx \dt + \left( \int_{\bbS^{n-1}} f^\infty \, d\nu^\infty_{x,t}\right) \, \lambda(\dx\dt),
\end{align}
for every continuous function on phase space $f\in C(\bbR^n)$ with continuous $L^2$- recession function 
$$f^\infty(\theta) = \lim_{s\to \infty} s^{-2}f(s\theta).$$
\end{definition}
The existence of such Young measures has been proved in \cite{AB1}. One can further extend the notion of (generalized) Young measures to represent the
effect of oscillations and concentrations in sequences of Young measures \cite{FSID3}.

We will frequently use the notation $\langle \nu_{x,t}, f\rangle := \int_{\bbR^n} f \, d\nu_{x,t}$ to denote the action of a probability measure $\nu_{x,t}$ on a function $f$, in the following.

\begin{definition}
\label{mvsdef}
A generalized Young measure $(\nu, \lambda, \nu^\infty)$ is a measure-valued solution (MVS) of the incompressible Euler equations \eqref{EQ} with initial data $v_0$, if it satisfies 
\begin{align}
\begin{split}
\int_0^\infty & \int_{\bbT^n} \langle \nu_{x,t}, \xi\rangle \partial_t \phi + \langle\nu_{x,t},\xi \otimes \xi \rangle : \nabla \phi \, \dx \dt 
\\ &+ \int_0^\infty \int_{\bbT^{n-1}} \langle\nu^\infty_{x,t},\theta \otimes \theta \rangle : \nabla \phi \; \lambda(\dx \dt) + \int_{\bbT^n} v_0(x) \phi(x, 0) \, \dx  = 0
\end{split}
\label{MVS}
\end{align}
for all $\phi \in C^\infty_c([0,\infty)\times \bbT^n; \bbR^n)$ with $\div(\phi) = 0$, and 
\begin{align*}
\int_0^\infty \int_{\bbT^n} \langle \nu_{x,t}, \xi\rangle \cdot \nabla \psi = 0
\end{align*}
for $\psi \in C^\infty_c([0,\infty)\times \bbT^n)$.
\end{definition}
Note that if the Young measure is atomic i.e, $\nu_{x,t} = \delta_{v(x,t)}$, $\lambda = 0$, then the definition of measure valued solutions reduces to the standard notion of weak solutions of \eqref{EQ}.
\begin{remark}
The above definition assumes that the initial data is an atomic measure. It is rather straightforward to generalize this notion to an initial data that is a non-atomic Young measure $\sigma$ by replacing $v_0$ in \eqref{MVS} with $\langle \sigma_x, \xi \rangle$.  This notion can also serve as a framework for uncertainty quantification (UQ), in the context of uncertain initial data \cite{FKMT1}.
\end{remark}
\subsection{Existence and uniqueness}
As mentioned in the introduction, the existence of a measure valued solution of \eqref{EQ} was shown by DiPerna and Majda in \cite{DM1}. 

We observe that the definition of measure valued solution provides for the dynamic evolution of the mean (first moment) of the Young measure $\nu$ only. The evolution of higher moments cannot be inferred directly from the evolution equation \eqref{MVS}. Hence, the concept of measure valued solutions contains a large scope for non-uniqueness of the resulting solutions. Further admissibility criteria need to be prescribed in order to recover uniqueness.

A natural admissibility criteria (see \cite{BDS1}) stems from the conservation (dissipation) of kinetic energy by the flow. Following \cite{BDS1}, we define
\begin{definition}\label{admissdef}
A MVS $(\nu, \lambda, \nu^\infty)$ with initial data $\sigma$, is called \emph{admissible}, if
\begin{align*}
 \int_{\bbT^n} \langle \nu_{x,t}, \vert \xi \vert^2 \rangle \dx +  \lambda_t(\bbR^{n}) \le \int_{\bbT^n} \langle \sigma_x, \vert \xi \vert^2\rangle \dx
\end{align*}
for almost all $t \in [0,\infty)$.
\end{definition}

Thus, we require that the measure valued solution either conserve or dissipate kinetic energy (in time). Energy dissipative solutions of the Euler equations are widely accepted as playing a key role in turbulence following the seminal work of Onsager \cite{ES}. Recent progress towards the construction of energy dissipative (weak) solutions  in three space dimensions is reported in \cite{DS4,IS1}.

However, the notion of admissibility does not suffice in restoring uniqueness. In fact, Szekelyhidi and Wiedemann \cite{SW1} show that any admissible measure valued solution can be approximated by a sequence of weak solutions. This partly reflects the fact that admissible weak solutions (atomic measure valued solutions) are non-unique. However, the notion of admissibility is enough to obtain the following  weak-strong uniqueness result,
\begin{theorem}\label{weakstronguniqueness}
\cite{BDS1}:
Let $v \in C([0,T];L^2(\bbT^n;\bbR^n))$ be a weak solution of \eqref{EQ} with $\int_0^T \Vert \nabla v + \nabla v ^T \Vert_{L^\infty} \dt < \infty$ and let $(\nu, \lambda, \nu^\infty)$ be an admissible measure-valued solution with atomic initial data $\sigma = \delta_{v(x,0)}$. Then $\nu_{x,t} = \delta_{v(x,t)}$ and $\lambda = 0$, i.e. $v$ is the unique admissible MVS in this situation. 
\end{theorem}
Thus, admissible measure valued solutions of the Euler equations coincide with a classical solution if it exists. The question of uniqueness (stability) of admissible measure valued solutions is further investigated in section \ref{sec:numex}. 

\section{Construction of admissible measure valued solutions}
\label{sec:spec}
From the previous section, we have seen that admissible measure valued solutions of the Euler equations exist (globally in time) and can be realized as a zero viscosity limit of weak solutions of Navier-Stokes equations. Our interest will be in computing measure valued solutions in an efficient manner. To this end, we will adapt the recent ensemble based algorithm of \cite{FKMT1}, for simulating compressible flows, to the current context of incompressible flows. As the algorithm of \cite{FKMT1} combines Monte Carlo sampling in probability space with a robust numerical discretization of the underlying PDEs, we start with a description of our choice of discretization framework for the incompressible Euler equations.
\subsection{Spectral (viscosity) methods}
Spectral methods approximate the Euler equations \eqref{EQ} in Fourier space \cite{Oz1}. Let $\bbT^n$ denote the n-dimensional torus. Let $(v,p)$ be solutions 
of the Euler equation \eqref{EQ} with periodic boundary conditions, then it satisfies,
\begin{gather}\label{Euler}
\left\{
	\begin{aligned}
		\partial_t v + v\cdot \nabla v + \nabla p &= 0, \\
		\div(v) &= 0.
	\end{aligned}
	\right.
\end{gather}
defined on $\bbT^n \times \bbR_{+}$. 

Consider the spatial Fourier expansion $v(x,t) = \sum_k \widehat v_k(t) e^{ikx}$ with coefficients given by
\begin{align*}
\widehat v_k(t) &= \frac{1}{(2\pi)^n}\int_{\bbT^n} v(x,t) \, e^{-ikx} \, dx.
\end{align*}
If $v$ is a solution of \eqref{Euler}, the above expression yields
\begin{align*}
\frac{d}{dt} \widehat v_k
&= \frac{1}{(2\pi)^n} \int_{\bbT^n} \partial_t v \, e^{-ikx} \dx \\
&= -\frac{1}{(2\pi)^n} \int_{\bbT^n} \left(v\cdot \nabla v + \nabla p\right)e^{-ikx} \dx \\
&= -\frac{1}{(2\pi)^n} \int_{\bbT^n} \sum_{\substack{\ell, m} }  (\widehat v_\ell \cdot im)\widehat v_m e^{i(\ell + m - k)x} \\ 
&\quad  -\frac{1}{(2\pi)^n} \int_{\bbT^n}  \sum_{\ell} \widehat p_\ell i\ell e^{i(\ell-k)x} \dx \\
&= (-i) \sum_{\substack{\ell, m \\ \ell + m -k = 0} } (\widehat v_\ell \cdot m)\widehat v_m -i \widehat p_k k.
\end{align*}
We note that $\div(v) = i \sum_k (\widehat v_k\cdot k) e^{ikx}=0$ is equivalent to $\widehat v_k \perp k$ for all $k$. Using $m = k-\ell$ and $\widehat v_\ell \perp \ell$ for all terms in the summation, we can thus rewrite the last equation in the form
\begin{equation} \label{EQF}
\frac{d}{dt} \widehat v_k= (-i) \sum_{\substack{\ell,m \\ \ell + m -k = 0} } (\widehat v_\ell \cdot k)\widehat v_m -i \widehat p_k k.
\end{equation}
This is the Fourier space version of the Euler equations \eqref{EQ}.
It becomes evident that the pressure term $-i \widehat p_k k$, which is collinear to $k$, serves as the orthogonal $L^2$ projection of the non-linear term $$(-i) \sum_{\substack{\ell,m \\ \ell + m -k = 0} } (\widehat v_\ell \cdot k)\widehat v_m$$ to the orthogonal complement of $k$, thus keeping $v$ divergence-free. \\

For the coefficient $\widehat v_k$ with $k=0$, equation \eqref{EQF} yields $\frac{d}{dt} \widehat v_0 = 0$. This corresponds to \emph{conservation of momentum}. Using Galilean invariance of the Euler equations, it implies that we can without loss of generality assume that $\widehat v_0 =  \frac{1}{(2\pi)^n}\int_{\bbT^n} v \, dx = 0$, in the following. \\

\subsubsection{Semi-discretization in Fourier space}

To obtain a discretized approximation to system \eqref{EQF}, we restrict our attention to only the Fourier modes below some threshold $N$. We thus consider divergence-free fields of the form $v(x,t) = \sum_{|k|\le N} \widehat v_k(t) \, e^{ikx}$, and we have to project the non-linear term to this space. We denote the corresponding projection operator by $\BP_N$. The projection operator is a combination of both Fourier truncation and projection to the space of divergence-free fields. More explicitely, $\BP_N$ is given by
\[\BP_N\left( \sum_{k\in \bbZ^2} \widehat w_k e^{ikx} \right)= \sum_{|k|\le N} \left(\widehat w_k- \frac{\widehat w_k \cdot k}{|k|^2} k\right)e^{ikx},\]
yielding a divergence-free vector field with Fourier modes $|k| \le N$. 
We can also add a small amount of numerical viscosity to ensure stability of the resulting scheme. 

This idea results in the following scheme: For given initial data $v_0(x)$, we obtain an approximate solution $v_N(x,t) \approx v(x,t)$ by solving the finite-dimensional problem 
\begin{gather}\label{spectralNS}
\left\{
	\begin{aligned} 
		\partial_t v_N + \BP_N\left(v_N\cdot \nabla v_N\right) &= \theta \Delta v_N, \\
		v_N(x,0) = \BP_Nv_0(x).
	\end{aligned}
	\right.
\end{gather}
In this scheme, the small number $\theta = \theta(N) > 0$ depends on $N$ and $\theta \to 0$ as $N\to \infty$. 

A refined version of this basic scheme was introduced by Tadmor \cite{Tad1}. In that version, we choose a small number $\varepsilon  > 0$ and an integer $m \le N$. The integer $m$ serves as a threshold between small and large Fourier modes. We apply a viscous regularization only to the large Fourier modes. With a judicious choice of $\varepsilon  = \varepsilon (N)$, $m=m(N)$, the resulting method can be shown to be spectrally accurate \cite{spectralerrorestimate}, \cite{BT1}. We obtain the corresponding approximate system
\begin{gather}\label{spEQ}
\left\{
	\begin{aligned} 
		\partial_t v_N + \BP_N\left(v_N\cdot \nabla v_N\right) &= \varepsilon  \div\left( \BQ_N \nabla v_N\right), \\
		v_N(x,0) = \BP_Nv_0(x),
	\end{aligned}
	\right.
\end{gather}
where $\BQ_N = \BI-\BP_m$, denotes the projection onto the higher modes.
System \eqref{spEQ} includes \eqref{spectralNS} for the special choice $m = 0$, $\varepsilon  = \theta$.

The resulting (semi-discrete) spectral method  \eqref{spEQ} is stable in the following sense,
\begin{lemma} \label{est}
{\em Stability:} If $v_N$ is the solution of the semi-discrete system \eqref{spEQ}, then
\begin{align}
\label{enest} 
		 \frac12 \lVert v_N \rVert_{L^2}^2 + \varepsilon \int_0^t \lVert (\BI-\BP_m)\nabla v_N \rVert_{L^2}^2 \dt
		&=  \frac12 \lVert \BP_N v_0 \rVert_{L^2}^2 \le  \frac12 \lVert v_0 \rVert_{L^2}^2.
\end{align}
In particular, we have $ \lVert v_N \rVert_{L^2} \le \lVert v_0 \rVert_{L^2}$, independently of $N, m,  \varepsilon $.
\end{lemma}
\begin{proof}
 Multiplying equation \eqref{spEQ} by $v_N$ and integrating over space, we obtain after an integration by parts, and using also the fact that all boundary terms vanish due to the periodic boundary conditions,
\begin{align*} 
		\frac{d}{dt} \frac12 \int_{\bbT^2} |v_N|^2 \dx 
		&= \int_{\bbT^2} -v_N\cdot \BP_N\left(v_N\cdot \nabla v_N \right)+v_N\cdot \varepsilon  \div \left( (\BI - \BP_m) \nabla v_N\right) \dx \\
		&= \int_{\bbT^2} -\BP_Nv_N\cdot \left(v_N\cdot \nabla v_N \right)-\varepsilon  \nabla v_N:   (\BI - \BP_m) \nabla v_N \dx \\
		&= \int_{\bbT^2} -v_N\cdot \left(v_N\cdot \nabla v_N \right)-\varepsilon  (\BI - \BP_m)  \nabla v_N:   (\BI - \BP_m) \nabla v_N \dx \\
		&= \int_{\bbT^2} -\div\left(\frac 12 |v_N|^2 v_N\right) - \varepsilon  \vert(\BI - \BP_m) \nabla v_N \vert^2  \dx \\
		&= - \varepsilon  \int_{\bbT^n} \vert (\BI - \BP_m) \nabla v_N \vert^2 \dx,
\end{align*}
i.e. we have
\begin{align}
\frac{d}{dt} \int_{\bbT^n} \frac12 \vert v_N \vert^2 \dx + \varepsilon  \int_{\bbT^n} \vert (\BI - \BP_m) \nabla v_N \vert^2 \dx &= 0.
\end{align}
Integrating in time from $0$ to $t$ yields \eqref{enest}.
\end{proof}
Furthermore, the spectral (viscosity) method also satisfies the following consistency property,
\begin{lemma}
\label{lem:cons}
{\em Consistency:} Let $v_N$ be an approximate solution of the Euler equations computed by the spectral (viscosity) method \eqref{spEQ}, then for any domain $D \subset   \bbT^n \times \bbR^+$ and all divergence-free test functions $\phi$ in $C_0^\infty(D)$, 
\begin{equation}
\label{eq:cons1}
\lim_{N \to \infty} \int _D \partial_t\phi \cdot v_N + \nabla \phi:v_N \otimes v_N \dx \dt = 0.
\end{equation}
\end{lemma}
\begin{proof}
Rewriting \eqref{spEQ} and using that $\div(v_N)=0$, we have
$$\partial_t v_N + \div\left(v_N\otimes v_N\right)+ \nabla p_N = \div\left((\BI-\BP_N)(v_N\otimes v_N)\right) + \varepsilon \, \div\left( (\BI-\BP_m)\nabla v_N\right).$$
 for any divergence-free test function $\phi\in C_0^\infty(D)$  we obtain after an integration by parts
\begin{align*}
\int_D \partial_t\phi\cdot v_N + \nabla \phi:v_N \otimes v_N \dx \dt 
&= - \int_D \phi \cdot \left(\partial_tv_N + \div(v_N \otimes v_N)\right) \dx \dt \\
&= - \int_D \phi\cdot \div\left( (\BI-\BP_N)\left(v_N\otimes v_N\right)\right) \dx\dt\\
&\qquad  -\varepsilon\int_D \phi \cdot  \div\left( (\BI-\BP_m)\nabla v_N\right) \dx \dt \\
&= (A) - (B).
\end{align*}
Now
\begin{align*}
(A) &= - \int_D \phi\cdot \div\left( (\BI-\BP_N)\left(v_N\otimes v_N\right)\right) \dx\dt  \\
&= \int_D \nabla\phi : \left( (\BI-\BP_N)\left(v_N\otimes v_N\right)\right) \dx\dt \\
&= \int_D \nabla (\BI-\BP_N)\phi : \left(v_N\otimes v_N\right) \dx\dt\\
\end{align*}
We notice that for a constant $C_n$ depending on the space dimension $n$ only:
\begin{align*}
\Vert v_N \Vert_{L^\infty_x} 
&\le \sum_{\vert k\vert \le N} \vert \widehat {(v_N)}_k\vert \le C_n N^{n/2} \Big(\sum_{\vert k\vert \le N} \vert \widehat {(v_N)}_k\vert^2 \Big)^{1/2} \\
&= C_n N^{n/2} \Vert v_N \Vert_{L^2_x} \le C_n N^{n/2} \Vert v_0 \Vert_{L^2}.
\end{align*}
Thus we can continue to estimate the term $(A)$ as
\begin{align*}
\vert (A) \vert &\le  \int_0^T \Vert v_N \Vert_{L^\infty_x} \Vert v_N \Vert_{L^2_x} \Vert \nabla (\BI-\BP_N)\phi \Vert_{L^2_x}  \dt \\
&\le C_n N^{n/2} \Vert v_0 \Vert^2_{L^2} \int_0^T \Vert \nabla (\BI-\BP_N)\phi \Vert_{L^2_x} \dt \\
&\le C_n \Vert v_0 \Vert^2_{L^2} \int_0^T \Vert (\BI-\BP_N)\phi \Vert_{H^{n/2+1}_x} \dt
\end{align*}
Since $\phi$ is smooth, it follows that $ \int_0^T \Vert (\BI-\BP_N)\phi \Vert_{H^{n/2+1}_x} \dt \to 0$ as $N\to \infty$. Hence, we obtain that $(A)\to 0$. 

The term $(B)$ is handled similarly. We have
\begin{align*}
(B) &= \varepsilon\int_D \phi \cdot  \div\left( (\BI-\BP_m)\nabla v_N\right) \dx \dt \\
&= \varepsilon\int_D \div\left( (\BI-\BP_m)\nabla\phi \right) \cdot  v_N \dx \dt\\
&= \varepsilon\int_D \left( (\BI-\BP_m)\Delta\phi \right) \cdot  v_N \dx \dt.
\end{align*}
This yields
$$|(B)| \le \varepsilon  \int_0^T \Vert (\BI-\BP_m)\phi \Vert_{H^2_x} \Vert v_N\Vert_{L^2}\dt \le \varepsilon   \Vert v_0\Vert_{L^2} \int_0^T \Vert (\BI-\BP_m)\phi \Vert_{H^2_x} \dt,$$
and again we see that the right hand side converges to $0$, if either $\varepsilon  \to 0$ or $m\to \infty$.
\end{proof}
\begin{remark}
The arguments used in the derivation of the estimates for Lemma \ref{lem:cons} also yield uniform Lipschitz continuity 
$$v_N \in \mathrm{Lip}([0,T];H^{-n/2-1}(\bbT^n;\bbR^n))$$
 for $n\ge 2$. Indeed, we have 
\begin{equation}\label{timecont}
\partial_t v_N = -\BP_N\div(v_N\otimes v_N)+\varepsilon  \div((\BI - \BP_m)\nabla v_N).
\end{equation}
An $H^{-n/2-1}$-bound for the first term is obtained in the estimate for (A), while the estimate for (B) implies an upper bound on the $H^{-2}$-norm. From the inclusion $H^{-2} \subset H^{-n/2-1}$, we obtain an upper bound in $H^{-n/2-1}$ for the right hand side of \eqref{timecont}. Upon integration in time, the claimed Lipschitz continuity of $v_N$ in $H^{-n/2-1}$ as a function of $t$ follows. 
\end{remark}
\subsection{An ensemble based algorithm to compute admissible measure valued solutions}
Next, we will combine the spectral (viscosity) method with the ensemble based algorithm of the recent paper \cite{FKMT1} in order to compute admissible measure valued solutions of the incompressible Euler equations. First, we assume that the initial velocity field is an arbitrary Young measure i.e, $v(x,0) = \sigma_x$, which satisfies the divergence constraint in a weak sense. Then, an algorithm for computing measure valued solutions is specified with the following steps:
\begin{algorithm}\label{alg:approxmv}~
\begin{description}
\item[\textbf{Step 1:}] Let $v_0: \Omega \mapsto L^{2} (\bbT^n;\bbR^n)$ be a random field on a probability space $(\Omega,\mathcal F,P)$ such that the initial Young measure $\sigma$ is the law of the random field $v_0$. The existence of such a random field can be shown analogous to \cite{FKMT1}, proposition A.3.
\item[\textbf{Step 2:}] We evolve the initial random field $v_0$ by applying the spectral (viscosity) scheme \eqref{spEQ} for every $\omega \in \Omega$ to obtain an approximation $v_N(\omega)$, to the solution random field $v(\omega)$, corresponding to the initial random field $v_0(\omega)$.
\item[\textbf{Step 3:}] Define the approximate measure-valued solution $\nu^N$ as the law of $v_N$.
\end{description}
\end{algorithm}
Then from proposition A. 3. 1 of the recent paper \cite{FKMT1}, $\nu^N$ is a Young measure. Next, we show that these approximate Young measures will converge in the appropriate sense to an admissible measure valued solution of the incompressible Euler equations \eqref{EQ}.
\begin{theorem}
\label{thm:conv}
Let the (kinetic) energy of the initial Young measure $\sigma$ be finite i.e,
$$
\int_{\bbT^n} \langle \sigma_x, |\xi|^2 \rangle dx \leq C < \infty,
$$
then the approximate Young measure $\nu^N_{x,t}$, generated by the algorithm \ref{alg:approxmv} converges (upto a subsequence) to an (admissible) measure valued solution $(\nu, \lambda, \nu^{\infty})$ of the incompressible Euler equations \eqref{EQ}.
\end{theorem} 
\begin{proof}
Given the initial bound on the energy and the fact that the energy estimate \eqref{enest} holds for every realization $\omega$, it is straightforward to see that for any $D \subset \bbT^n \times (0,\infty)$, we obtain
$$
\int_D |v_N(\omega)|^2 dx dt \leq C(D), \quad \forall \omega \in \Omega.
$$
Given the fact that $\nu^N$ is the law of the random field $v_N$, the above estimate translates to
$$
\int_D \langle \nu^N_{x,t}, |\xi|^2 \rangle dx dt \leq C(D).
$$
Therefore, by a straightforward modification of the Young measure theorem of \cite{AB1} (see recent paper \cite{FSID3}), we obtain as $N \rightarrow \infty$, that (upto a subsequence), $\nu^N$ converges (narrowly) to a (generalized) Young measure 
$( \nu, \lambda,\nu^\infty)$, such that 
\[\int _{D} \langle \nu^N, g \rangle \phi \dx \dt \to \int _{D} \langle \nu_{x,t}, g\rangle \phi \dx \dt +  \int _{D} \langle \nu^\infty_{x,t}, g^\infty\rangle \phi \, \lambda(\dx \dt),\]
for all $\phi \in C^\infty_0(D)$, 

In particular, we can apply this to the particular choice $g(\xi) = \xi$ with $g^\infty \equiv 0$ and test function $\partial_t\phi$  (for each component) to obtain,
\[ \int _D \partial_t\phi \cdot \langle \nu^N_{x,t}, \xi \rangle \dx \dt \to \int _D \partial_t\phi \cdot  \langle \nu_{x,t}, \xi \rangle \dx \dt.\]
Similarly, with $g(\xi) = \xi \otimes \xi$, $g^\infty(\theta) = \theta \otimes \theta$ and test function $\nabla \phi$, we obtain
\[ \int _D \nabla\phi : \langle \nu^N_{x,t}, \xi \otimes \xi \rangle \dx \dt \to \int _D \nabla\phi :  \langle \nu_{x,t}, \xi\otimes \xi \rangle \dx \dt+  \int _{D} \nabla\phi : \langle \nu^\infty_{x,t}, \theta\otimes \theta \rangle \, \lambda(\dx \dt).\]

Furthermore, the consistency property \eqref{eq:cons1} also holds for every $\omega \in \Omega$, therefore, 
\begin{equation}
\label{eq:cons2}
\lim_{N \to \infty} \int _D \partial_t\phi \cdot v_N(\omega)  + \nabla \phi:v_N (\omega) \otimes v_N (\omega) \dx \dt = 0, \quad \forall \omega \in \Omega.
\end{equation}
In terms of the Young measure $\nu^N$, the above consistency is expressed as,
\begin{equation}
\label{eq:cons3}
\lim_{N \to \infty} \int _D \partial_t\phi \cdot \langle \nu_{x,t}^N,\xi \rangle  + \nabla \phi:\langle \nu^N_{x,t}, \xi \otimes \xi \rangle \dx \dt = 0.
\end{equation}
Thus, we obtain,
\begin{align*}
\int _D &\partial_t\phi \cdot  \langle \nu_{x,t}, g\rangle \dx \dt + \int _D \nabla\phi :  \langle \nu_{x,t}, \xi\otimes \xi \rangle \dx \dt+  \int _{D} \nabla\phi : \langle \nu^\infty_{x,t}, \theta\otimes \theta \rangle \, \lambda(\dx \dt)
\\
&= \lim_{N\to \infty} \int _D \partial_t\phi \cdot \langle \nu_{x,t}^N,\xi \rangle  + \nabla \phi:\langle \nu^N_{x,t}, \xi \otimes \xi \rangle \dx \dt = 0.
\end{align*}
Similarly, we obtain for any $\psi \in C^\infty_0(D)$ that
\[
\int _D \nabla \psi \cdot \langle \nu_{x,t}, \xi \rangle \dx \dt = \lim_{N\to \infty} \int _D \nabla \psi \cdot \langle \nu^N_{x,t}, \xi \rangle \dx \dt = 0,
\]
since $\div(v_N(\omega)) = 0, \; \forall \omega \in \Omega$,
The above proof can be easily adapted to include the initial conditions. Thus, we prove that $\nu^N$ is a measure valued solution of the incompressible Euler equations \eqref{EQ}. Admissibility follows as a straightforward consequence of the energy estimate \eqref{enest}.

\end{proof}
\subsubsection{Approximate measure valued solutions for atomic initial data}
The case of atomic initial data i.e $\sigma = \delta_{v_0}$ with a divergence free velocity field $v_0 \in L^2$ is particularly interesting for applications as it represents the zero uncertainty (in the initial conditions) limit. To compute the measure valued solutions associated with atomic initial data, we use the following algorithm,
\begin{algorithm}
\label{alg1}
Let $(\Omega, \mathcal F, P)$ be a probability space and let $X : \Omega\to L^2(\bbT^n;\bbR^n)$ be a random field satisfying $\|X\|_{L^2(\bbT^n)} \leq 1$ $P$-almost surely.
\begin{description}
\item[\textbf{Step 1:}] Fix a small number $\epsilon >0$.  Perturb $v_0$ by defining $v_0^{\epsilon}(\omega,x) := v_0(x) + \epsilon X(\omega,x)$. Let $\sigma^{\epsilon}$ be the law of $v_0^{\epsilon}$.
\item[\textbf{Step 2:}] For each $\omega\in\Omega$ and $\epsilon> 0$, let $v_N^{\epsilon}(\omega)$ be the solution computed by the spectral method \eqref{spEQ}, corresponding to the initial data $v^{\epsilon}_0$.
\item[\textbf{Step 3:}] Let $\nu^{N,\epsilon}$ be the law of $v_N^{\epsilon}$.
\qed\qedhere
\end{description}
\end{algorithm}

\begin{theorem}
\label{thm:alphaconv}
Let $\{\nu^{N,\epsilon}\}$ be the family of approximate measure valued solutions constructed by algorithm \ref{alg1}. Then there exists a subsequence $(N_n,\epsilon_n) \to 0$  such that
\[
\nu^{N_n,\epsilon_n}\rightarrow (\nu,\lambda,\nu^{\infty}),
\]
with $(\nu, \lambda, \nu^{\infty})$ being an admissible measure valued solution of the incompressible Euler equations \eqref{EQ} with atomic initial data $v_0$.
\end{theorem}
The proof of this theorem is a straightforward extension of the proof of theorem \ref{thm:conv} and we omit it here.
\begin{remark}
There is an analogy between the zero viscosity limit and the zero uncertainty limit considered above. It is commonly argued that in real-world systems, viscosity effects are unavoidable. In order to obtain the correct solution in e.g. the context of conservation laws, a small amount of viscosity should therefore be added to the equations. In situations where viscosity effects are assumed to play a secondary role, the zero viscosity viscosity limit must then be considered and will lead to the correct physical solution.

Along the same lines it can be argued that in real-world systems, uncertainties in the initial data, arising e.g. from uncertainties in measurements, are unavoidable. To account for this fact a small amount of uncertainty should be introduced. If the uncertainties are assumed to be negligible, the correct solution should correspondingly be obtained in the zero uncertainty limit as described in algorithm \ref{alg1} and theorem \ref{thm:alphaconv}.
\end{remark}
\subsection{Computation of space-time averages}
The algorithms \ref{alg:approxmv} and \ref{alg1} compute space-time averages with respect to the measure $\nu^N$,
\begin{equation}
\label{eq:av}
\int_{\bbT^n} \int_{\bbR_+} \varphi(x,t) \langle \nu_{x,t}^N,g\rangle dx dt,
\end{equation}
for smooth test functions $\varphi$ and for admissible functions $g$, i.e. $g\in C^\infty(\bbR^n)$ for which $g^\infty(\theta) = \lim_{r\to \infty} g(r\theta)/r^2$ exists and  $g^\infty \in C(\bbS^{n-1})$ is continuous. 

Following \cite{FKMT1}, we will compute space-time averages \eqref{eq:av} by using a Monte-Carlo sampling procedure
To this end, we utilize the equivalent representation of the measure $\nu^{N}$ as the law of the \emph{random field} $v_N$:
\begin{equation}\label{eq:func2}
  \langle \nu_{x,t}^N,g\rangle := \int\limits_{\bbR^n} g(\xi)\ d\nu^{N}_{(x,t)}(\xi) = \int_{\Omega} g(v_N(\omega;x,t))\ dP(\omega).
\end{equation}
We will approximate this integral by a Monte Carlo sampling procedure:
\begin{algorithm}\label{alg:montecarlo}
Let $N > 0$ and let $M$ be a positive integer. Let $\sigma$ be the initial Young measure and let $v_0$ be a (spatially divergence free) random field $v_0:\Omega\times\bbT^n \to \bbR^n$ such that $\sigma$ is the law of $v_0$.
\begin{description}
\item[{\bf Step 1:}] Draw $M$ independent and identically distributed random fields $v_0^k$ for $k=1,\dots, M$.
\item[{\bf Step 2:}] For each $k$ and \emph{for a fixed} ${\omega}\in\Omega$, use the spectral method \eqref{spEQ} to numerically approximate the incompressible Euler equations with initial data $v_0^{k}({\omega})$. Denote $v_{N,k}({\omega})$ as the computed solution.
\item[{\bf Step 3:}] Define the approximate measure-valued solution
\[
\nu^{N,M} := \frac{1}{M}\sum_{k=1}^M \delta_{v_{N,k}({\omega})}.
\]
\qed\qedhere
\end{description}
\end{algorithm}
For every admissible test function $g$, the space-time average \eqref{eq:av} is then approximated by
\begin{equation}\label{eq:mc1}
\int_{\bbT^n} \int_{\bbR_+} \varphi(x,t) \langle \nu_{x,t}^N,g\rangle dx dt \approx \frac{1}{M} \sum_{k=1}^{M} \int_{\bbR_+}\int_{\bbT^n} \varphi(x,t) g\bigl(v_{N,k}(\omega;x,t)\bigr)\ dxdt.
\end{equation}
The convergence of the approximate Young measures $\nu^{N,M}$ to a measure valued solution of the incompressible Euler equations \eqref{EQ} as $N,M \rightarrow \infty$ follows as a consequence of the law of large numbers. The proof is very similar to that of theorem 4.9 of \cite{FKMT1}.
\subsection{The spectral method in two space dimensions}
The spectral (viscosity) method \eqref{spEQ} is considerably simplified when we consider the equations \eqref{EQ} in two space dimensions.  
We recall that the divergence-free condition $\div(v) = 0$ is equivalent to the requirement that $\widehat v_k \perp k$ for all Fourier coefficients $\widehat v_k$ of $v$. In two spatial dimensions, this implies that we can write $\widehat v_k = a_k \, \BJ k$ for scalar coefficients $a_k$ and where $\BJ$ denotes the rotation matrix $$\BJ = \begin{pmatrix} 0 & -1 \\ 1 & 0 \end{pmatrix}.$$

The corresponding evolution equations for the coefficients $a_k$ are found by taking the inner product of $\frac{d}{dt} \widehat v(k)$ with $\BJ\,k/k^2$. This yields the following equivalent form of \eqref{EQF}:
\begin{equation}
\label{spODE}
\frac{d a_k }{dt} =   \frac{(-i)}{k^2} \sum_{\substack{\ell, m \\ \ell + m -k = 0} }  (\BJ k\cdot \ell)(k\cdot m) \, a_\ell a_m .
\end{equation}
It is now natural to define a (real-valued) function $\psi$ by $$\psi(x,t) = (-i)\sum_k a_k(t) e^{ikx}.$$ This function is usually referred to as the stream function. 

If $\nabla^\top$ denotes the operator $(-\partial_{x_2},\partial_{x_1})^T = \BJ \nabla$ acting on functions, then $v$ is given by $v = \nabla^\top \psi$. Thus, $\psi$ determines $v$ uniquely. On the other hand, given $v$, we can recover $\psi$ by solving
 $\Delta \psi= \curl v$. This equation has a unique solution for sufficiently smooth $v$ if we require in addition that $\int \psi \dx = 0$. Solving the equations for the scalar $a_k$ are an attractive form of the spectral method \eqref{spEQ}.

The introduction of the stream function also provides a connection to the vorticity formulation of the incompressible Euler equations. In two space dimensions, the vorticity equation \eqref{eq:vor} is simplified to,
 \begin{gather}\label{vortEQ}
 \left\{
 \begin{aligned}
 \partial_t \eta+ v \cdot \nabla \eta &= 0, \\
  \curl v &= \eta.
  \end{aligned}
  \right.
  \end{gather}
 Corresponding to \eqref{spEQ}, we also obtain the semi-discretized version for to the vorticity formulation.
  \begin{gather}\label{vortspEQ}
\left\{
	\begin{aligned} 
		\partial_t \eta_N + \BP_N\left(v_N\cdot \nabla \eta_N\right) &= \varepsilon  \div\left( (\BI - \BP_m) \nabla \eta_N \right), \\
		\eta_N(x,0) = \curl \BP_Nv_0(x).
	\end{aligned}
	\right.
\end{gather}
The system of equations \eqref{vortEQ} is formally equivalent to \eqref{EQ}.\footnote{The two equations are strictly equivalent only if the flow is sufficiently smooth.} The important observation for us is the following: Even though the two full systems of equations might not be equivalent, their Fourier truncated versions are \emph{always} equivalent. 

\begin{lemma}
\label{equivalence}
The truncated systems with spectrally small vanishing viscosity \eqref{spEQ} for $v_N$ and \eqref{vortspEQ} for $\eta_N$ are equivalent.
\end{lemma}

\begin{proof}
Let $v_N$ and $\eta_N$ be solutions of \eqref{spEQ} and \eqref{vortspEQ}, respectively. Since $v_N$ is smooth for fixed $N$ and because the projection operators $\BP_N$ commute with differentiation, we can take the curl of \eqref{spEQ} to obtain
\begin{gather}\label{spectralNS}
\left\{
	\begin{aligned} 
		\partial_t \curl v_N + \BP_N\left(\curl(v_N\cdot \nabla v_N)\right) &=  \varepsilon  \div\left( (\BI - \BP_m) \nabla  \curl v_N\right), \\
		\curl v_N(x,0) = \curl \BP_N v_0(x).
	\end{aligned}
	\right.
\end{gather}
We note that $\curl(v_N \cdot \nabla v_N) = v_N \cdot \nabla \curl v_N$. Hence, both $\eta_N$ and $\curl v_N$ satisfy system \eqref{vortspEQ}. By classical uniqueness results for ODEs, this implies that we must have $\eta_N = \curl v_N$ and the two systems are seen to be equivalent.
\end{proof}

In particular, by Lemma \ref{equivalence} we may use the apparently simpler system \eqref{vortspEQ} for our numerical computations, rather than the larger system \eqref{spEQ}. This reduces the computational cost considerably.
\subsubsection{Time stepping}
The spectral (viscosity) method in the velocity formulation \eqref{spEQ} or the equivalent vorticity formulation \eqref{vortspEQ} constitute a system of ODEs (for the corresponding Fourier coefficients) at each time step. These ODEs are integrated in time by using a third-order strong stability preserving Runge-Kutta (SSP-RK3) method of \cite{GST1}.
\section{Numerical experiments}
\label{sec:numex}
In this section, we will provide numerical experiments that demonstrate the theory developed in the previous section (particularly the convergence of algorithm \ref{alg1}).
\subsection{Rotating vortex patch}
\begin{figure}
	\begin{subfigure}{0.49\textwidth}
\includegraphics[width=\textwidth]{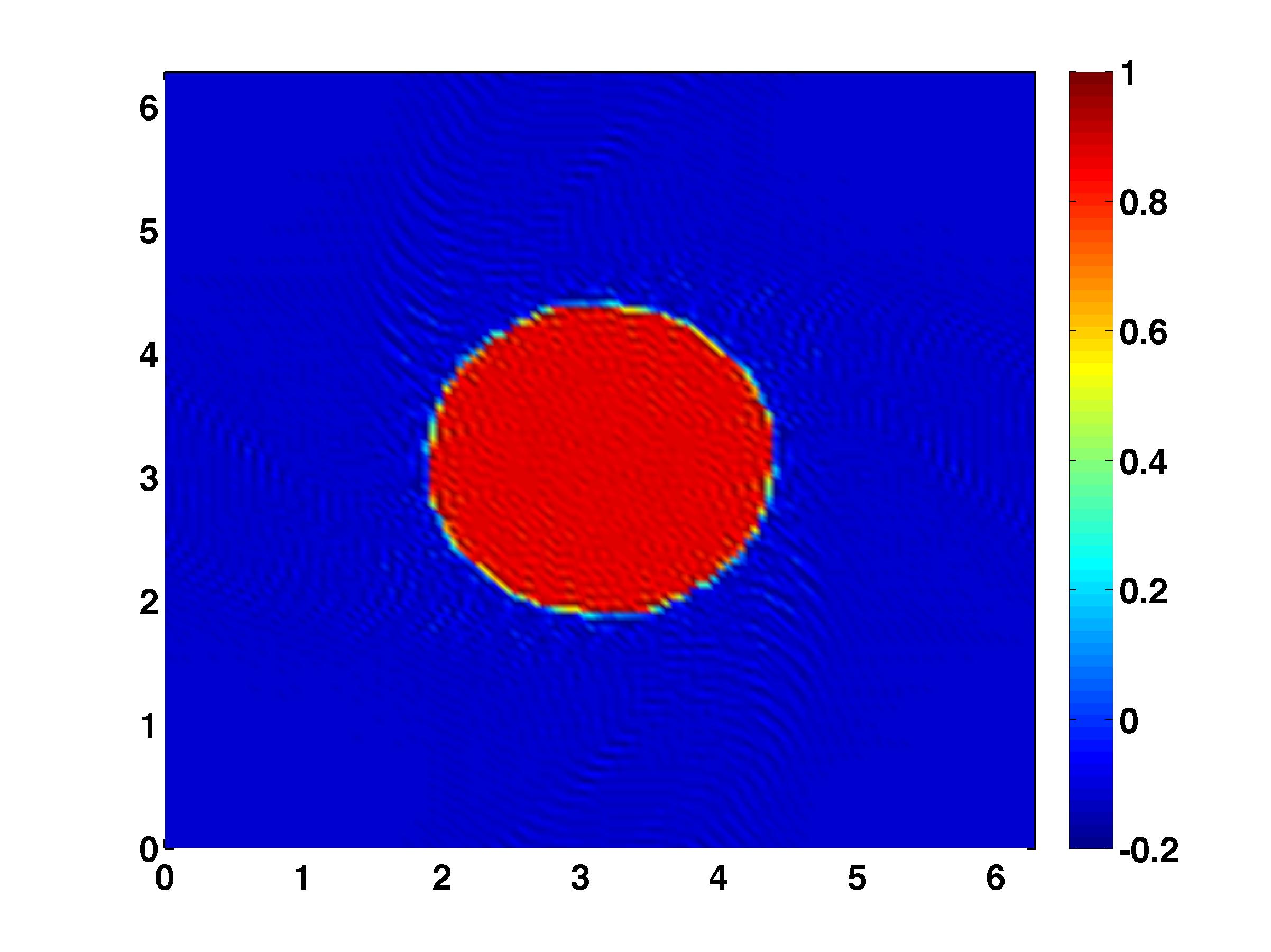}
\caption{$N = 128$}
\end{subfigure}
\begin{subfigure}{0.49\textwidth}
\includegraphics[width=\textwidth]{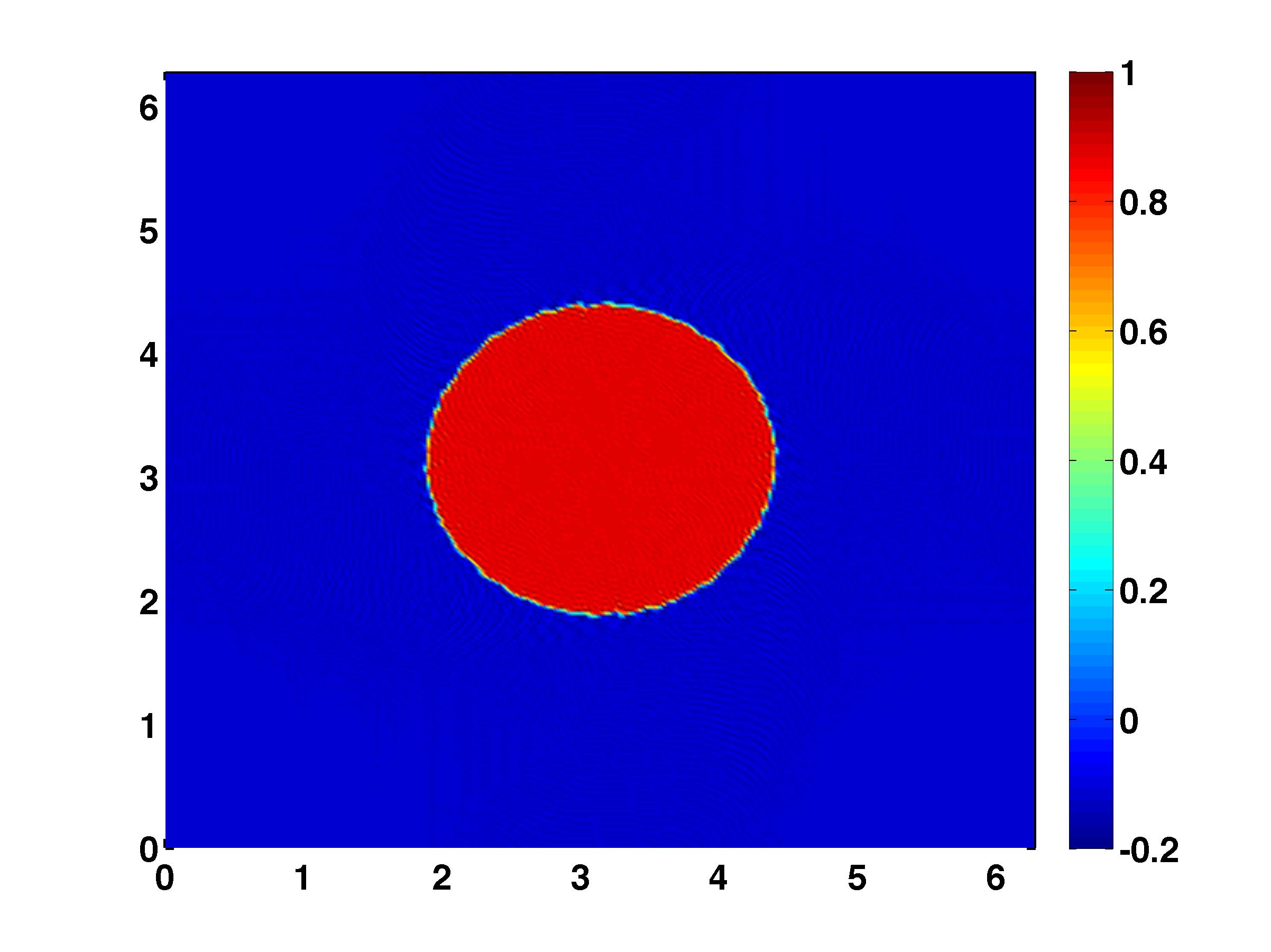}
\caption{$N = 256$}
\end{subfigure} \\
\begin{subfigure}{0.49\textwidth}
\includegraphics[width=\textwidth]{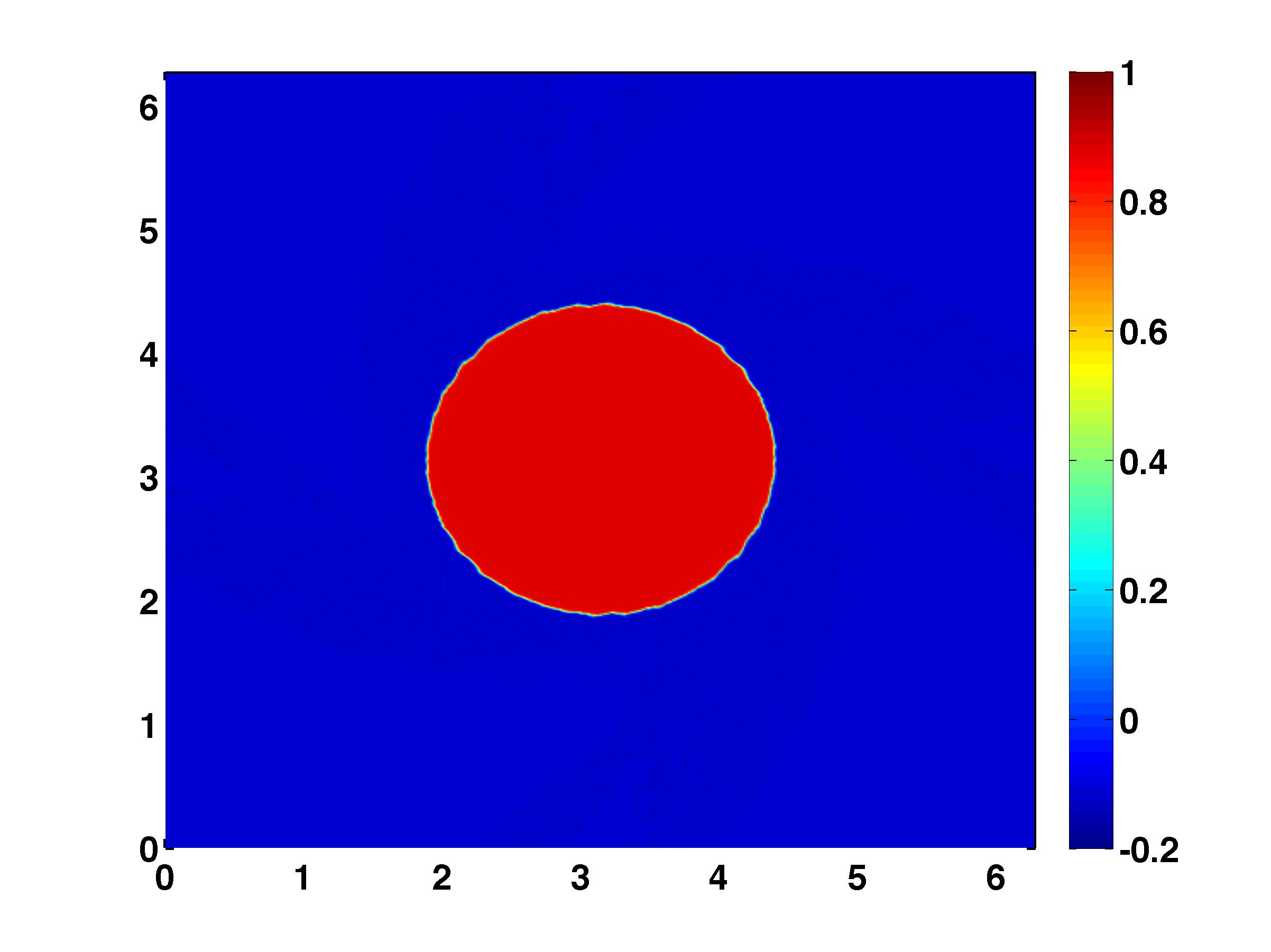}
\caption{$N = 512$}
\end{subfigure}
\begin{subfigure}{0.49\textwidth}
\includegraphics[width=\textwidth]{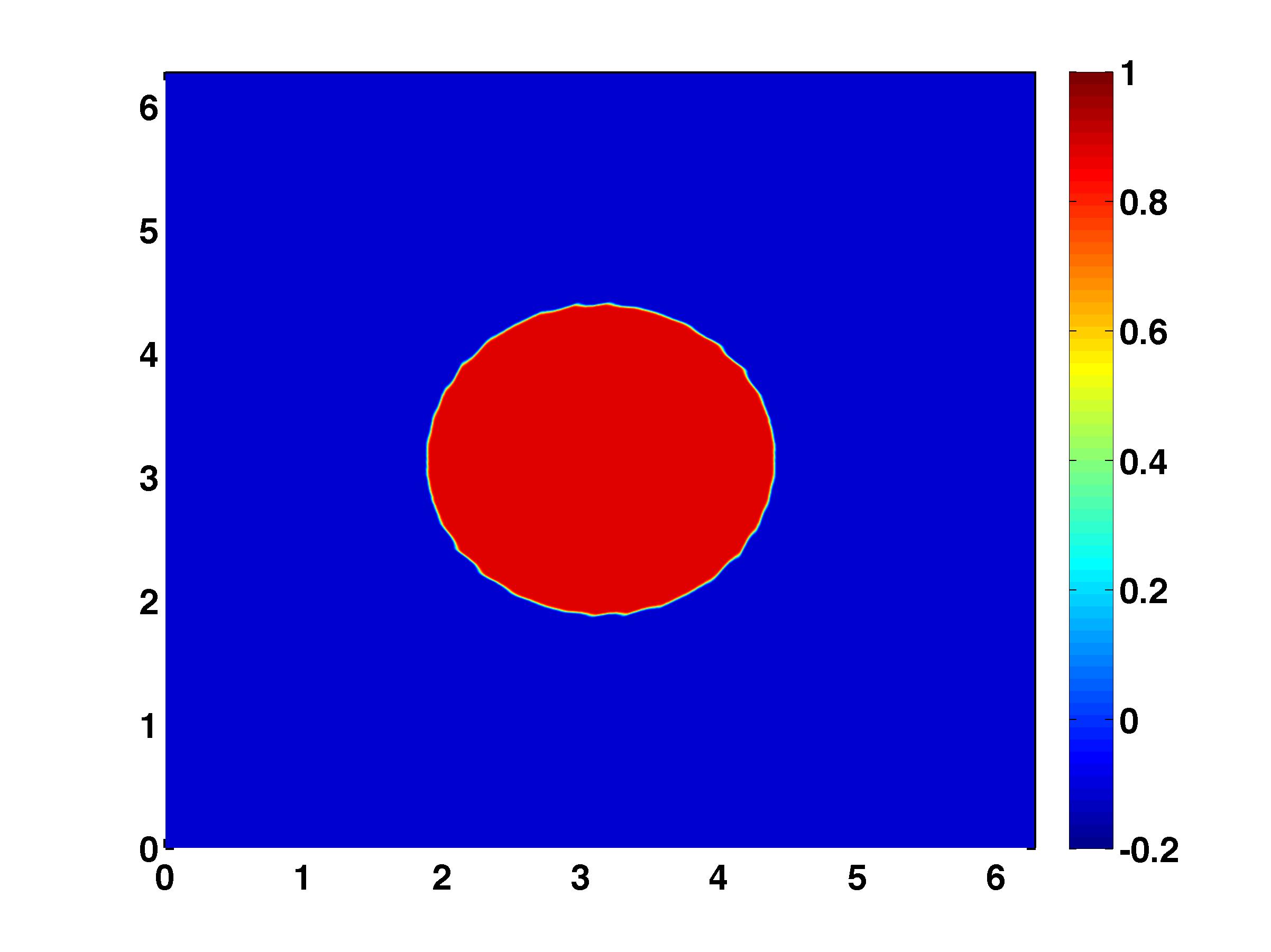}
\caption{$N = 1024$}
\end{subfigure}
\caption{Rotating vortex patch: sample convergence of the vorticity wrt number of Fourier modes $N$: Top left $N= 128$, Top right: $N=256$. Bottom left $N=512$. Bottom right: $N=1024$.}
	\label{fig:1}
\end{figure}
\begin{figure}
	\begin{subfigure}{.45\textwidth}
\includegraphics[width=\textwidth]{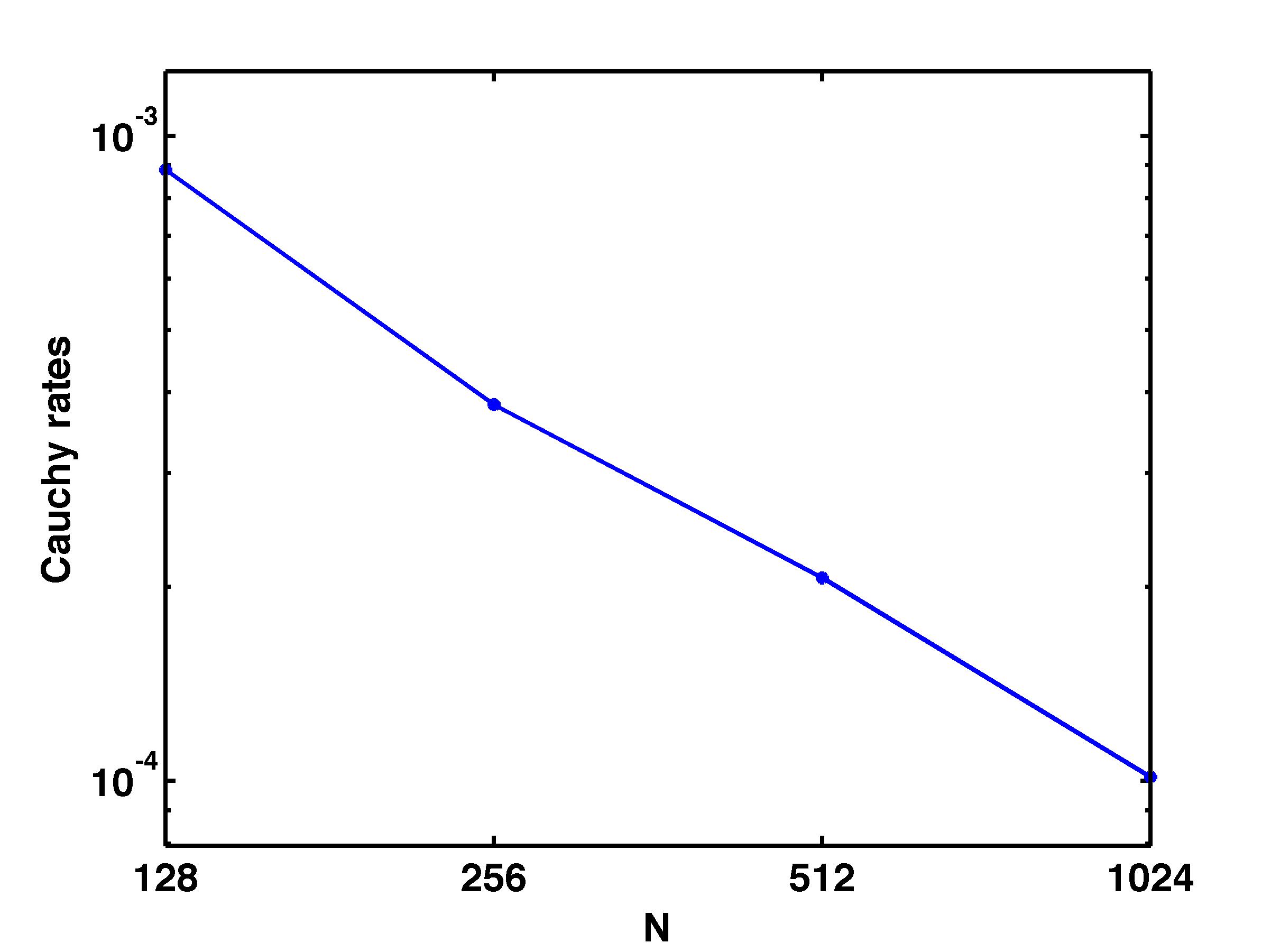}
\caption{wrt $N$}
\end{subfigure}
\begin{subfigure}{.45\textwidth}
\includegraphics[width=\textwidth]{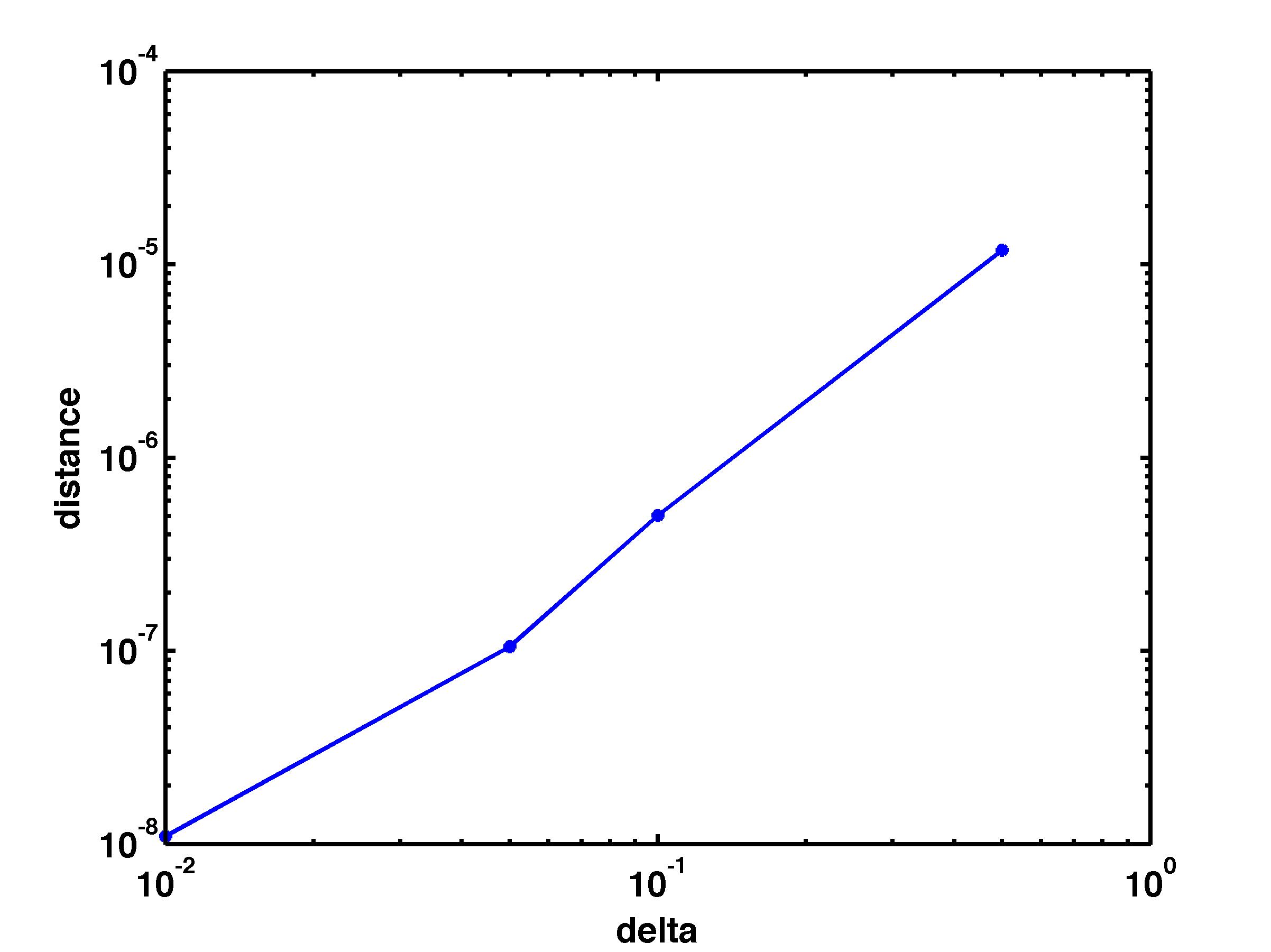}
\caption{wrt $\delta$}
\end{subfigure}
	
	\caption{Rotating vortex patch: Left: Cauchy rates \eqref{eq:diff} wrt $N$. Right: Cauchy rates wrt $\delta$}
	\label{fig:2}
\end{figure}

A rotating vortex patch can be simulated by considering the two-dimensional Euler equations with initial data for the vorticity as,
 \begin{equation}
 \label{eq:vp1}
 \eta_0(x) = \begin{cases} 1, & (x_1-\pi)^2+(x_2-\pi)^2 \le \pi/2, \\ 0, & \text{otherwise}. 
 \end{cases}
 \end{equation}

As our objective is to test the algorithm \ref{alg1}, we consider a perturbed version of the rotating vortex patch (see step $1$ of algorithm \ref{alg1}). In radial coordinates about the center $(\pi,\pi)$, we define a random perturbation 
$$p_\delta(\theta) = 1 + \sum_{k=1}^K a_k \sin(b_k + (20+k)\theta),$$
with $a_1, \dots, a_K \in [0,1],\, b_1, \dots, b_K \in [0,2\pi]$ independent and identically distributed random variables, chosen according to a uniform distribution with renormalization $\sum_{k=1}^K |a_k|^2 = \delta$. We set $K=20$ in our computations.The perturbed initial data, depending on the perturbation parameter $\delta > 0$, is given in terms  vorticity,
\begin{equation}
\label{eq:vp} 
\eta_0^\delta(r,\theta) = \eta_0(r- p_\delta(\theta), \theta).
\end{equation}
The corresponding velocity field $v_0^{\delta}$ is obtained from the Biot-Savart law \cite{bertozzi}.

First, we fix a realization of the random field $v_0^{\delta}(\omega)$ by setting $\delta = 0.0128$. This initial data is evolved using the spectral (viscosity) method with $\epsilon = 10^{-5}$, $m = 0$. The results are then presented in figures \ref{fig:1} and \ref{fig:2}. In figure \ref{fig:1}, we present the vorticity as the number of modes $N$ is increased from $128$ to $1024$. We see that the vortex patch is well resolved with increasing resolution. 

Next, we compute the differences between successive resolutions,
\begin{equation}
\label{eq:diff}
\|v^{\delta}_N(t) - v^{\delta}_{N/2}(t) \|_{L^2}^2.
\end{equation}
This difference at time level $T= 2$, shown in figure \ref{fig:2} (left), clearly converges as $N \rightarrow \infty$. Consequently, the sequence of approximations (for a single realization) forms a Cauchy sequence and hence converges.

Since, the algorithm \ref{alg1} is based on setting the perturbation amplitude $\delta \rightarrow 0$, we fix the number of approximating Fourier modes $N = 512$ and decrease $\delta$. The corresponding difference between two successive values of $\delta$ is shown in figure \ref{fig:2} (right) and shows that the approximations clearly converge as the perturbation amplitude is reduced. Thus, for each fixed realization (sample), we already observe convergence of the spectral method as well as stability of the computed solutions with respect to perturbations in initial data. Although the initial data is not smooth (the vorticity is discontinuous), this convergence and stability are not surprising as the solution does not possess any fine scale features. Consequently, the computed measure valued solution $\nu$ is \emph{atomic} as shown in figure \ref{fig:3}.
\begin{figure}
	\centering
	\includegraphics[width=\textwidth]{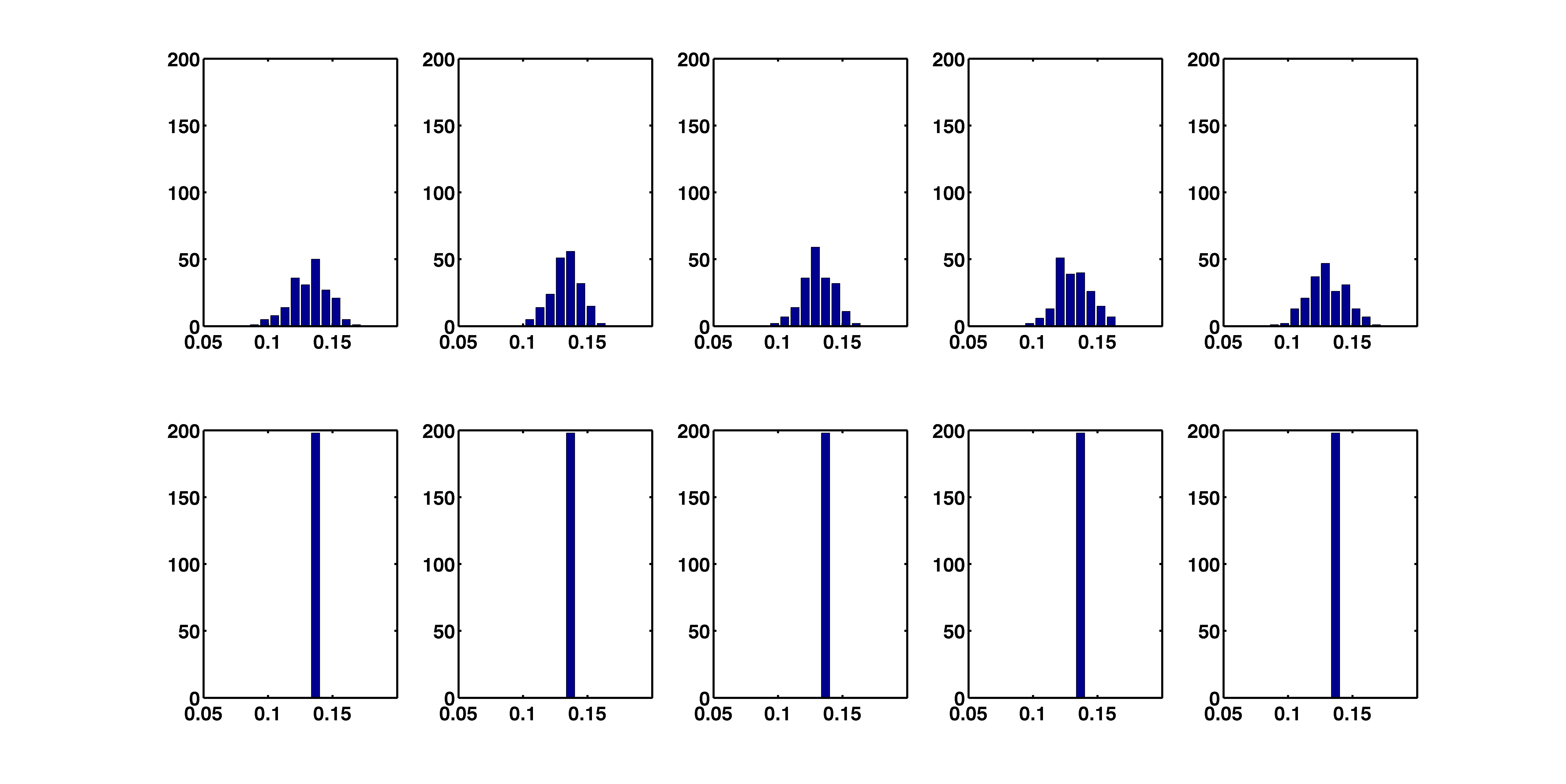}
	\caption{Rotating vortex patch: PDFat $x =  2\pi\cdot(0.65,0.55)$, with respect to time $t=0,0.5,1,2,4$ at two different delta values $\delta = 0.512$ (top), illustrating atomicity}
	\label{fig:3}
\end{figure}

\subsection{Flat vortex sheet}
\label{subsec:fvs}
Next, we consider a flat vortex sheet as a prototype for two-dimensional Euler flows with singular behavior. The underlying initial data is
\begin{equation} 
\label{eq:vs}
v^0(x) = \begin{cases} (-1,0), & \text{if } \; \pi/2 < x_2 \le 3\pi/2, \\ (1,0), & \text{if} \; x_2 \le \pi/2 \; \text{ or } \; 3\pi/2 < x_2,
\end{cases}
\end{equation}
on a periodic domain $[0,2\pi]^2$. The initial vorticity in this case is a bounded measure. 

It is straightforward to check that the initial data for the flat vortex sheet \eqref{initdata} is a stationary (steady state) weak solution of the two-dimensional Euler equations. However, this datum also belongs to the class of \emph{wild initial data} in the sense of Szekelyhidi \cite{Sz1}. Thus, infinitely many admissible weak solutions were constructed in \cite{Sz1}. 

Our objective is to compute the (admissible) measure valued solution, corresponding to this atomic initial data, by applying the algorithm \ref{alg1}. To this end, we mollify the initial data $v^0$ to obtain a smooth approximation $v^0_{\rho} = \left(\pi_1v^{0}_{\rho},\pi_2v^{0}_{\rho}\right)$ of \eqref{initdata},
 \begin{gather*}
\pi_1v^{0}_{\rho}(x_1,x_2) =
\left\{
\begin{aligned}
\tanh\left(\dfrac{x_2-\pi/2}\rho\right), \quad & (x_2 \le \pi) \\ 
\tanh\left(\dfrac{3\pi/2-x_2}\rho\right), \quad & (x_2 > \pi)
 \end{aligned}\; \right\} , \quad 
 \pi_2v^{0}_{\rho}(x_1,x_2) = 0.
 \end{gather*}
 with a small parameter $\rho$ that determines the sharpness of the transition between $-1$ and $1$ across the interfaces.

 To obtain a random field (as required by Step 1 of algorithm \ref{alg1}, we further introduce perturbations of the two interfaces by the following perturbation ansatz,
  $$p_\delta(x) = \sum_{k=1}^{K} \alpha_k \sin(kx_1-\beta_k),$$
for randomly chosen numbers $\alpha_1, \dots, \alpha_K \in \bbR$, $\beta_1, \dots, \beta_K\in [0,2\pi)$ with $\sum_{k=1}^K |\alpha_k|^2 = \delta$. For our computations, we used a fixed value of $K = 10$ perturbation modes. 
\begin{figure}
	\centering
	\includegraphics[width=.5\textwidth]{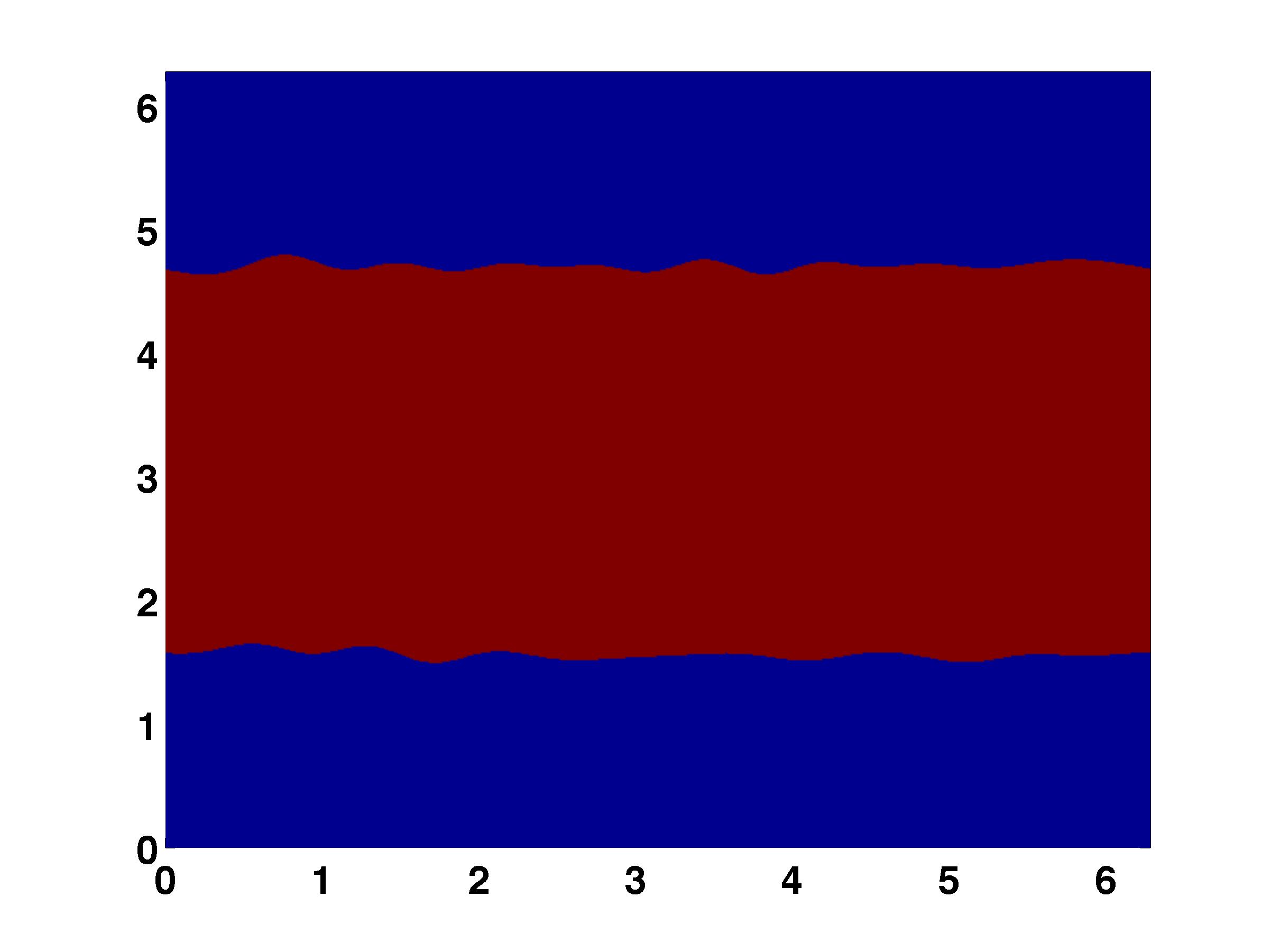}
	\caption{Initial tracer distribution corresponding to the perturbed flat vortex sheet with $\delta = 0.0512$.}
	\label{initdata}
\end{figure}

The result of this ansatz is a random field $v^0_{\rho}(x_1,x_2-p_\delta(x_1))$ depending on two parameters $\rho$ and $\delta$. The parameter $\delta$ controls the magnitude of the permutation, while $\rho$ determines the smoothness across the interfaces. Projecting  this random field back to the space of divergence-free vector fields (using the Leray projection), we obtain our initial random perturbation $X^0_{\rho,\delta}$ which serves as our initial data. In order to visualize this initial velocity field, we consider a passively advected tracer and plot it in figure \ref{initdata}. 

For a fixed number of Fourier modes $N$, we aim to compute the corresponding approximate Young measure $\nu^{\rho,\delta}_N$ (Step 2 of algorithm \ref{alg1}). Then, the measured valued solution of \eqref{EQ} will be realized as a limit of $\nu^{\rho,\delta}_N$ as $N \rightarrow \infty, \rho,\delta \rightarrow 0$. 

First, we fix a single realization of the random field $X^0_{\rho,\delta}(\omega)$ by fixing $\omega$. To visualize the resulting approximate solutions, we show a passive tracer (advected by the velocity field) in figure \ref{fig:5}, at time $t=2$  and with $(\delta,\rho) = (0.01,0.001)$, at different Fourier modes $N$ ranging from $N=128$ to $N=1024$. We see from the figure that as the resolution is refined, finer and finer scale features emerge, indicating that the tracer is getting mixed by the fluid at smaller and smaller scales. Furthermore, this indicates that the underlying velocity field may not converge as the number of Fourier modes is increased. This is indeed verified in figure \ref{fig:5} (left), where we show the successive differences \eqref{eq:diff} of the approximate solution in $L^2$ (for a single sample). The differences do not seem to converge, indicating the approximate solutions may not form a Cauchy sequence. Hence and in contrast with the vortex patch example, the approximate solutions for a single realization (sample) may not converge, at least for the computed resolutions. 
\begin{figure}
	\centering
	\includegraphics[width=\textwidth]{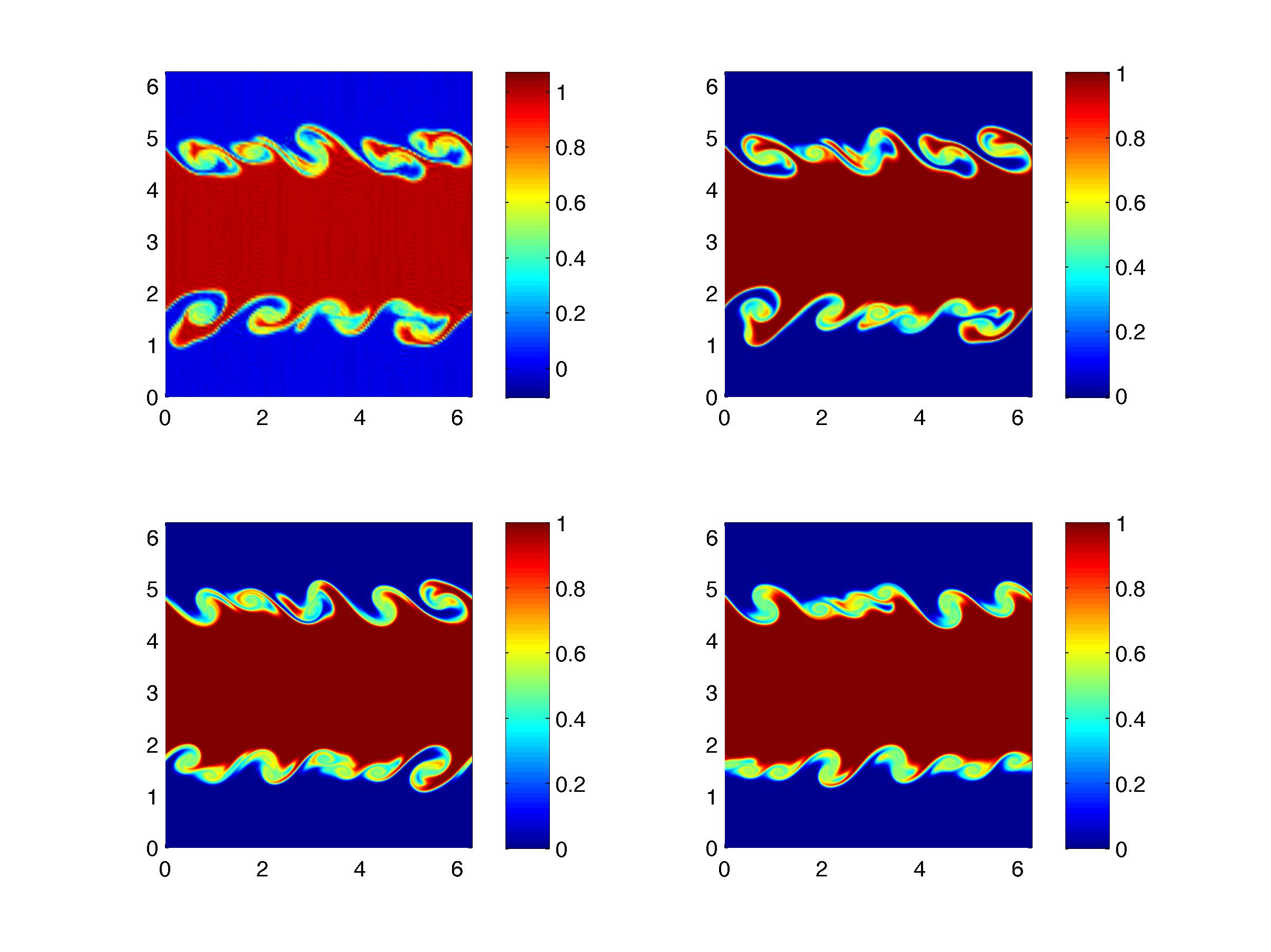}
		\caption{Flat vortex sheet (sample): Non-convergence with respect to $N$. A passive tracer (advected by the velocity field) is shown at time $t=2$ for different Fourier modes. Top left: $N=128$, Top right $N=256$, Bottom left $N=512$ and Bottom right $N=1024$.}
	\label{fig:5}
\end{figure}

\begin{figure}
	\begin{subfigure}{.45\textwidth}
\includegraphics[width=\textwidth]{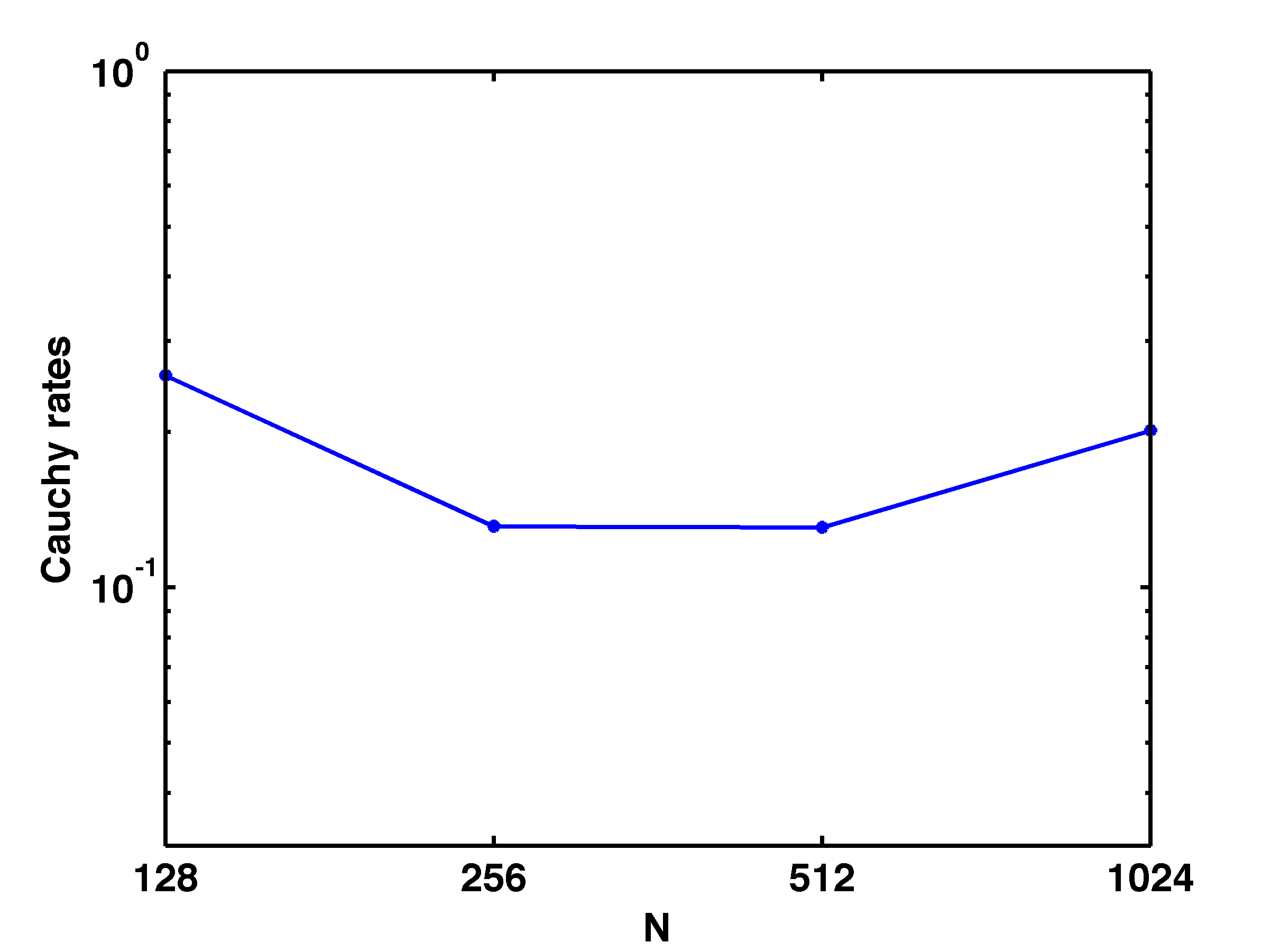}
\caption{wrt $N$}
\end{subfigure}
\begin{subfigure}{.45\textwidth}
\includegraphics[width=\textwidth]{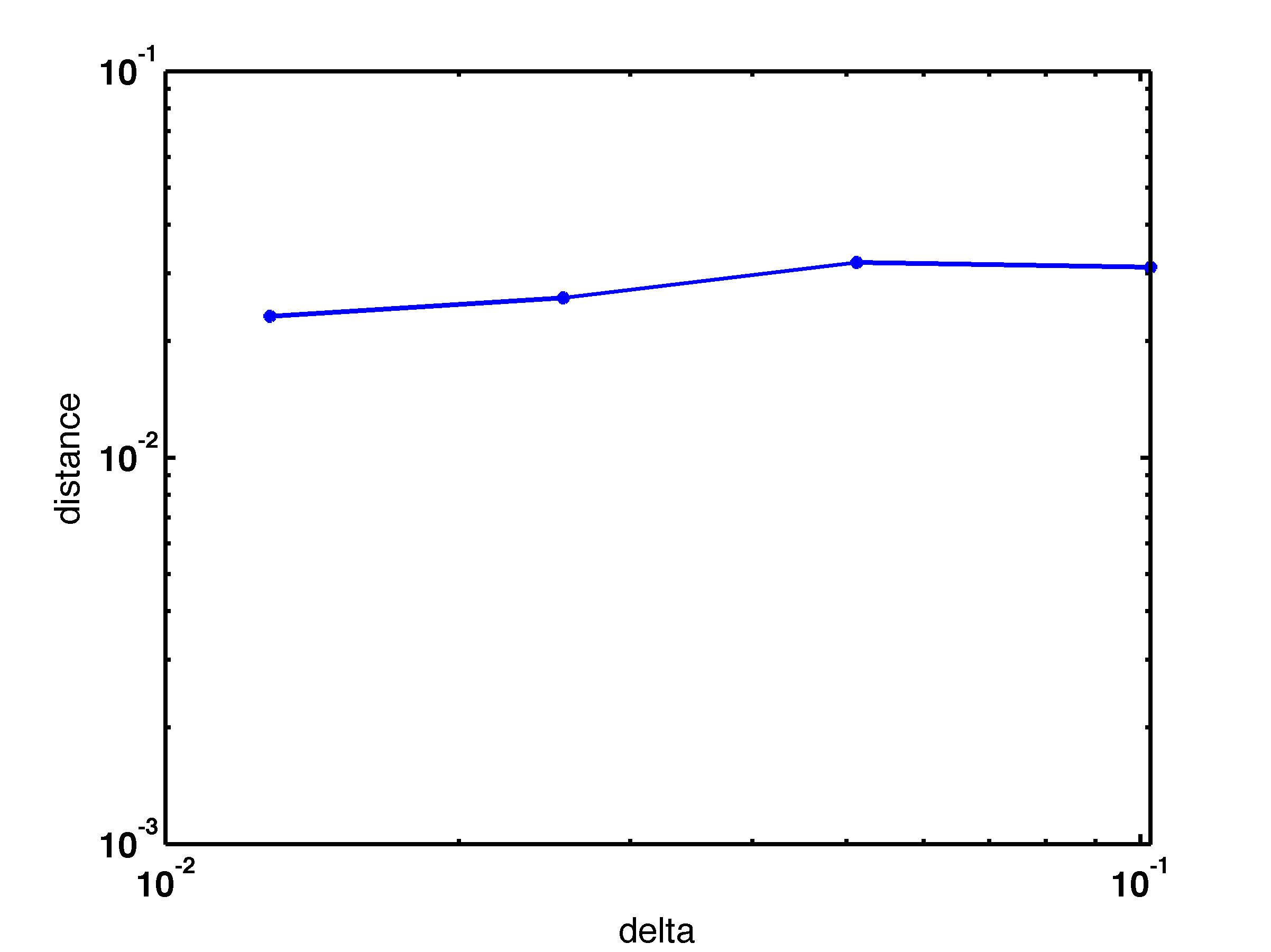}
\caption{wrt $\delta$}
\end{subfigure}
	\caption{Flat vortex sheet (sample): (left) Cauchy rates wrt $N$, (right) Cauchy rates wrt $\delta$}
	\label{fig:6}
\end{figure}

Next, we consider the stability of the approximate solutions (for a single realization) with respect to the perturbation parameter $\delta$. For a fixed $N=512$ and time $t=2$, we show a passively advected tracer, for different values of $\delta$ in figure \ref{fig:7}. Again, the fine scale structure of the solutions is very different for each value of $\delta$. As shown in figure \ref{fig:6} (right), the difference (in $L^2$) for successive values of $\delta$ does not decrease as $\delta$ decreases. This indicating that the perturbed solutions do not converge as the perturbation tends to zero, indicating instability of the flat vortex sheet \eqref{eq:vs} with respect to perturbations. 

\begin{figure}
	\centering
		\includegraphics[width=\textwidth]{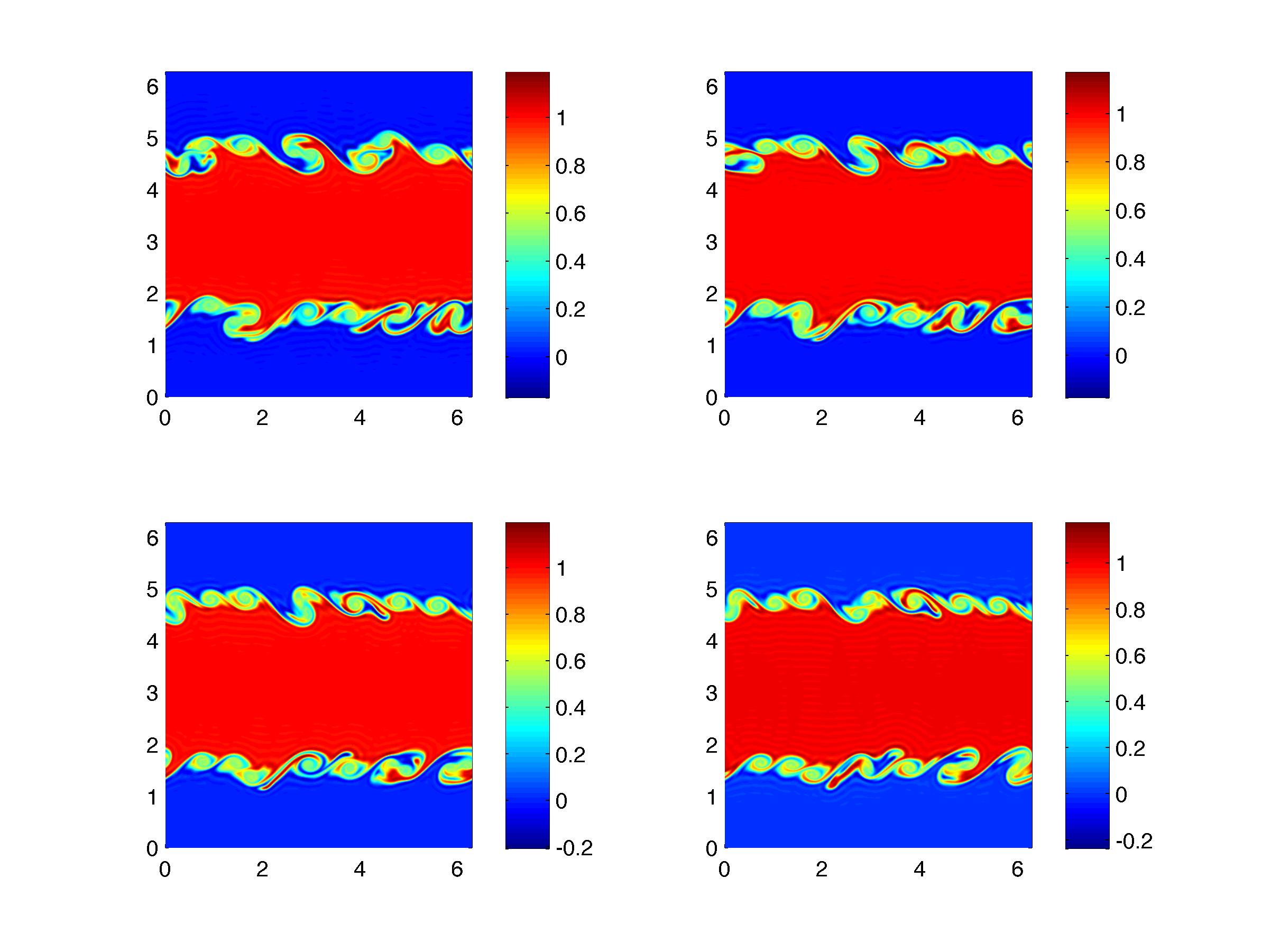}
	\caption{Flat vortex sheet: Instability with respect to perturbation parameter $\delta$. Passively advected tracer at $t=2$ for Top left: $\delta = 0.0512$, Top right: $\delta = 0.0256$, Bottom left $\delta = 0.0128$, Bottom right $\delta = 0.0064$.}
	\label{fig:7}
\end{figure}


Having observed the lack of convergence (and stability) for single realizations of the perturbed vortex sheet, we apply the algorithm \ref{alg1} to compute the approximate Young measure. To this end, we use the Monte Carlo algorithm \ref{alg:montecarlo} with $M=400$ samples. We compute the mean of the approximate Young measure by setting $g(\xi) = \xi$ in \eqref{eq:func2}. Similarly, the second moments are computed by setting $g(\xi) = \xi \otimes \xi$ in \eqref{eq:func2}. The mean of the first component and second moment of the second component ($g(\xi) = \xi^2_2$) at time $t=2$, for different number of Fourier modes are shown in figures \ref{fig:7a} and \ref{fig:7b}, respectively. In complete contrast to figure \ref{fig:5} (single sample) and as predicted by Theorem \ref{thm:alphaconv}, both the mean as well as the second moment seem to converge as the number of Fourier modes is increased. This convergence is further verified in figure \ref{fig:7c} (left), where successive $L^2$ differences of the mean velocity field and the second $\xi_2\xi_2$ moment are displayed. The convergence in the second moment is slower than than that of the mean. This is not unexpected as we use the same number of samples for the computation of both the mean and the second-moment. Furthermore, from figure \ref{fig:7a} and in sharp contrast with the case of single realizations, we observe that small scale features are averaged out in the statistical quantities such as the mean and the second moment.
\begin{figure}
	\centering
	\begin{subfigure}{.49\textwidth}
	\includegraphics[width=\textwidth]{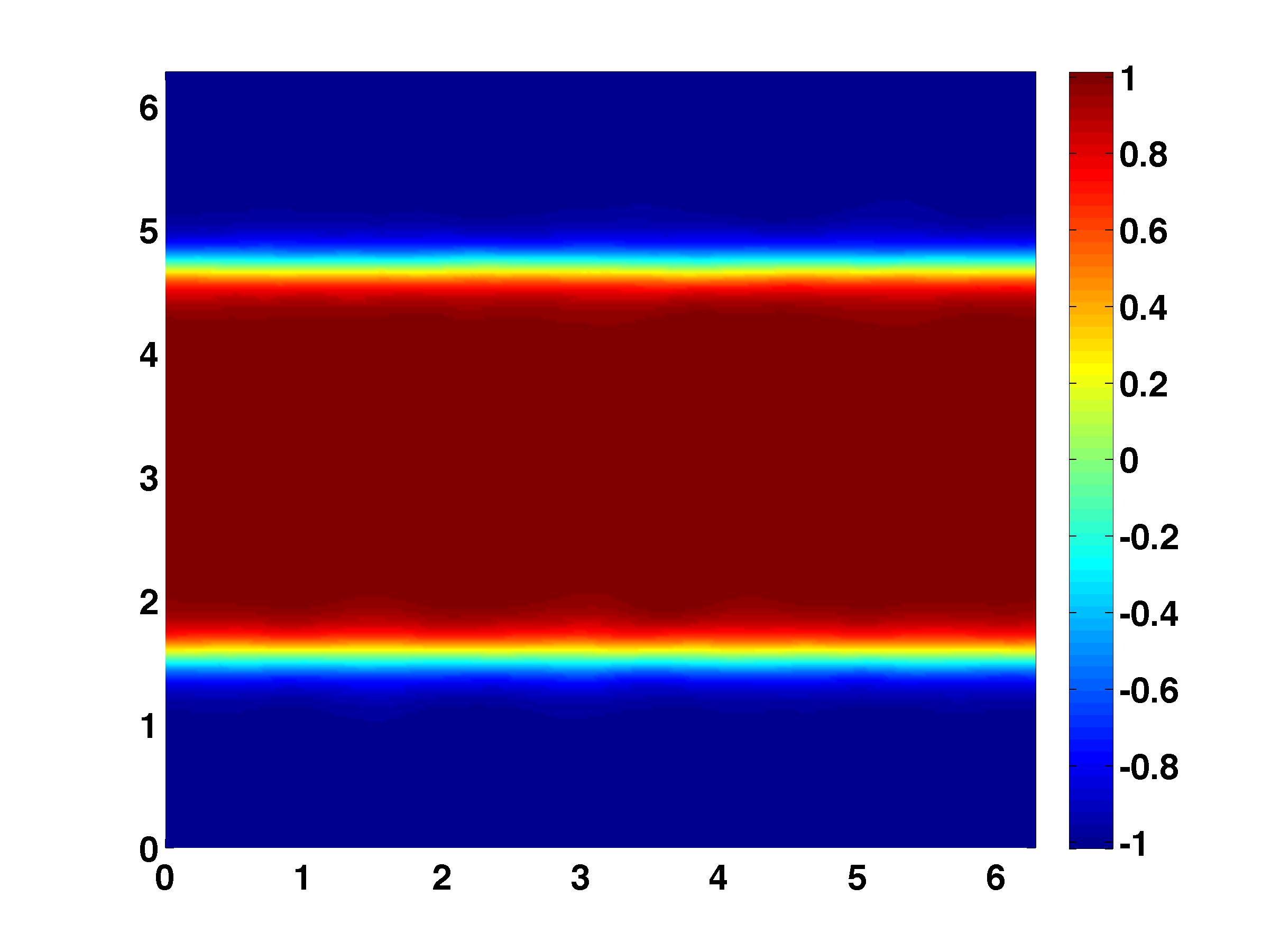}
	\end{subfigure}
	\begin{subfigure}{.49\textwidth}
	\includegraphics[width=\textwidth]{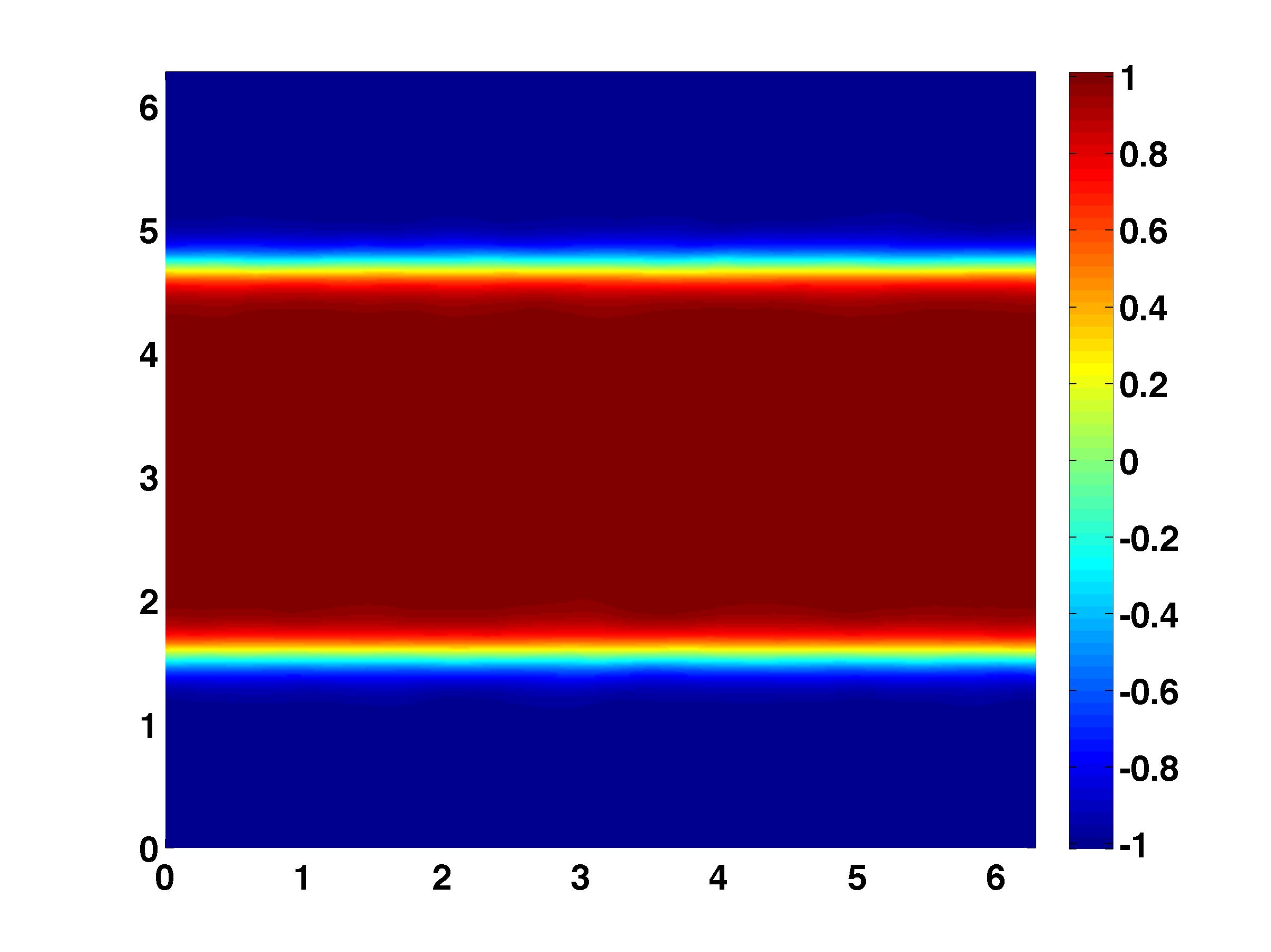}
	\end{subfigure}\\
	\begin{subfigure}{.49\textwidth}
	\includegraphics[width=\textwidth]{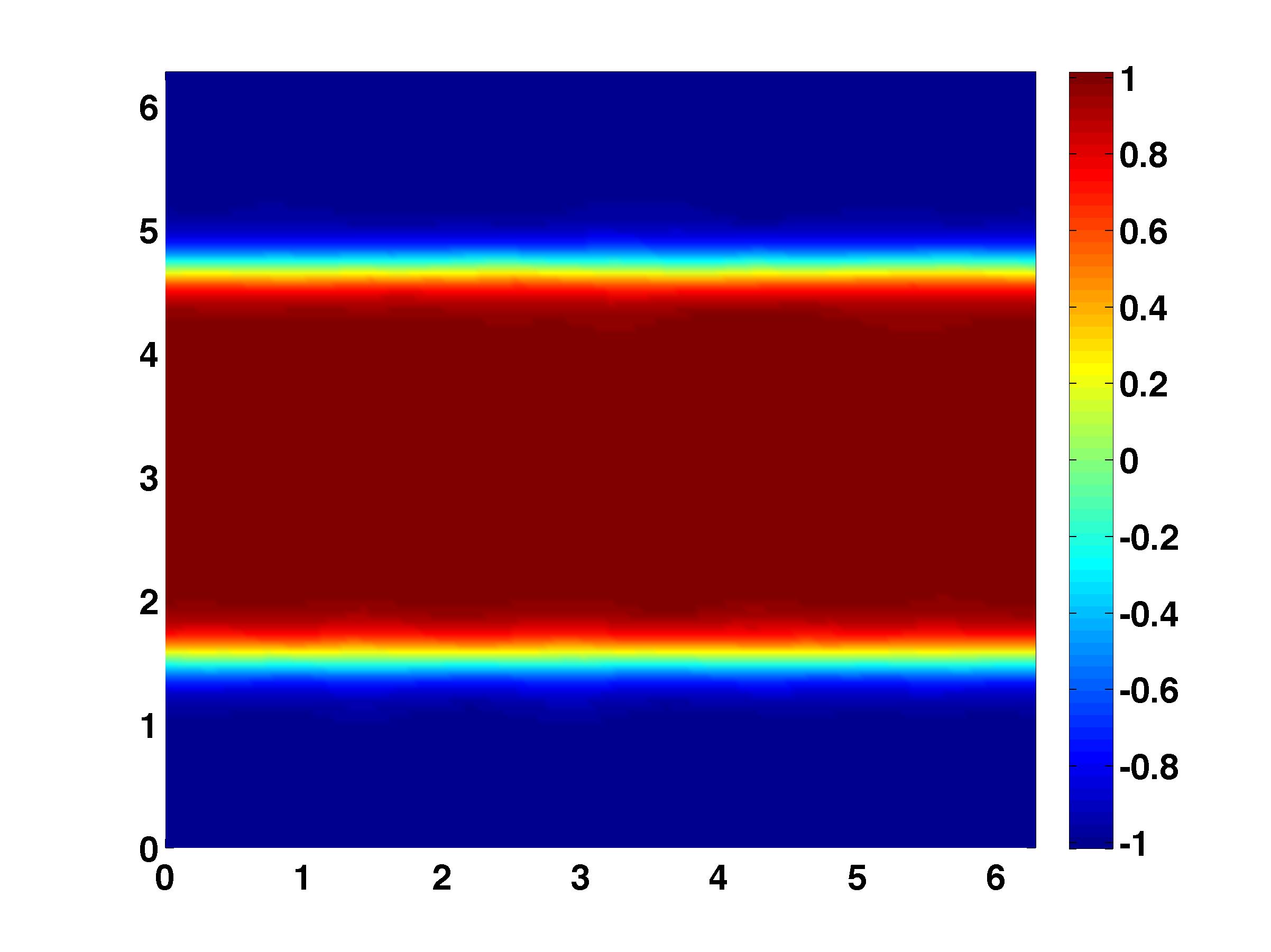}
	\end{subfigure}
	\begin{subfigure}{.49\textwidth}
	\includegraphics[width=\textwidth]{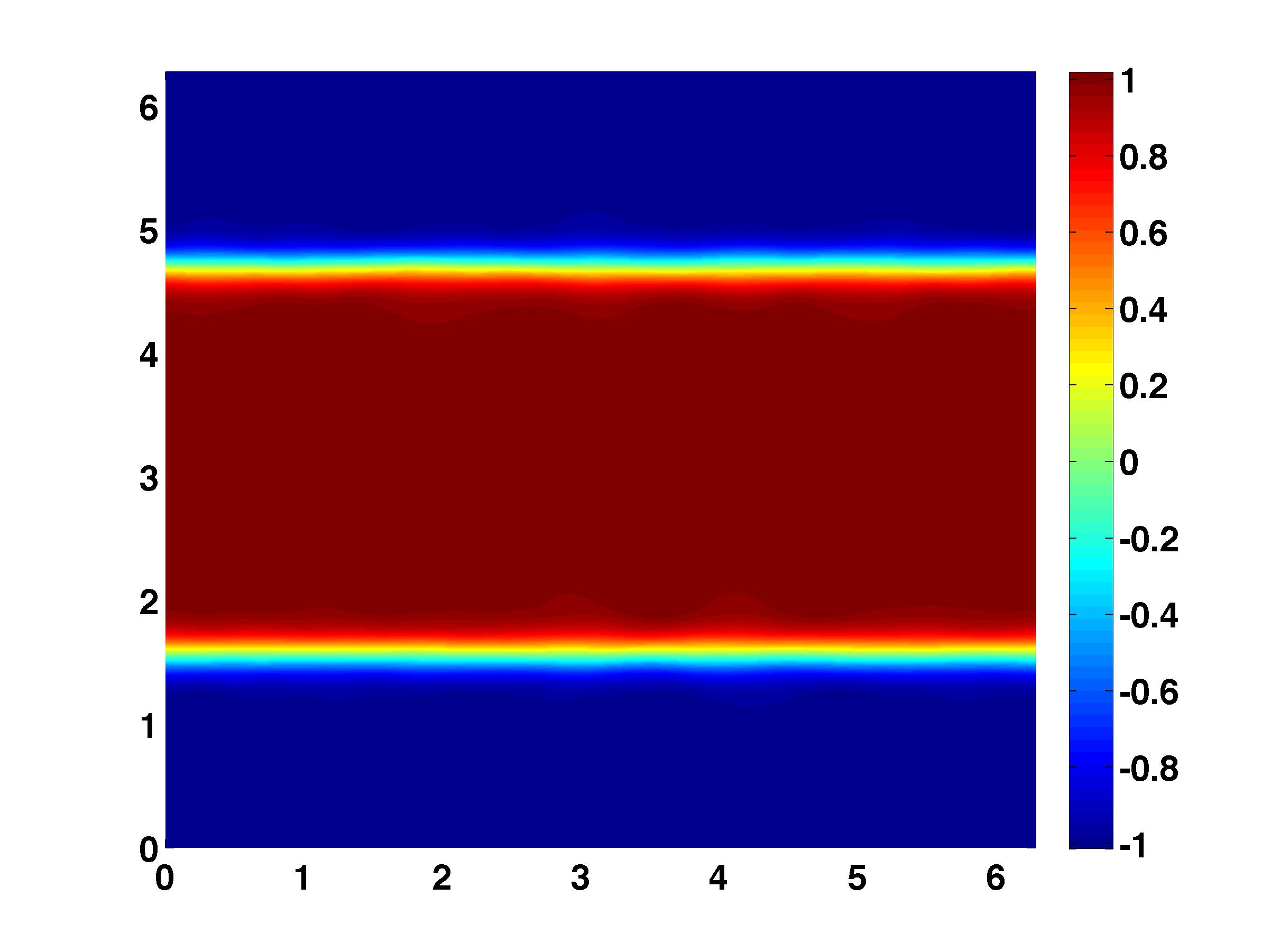}
	\end{subfigure}
	\caption{Flat vortex sheet: Convergence of mean of the first component of the velocity field at time $t=2$ wrt N (number of Fourier modes). Top left $N=128$, Top right $N=256$, Bottom left $N=512$ and Bottom right $N=1024$.}
	\label{fig:7a}
\end{figure}

\begin{figure}
	\centering
	\begin{subfigure}{.49\textwidth}
	\includegraphics[width=\textwidth]{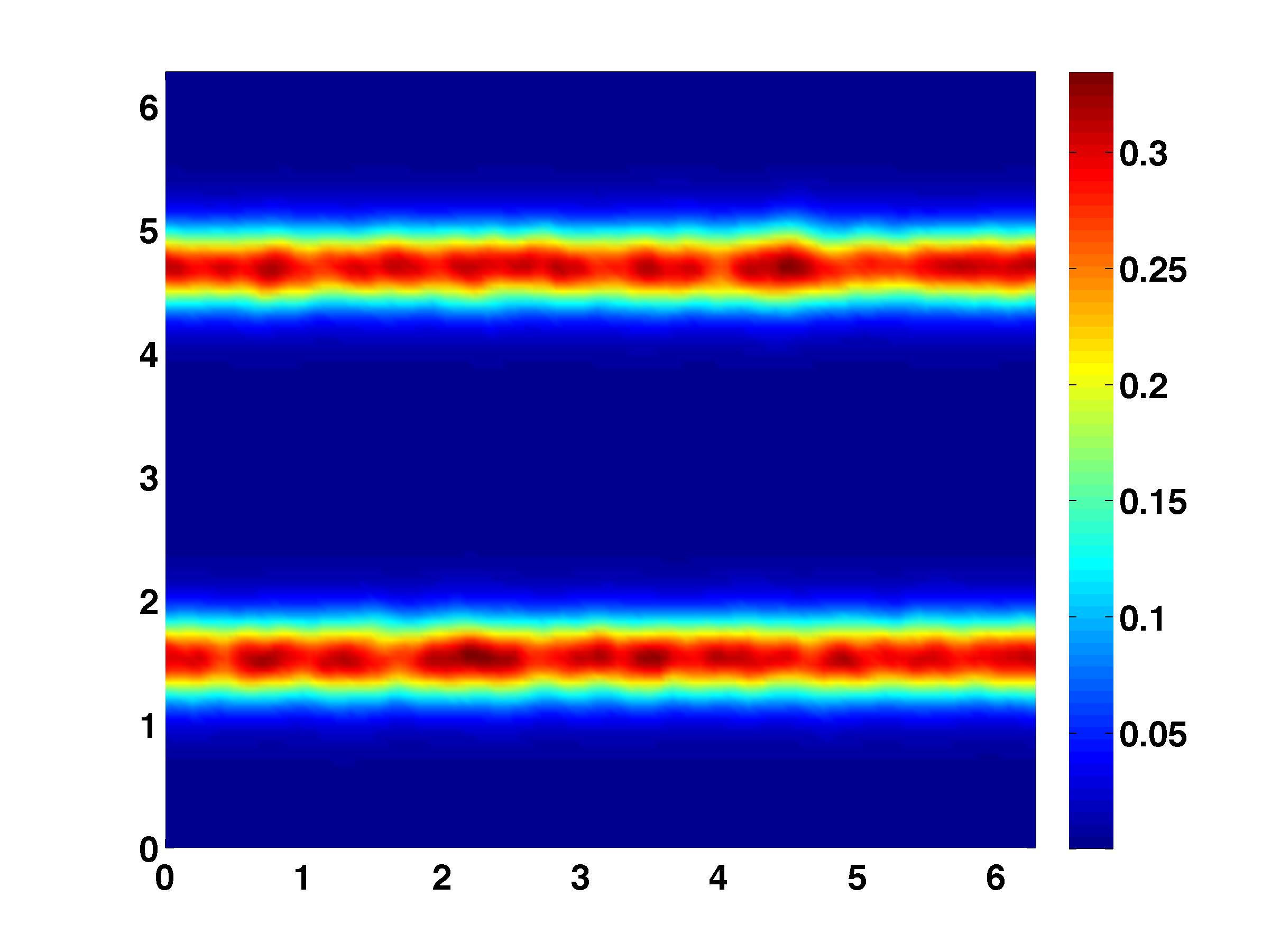}
	\end{subfigure}
	\begin{subfigure}{.49\textwidth}
	\includegraphics[width=\textwidth]{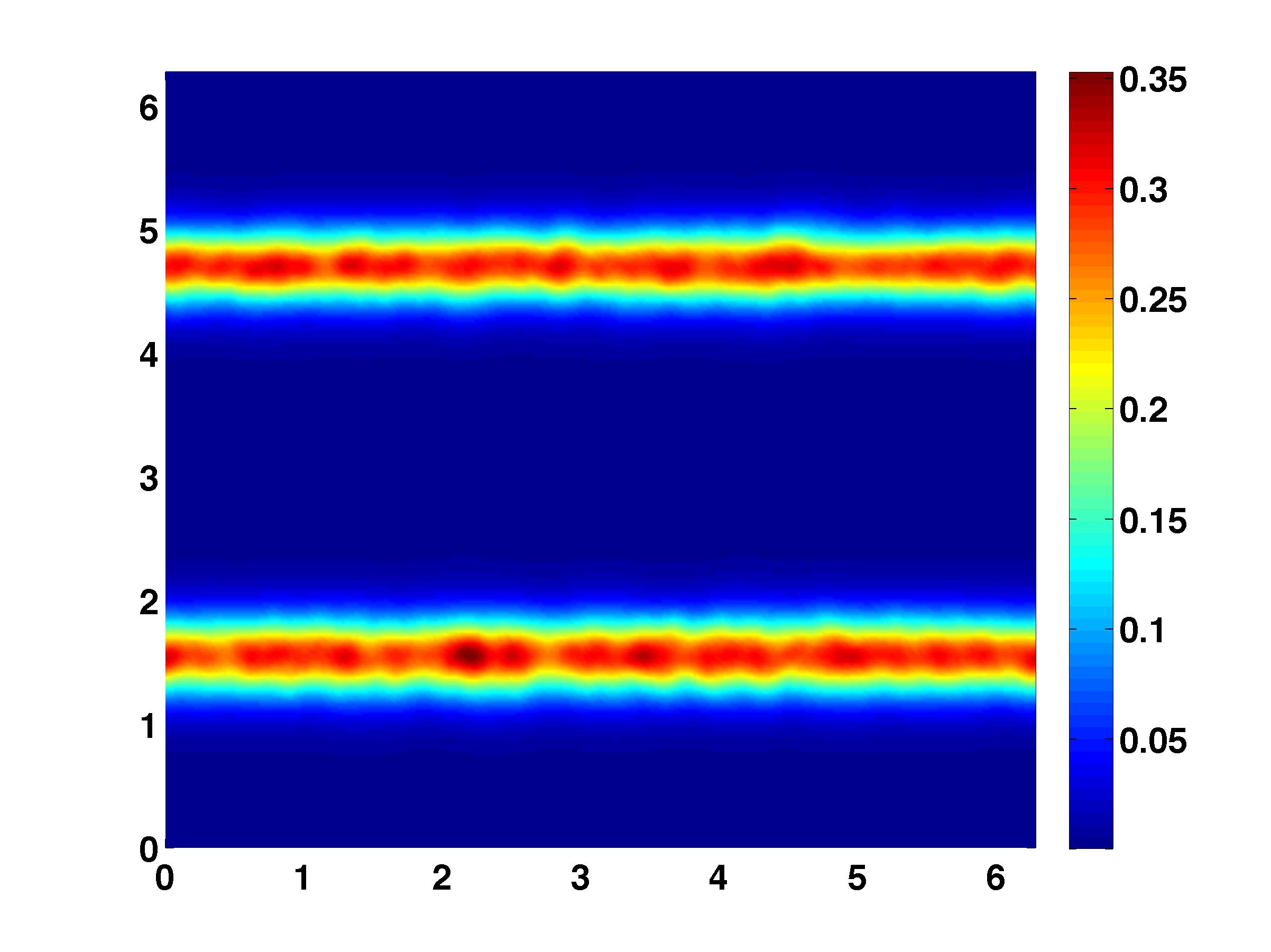}
	\end{subfigure}\\
	\begin{subfigure}{.49\textwidth}
	\includegraphics[width=\textwidth]{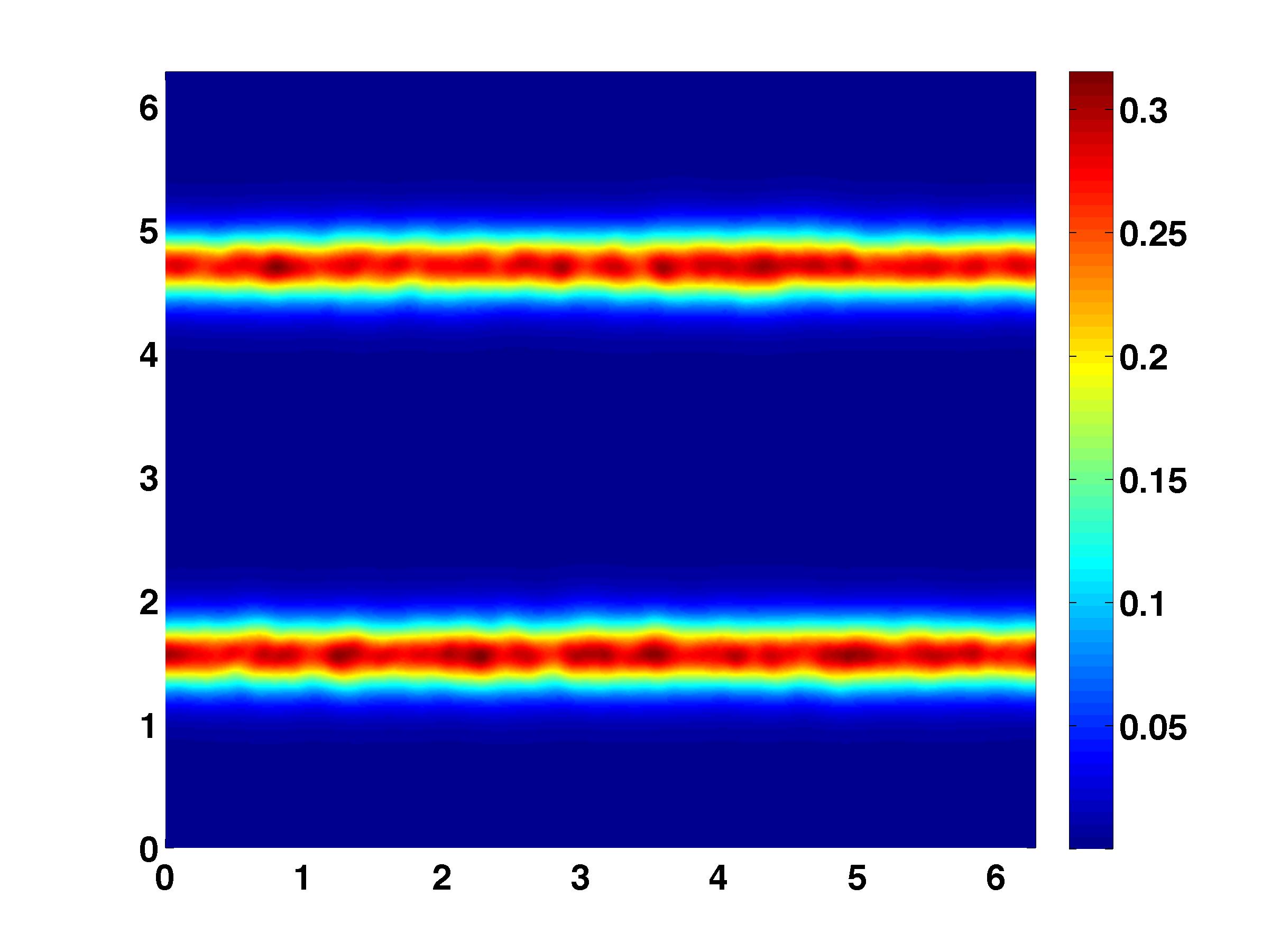}
	\end{subfigure}
	\begin{subfigure}{.49\textwidth}
	\includegraphics[width=\textwidth]{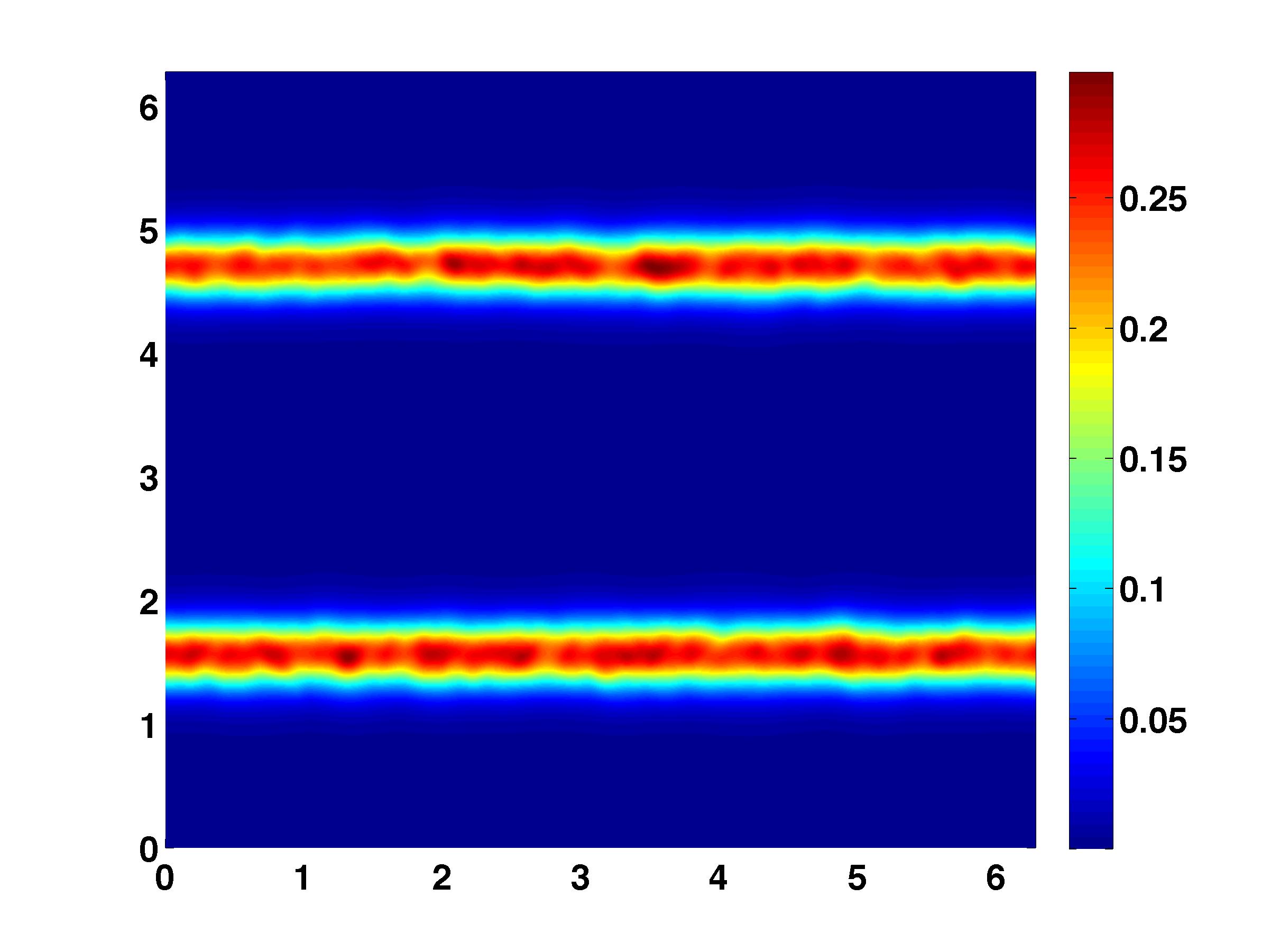}
	\end{subfigure}
	\caption{Flat vortex sheet: Convergence of second moment $\xi_2\xi_2$ of the velocity field at time $t=2$ wrt N (number of Fourier modes). Top left $N=128$, Top right $N=256$, Bottom left $N=512$ and Bottom right $N=1024$.}
	\label{fig:7b}
\end{figure}

\begin{figure}
\begin{subfigure}{.45\textwidth}
\includegraphics[width=\textwidth]{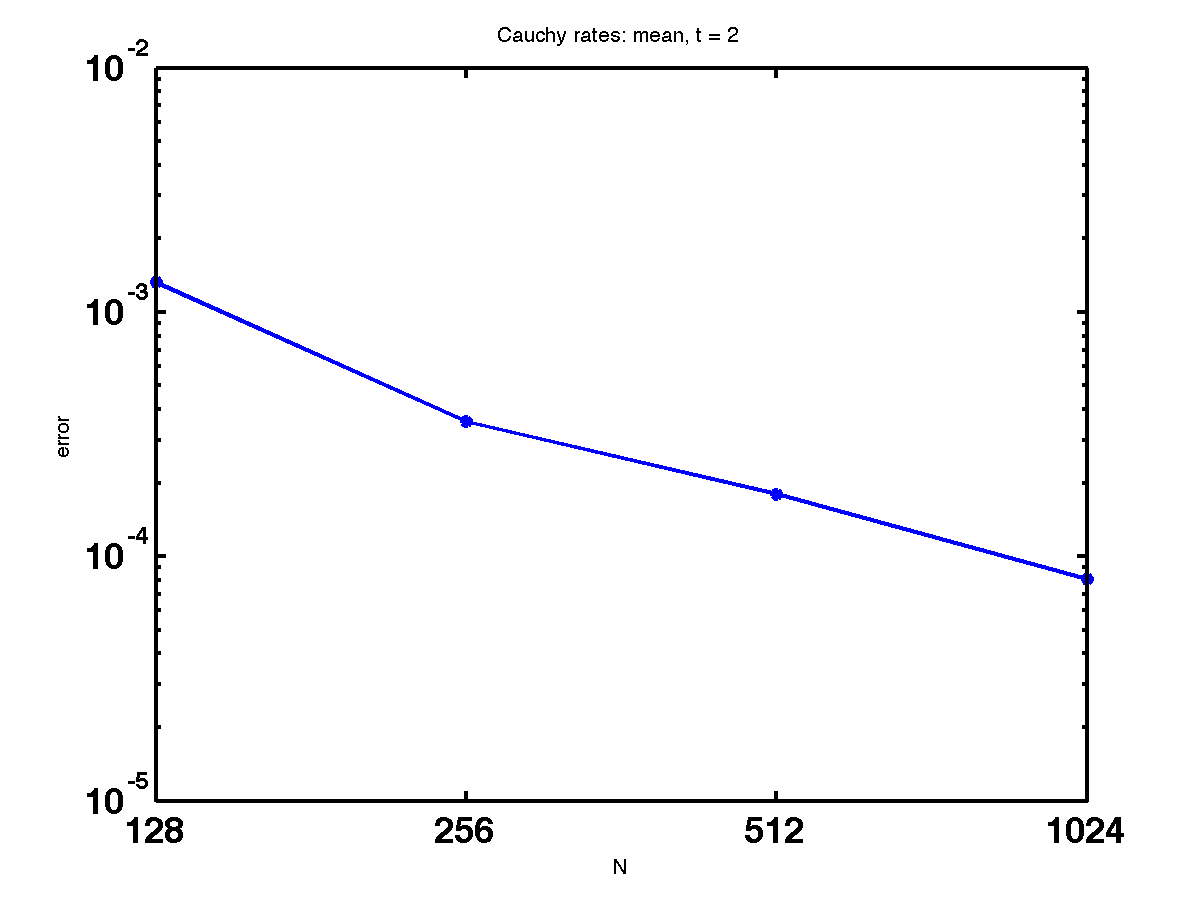}
\caption{Mean}
\end{subfigure}
\begin{subfigure}{.45\textwidth}
\includegraphics[width=\textwidth]{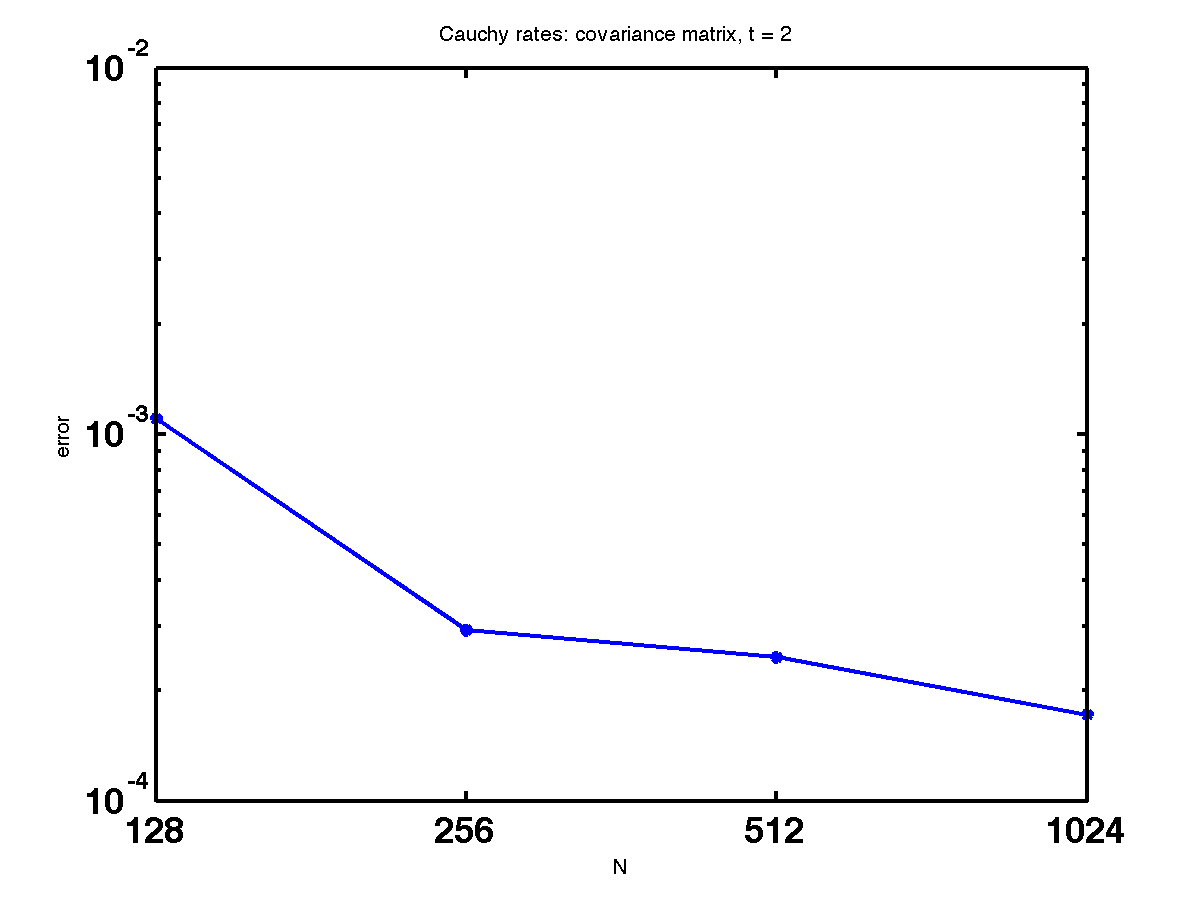}
\caption{Second-moment}
\end{subfigure}
\caption{Flat vortex sheet. Cauchy rates wrt N left (mean) right (second moment)}
	\label{fig:7c}
\end{figure}

Given that the flat vortex sheet corresponds to an initial atomic young measure, the final step of algorithm \ref{alg1} consists of letting the perturbation parameter $\delta \rightarrow 0$. For this purpose, we fix $N=512$ and consider approximate Young measures $\nu^{\rho,\delta}_N$ for successively smaller values of $\delta$. The results for the mean of the first component of the velocity field and the second moment of the second component of the velocity field, plotted in figures \ref{fig:9a} and \ref{fig:9b}, show that these statistical quantifies also converge with decreasing perturbation amplitude. This convergence is also verified in figure \ref{fig:9}, where successive differences of the mean and the second moment in $L^2$ are displayed. 

\begin{figure}
 \begin{subfigure}[b]{.45\textwidth}
    \includegraphics[width = \textwidth]{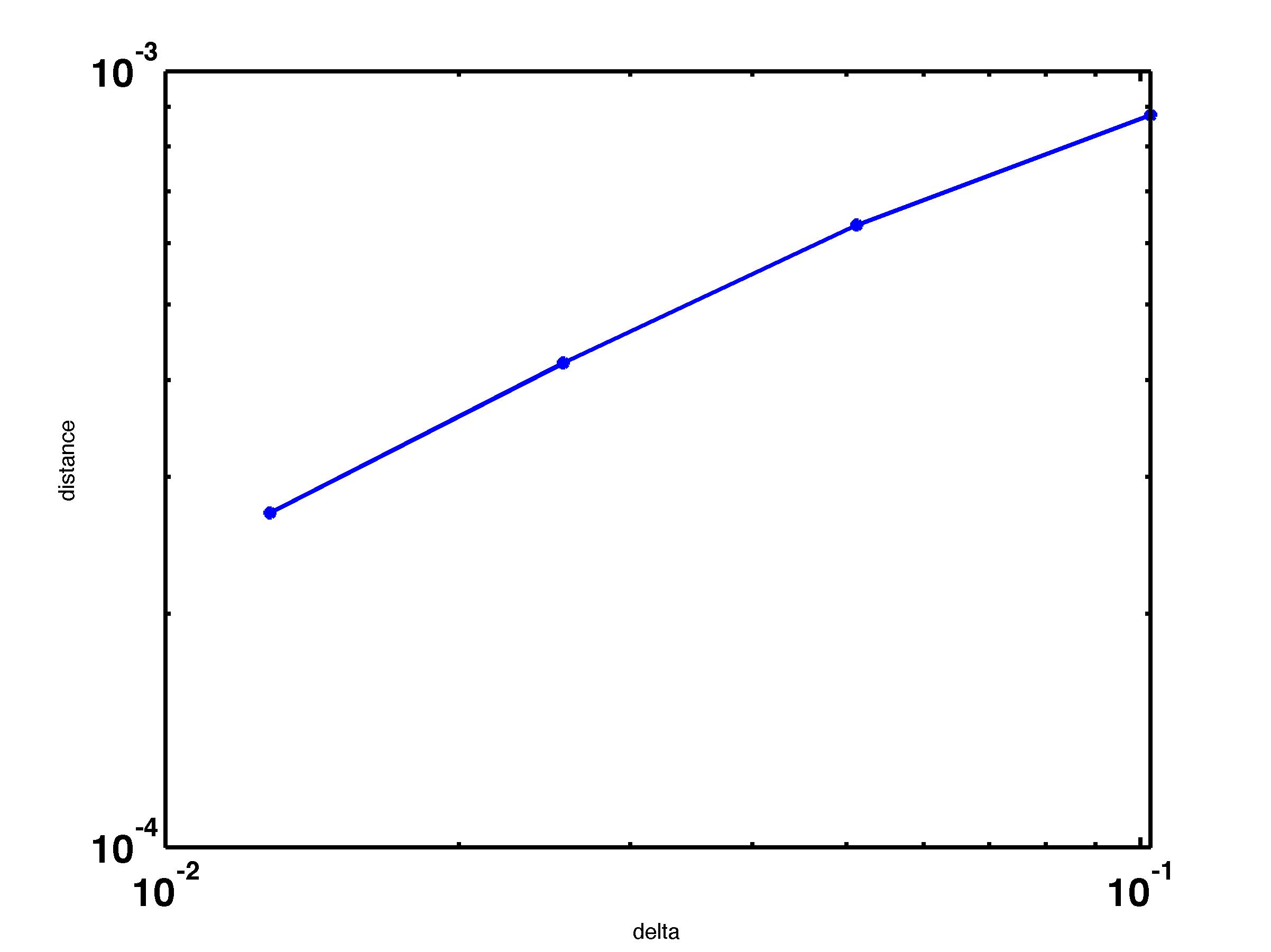}
    \caption{mean}
    \end{subfigure}
    \begin{subfigure}[b]{.45\textwidth}
    \includegraphics[width = \textwidth]{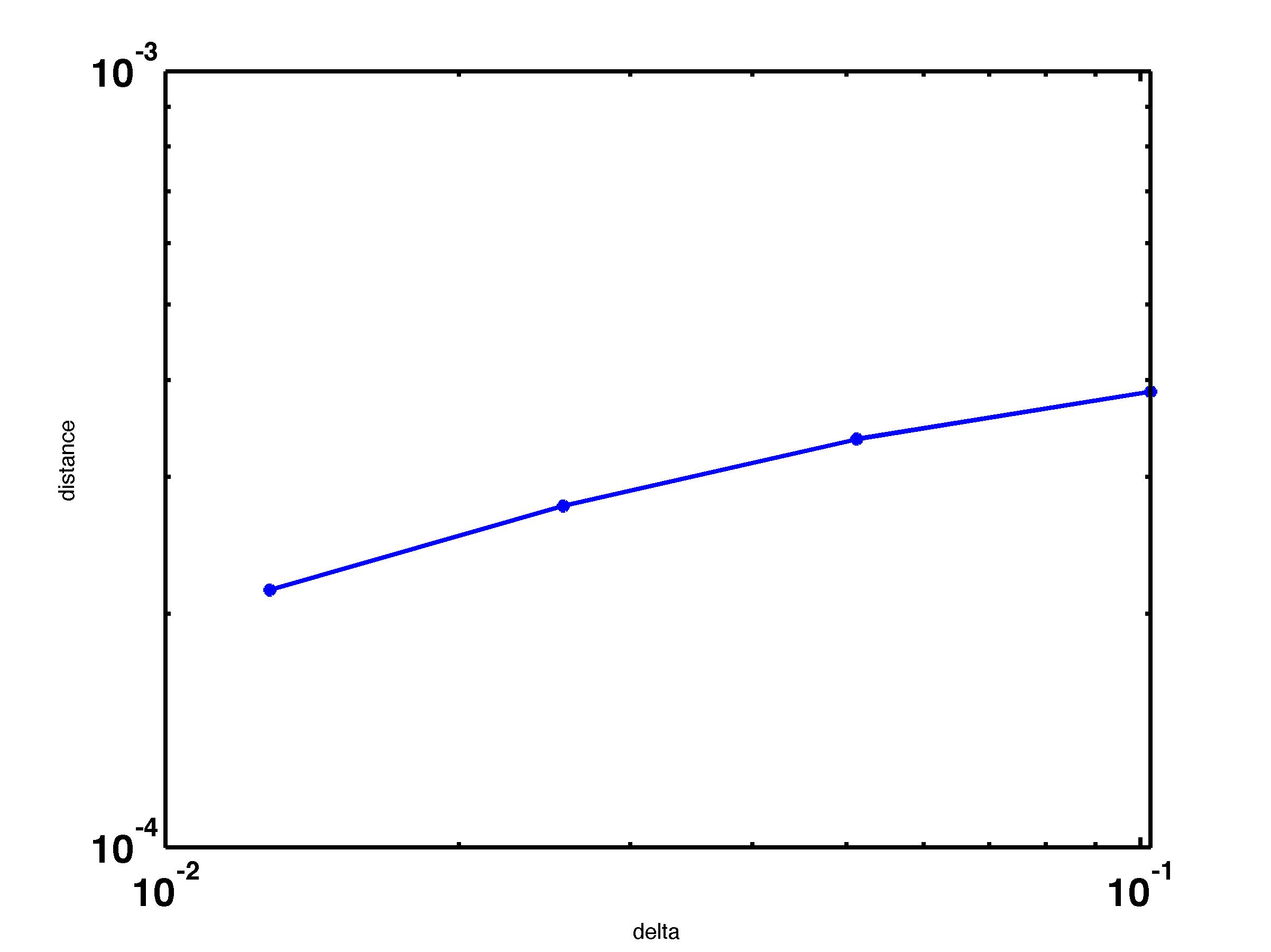}
    \caption{$(x_2x_2)$--second moment}
    \end{subfigure}
		\caption{Flat vortex sheet: Cauchy rates wrt $\delta$ left (mean) right (second moment)}
	\label{fig:9}
\end{figure}

\begin{figure}
	\centering
	\begin{subfigure}{.49\textwidth}
	\includegraphics[width=\textwidth]{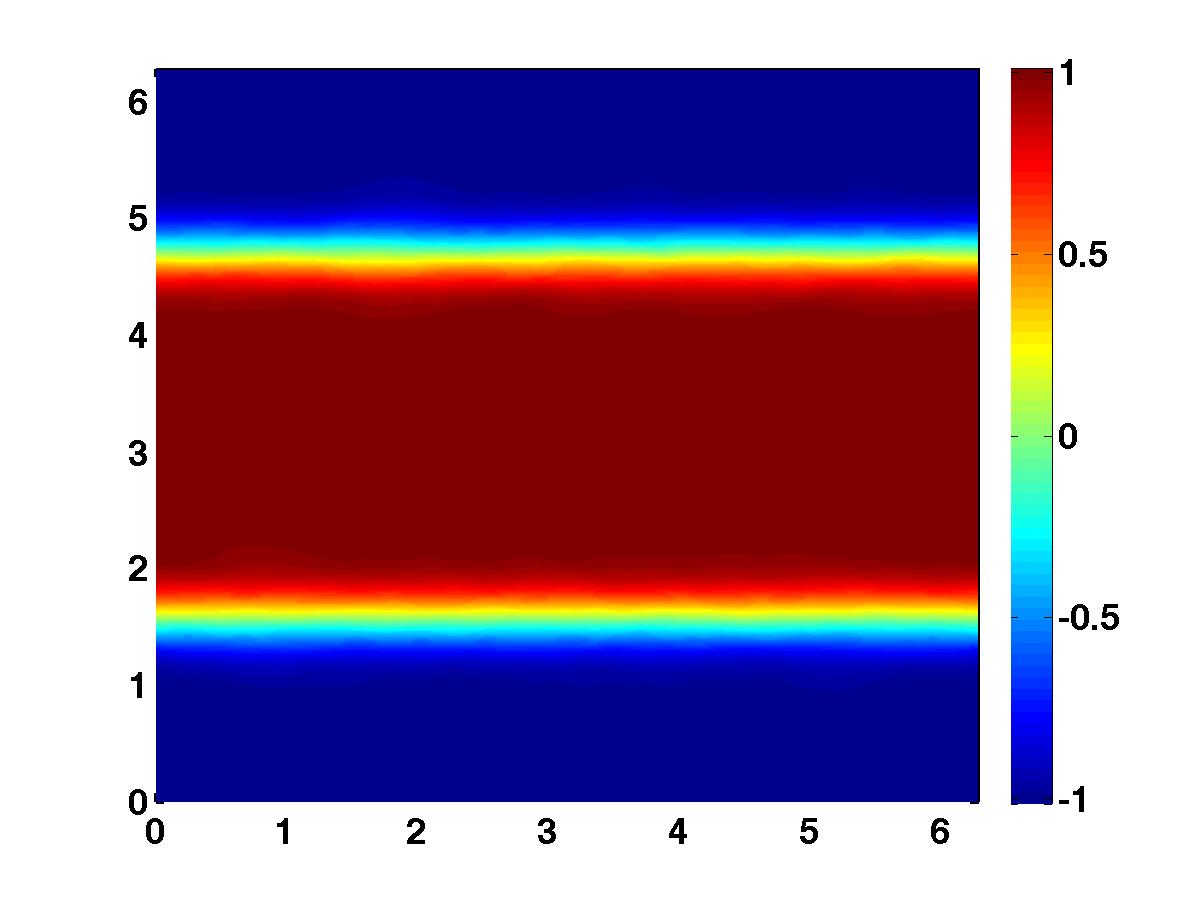}
	\end{subfigure}
	\begin{subfigure}{.49\textwidth}
	\includegraphics[width=\textwidth]{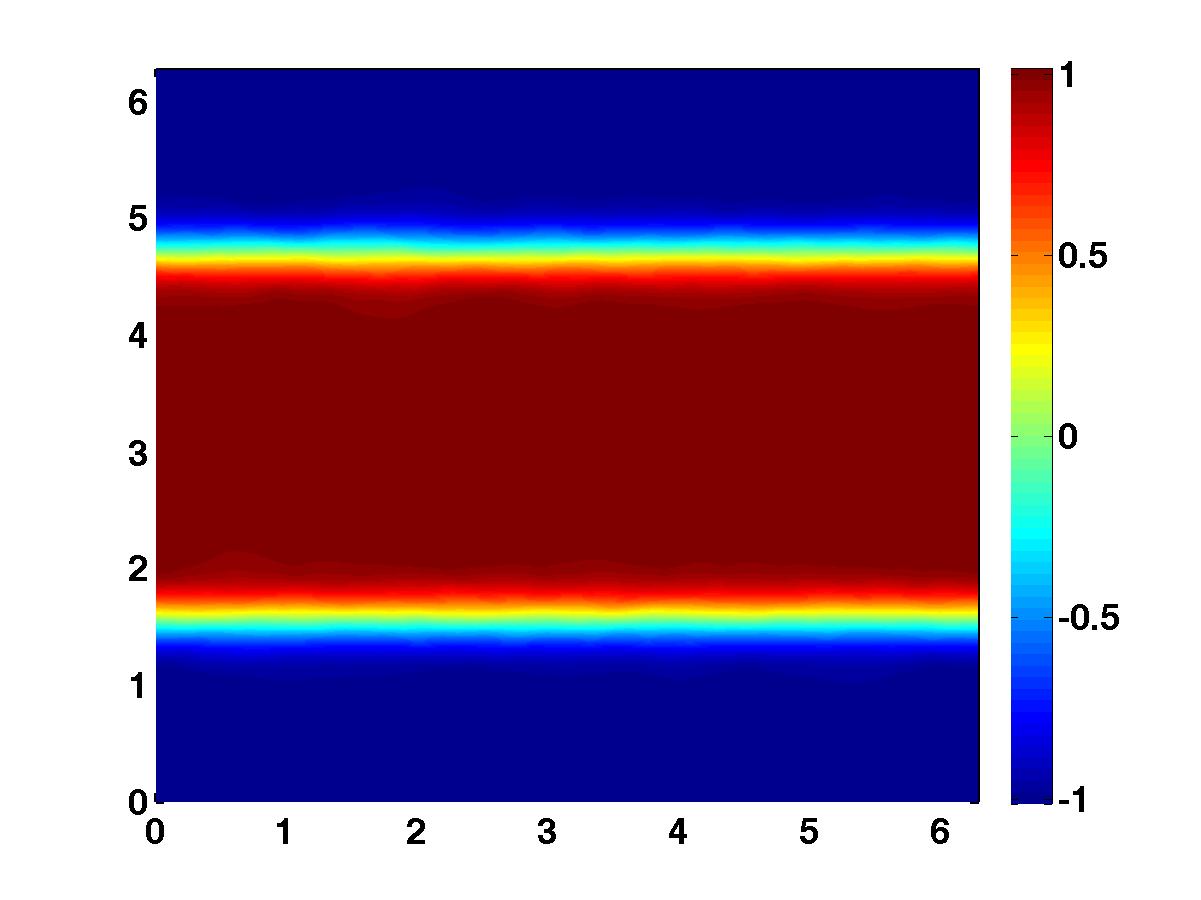}
	\end{subfigure}\\
	\begin{subfigure}{.49\textwidth}
	\includegraphics[width=\textwidth]{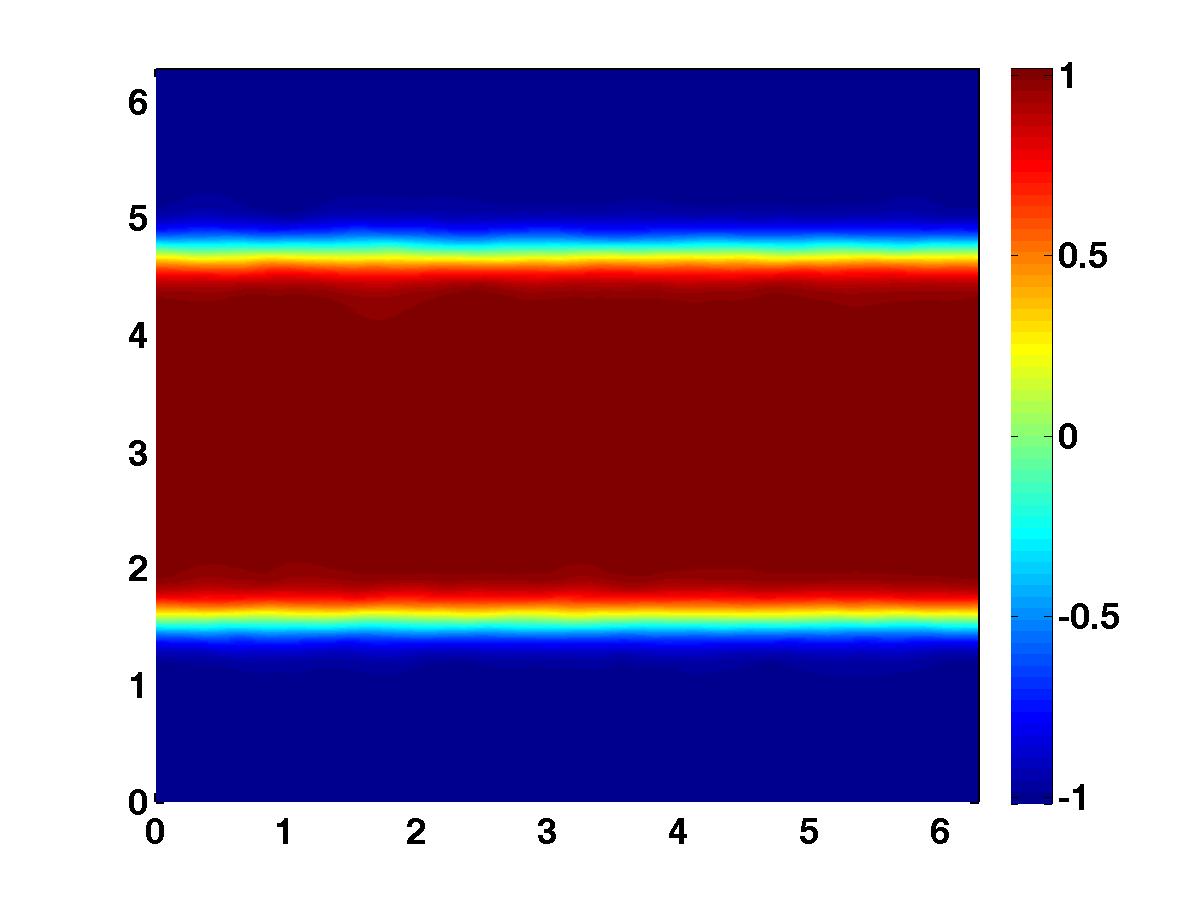}
	\end{subfigure}
	\begin{subfigure}{.49\textwidth}
	\includegraphics[width=\textwidth]{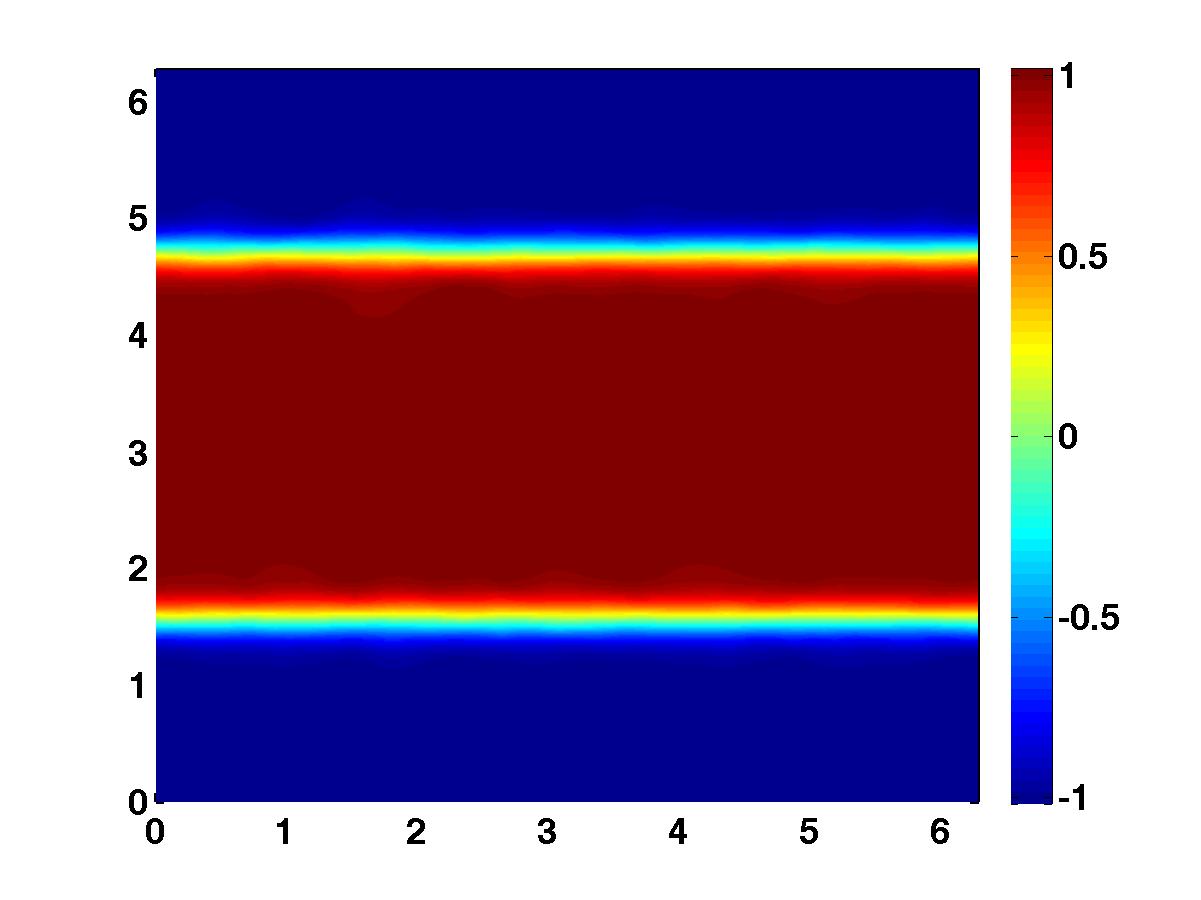}
	\end{subfigure}
	\caption{Flat vortex sheet: Convergence of mean of the first component of the velocity field at time $t=2$ wrt $\delta$ (perturbation parameter). Top left $\delta=0.1024$, Top right $\delta=0.0512$, Bottom left $\delta = 0.0256$ and Bottom right $\delta=0.0128$.}
	\label{fig:9a}
\end{figure}

\begin{figure}
	\centering
	\begin{subfigure}{.49\textwidth}
	\includegraphics[width=\textwidth]{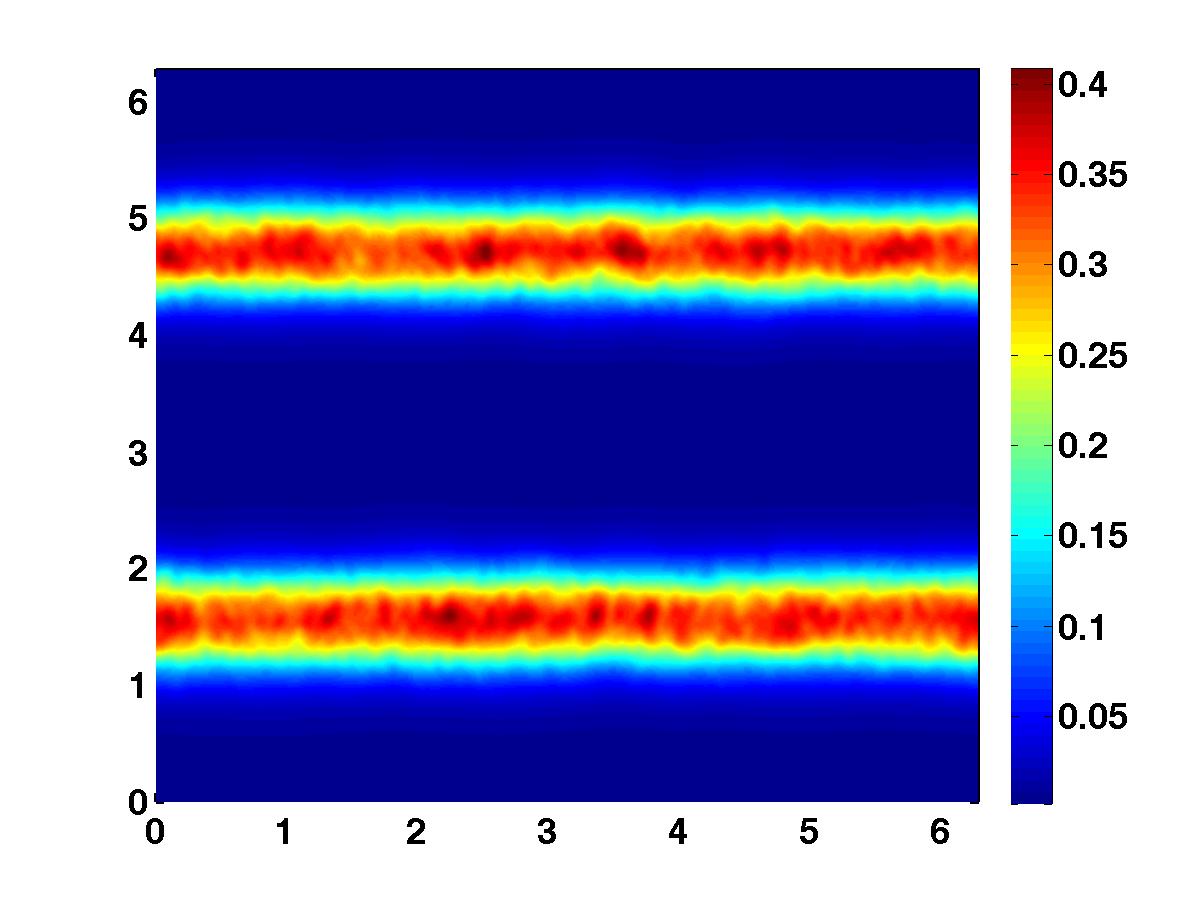}
	\end{subfigure}
	\begin{subfigure}{.49\textwidth}
	\includegraphics[width=\textwidth]{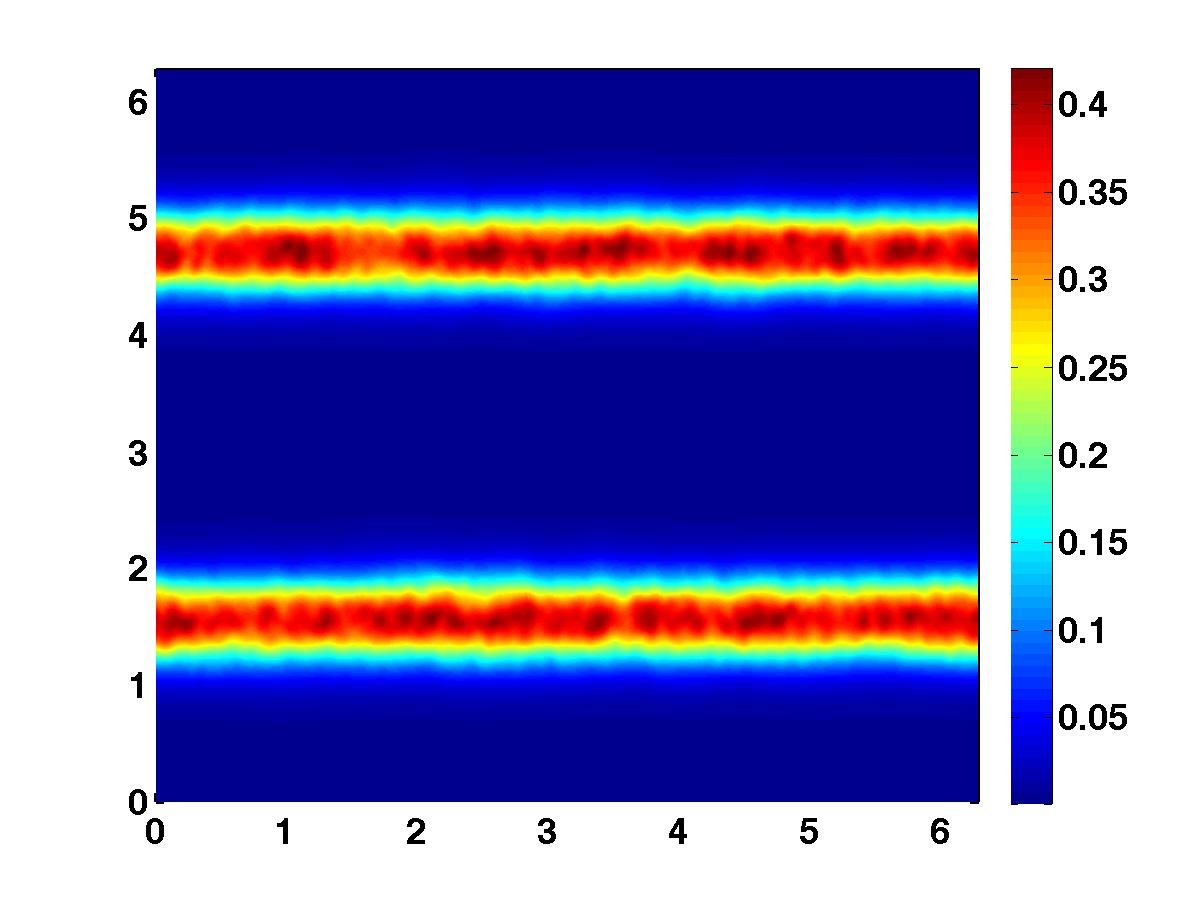}
	\end{subfigure}\\
	\begin{subfigure}{.49\textwidth}
	\includegraphics[width=\textwidth]{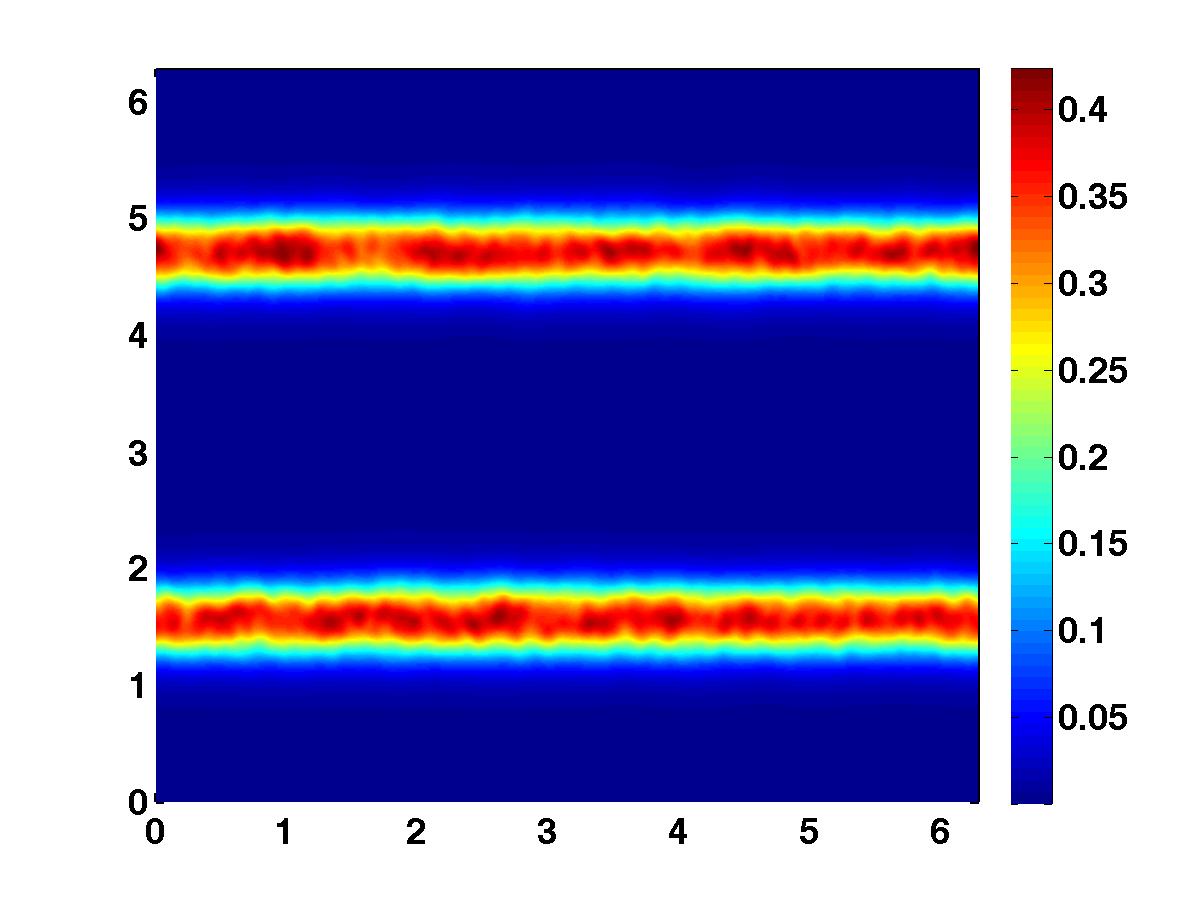}
	\end{subfigure}
	\begin{subfigure}{.49\textwidth}
	\includegraphics[width=\textwidth]{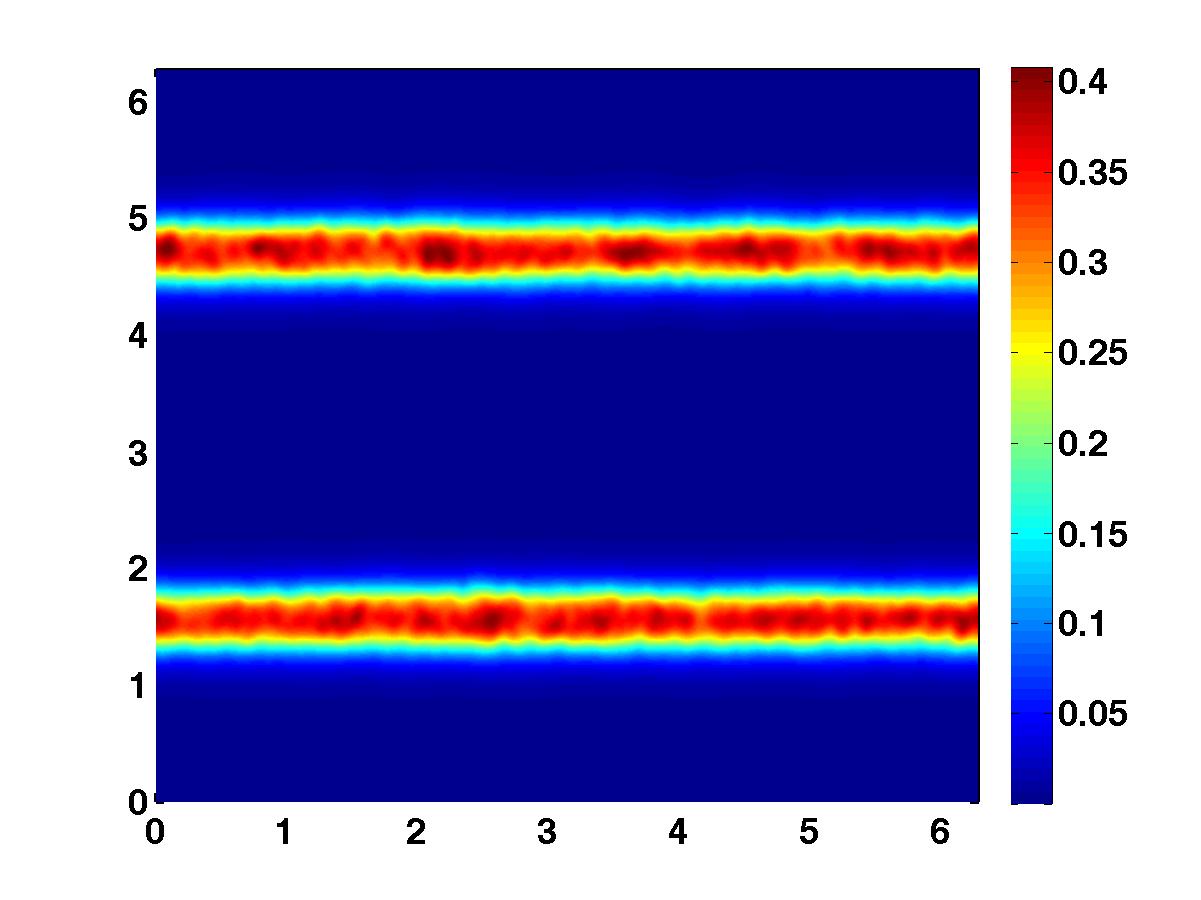}
	\end{subfigure}
	\caption{Flat vortex sheet: Convergence of second-moment of the second component of the velocity field at time $t=2$ wrt $\delta$ (perturbation parameter). Top left $\delta=0.1024$, Top right $\delta=0.0512$, Bottom left $\delta = 0.0256$ and Bottom right $\delta=0.0128$.}
	\label{fig:9b}
\end{figure}
\begin{figure}
\begin{subfigure}[b]{.4\textwidth}
\includegraphics[width=\textwidth]{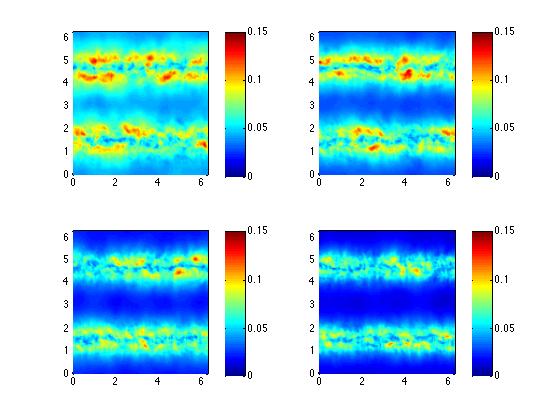}
\caption{distribution in space, $t = 4$, $\delta \to 0$ }
\end{subfigure}
\begin{subfigure}[b]{.4\textwidth}
\includegraphics[width=\textwidth]{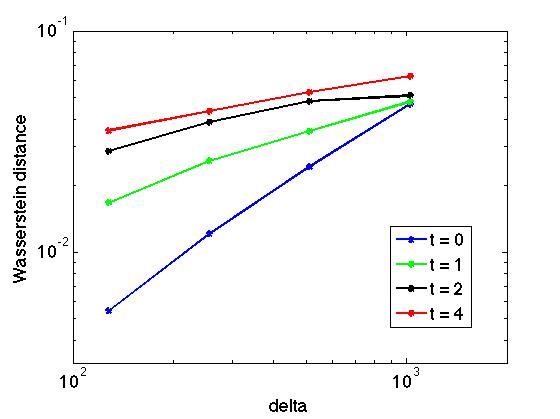}
\caption{Cauchy rates in the mean}
\end{subfigure}
\caption{Cauchy rates in the Wasserstein distance $\mathcal W_1$}
\label{wasserstein}
\end{figure}

The convergence results for statistical quantities such as the mean and the second moment, with respect of the resolution as well as the perturbation parameter, are consistent with the prediction of narrow convergence in Theorem \ref{thm:alphaconv}. Is the convergence even stronger than the predicted narrow convergence ? To examine this question, we follow \cite{FKMT1} and compute the Wasserstein distance for $\nu^{\rho,\delta}_{x,t}$ as probability measures in phase space. Again, we have computed the $1$-Wasserstein distance between successive approximations $\delta$ vs. $\delta/2$, as $\delta \to 0$. 
The results are shown in Figure \ref{wasserstein}. (A) displays the pointwise values $\mathcal W_1(\nu^{\rho,\delta}_{x,t}, \nu^{\rho,\delta/2})$, while (B) is a plot of the mean rates  
\[\int \mathcal W_1(\nu^{\rho,\delta}_{x,t}, \nu_{x,t}^{\rho,\delta/2})\dx,\]
at different times $t=0,1,2,4$.

Unexpectedly, We observe convergence even in the much stronger Wasserstein metric. This type of strong convergence was also observed in the context of compressible Euler equations of gas dynamics in \cite{FKMT1}. 
\section{Further properties of the vortex sheet}
\label{sec:vs}
In this section, we will investigate the above described computed (admissible) measure valued solution of the Euler equations, corresponding to the flat vortex sheet data \eqref{initdata} in considerable detail. To begin with, we can fix the smoothing parameter $\rho > 0$ and the perturbation parameter $\delta$ and let the number of Fourier modes $N \rightarrow \infty$. Numerical results, presented in figures \ref{fig:7c}, show that the approximation converge to a Young measure $\nu^{\rho,\delta}$. In fact, one can also realize $\nu^{\rho,\delta}$ as the law of the random field $X_{\rho,\delta}$ which corresponds to the (path-wise) solution of the Euler equations with smooth initial data $X^0_{\rho,\delta}(\omega)$. We summarize this fact and some other interesting analytical properties of the limit measure $\nu^{\rho,\delta}$ below.
\begin{theorem}\label{atomicuniqueness}
For all values of $\rho, \delta$, the measure-valued solution $\nu^{\rho, \delta}$ has the following properties.
\begin{itemize}
\item $\nu^{\rho, \delta}$ is translationally invariant with respect to the $x_1$-direction, i.e. we have 
\[\nu^{\rho, \delta}_{x_1,x_2,t} = \nu^{\rho, \delta}_{x_1+h,x_2,t} \]
for any $h \in \bbR$ and $(x_1,x_2,t) \in \bbT^2\times \bbR_+$.
\item The mean $\overline \nu^{\rho,\delta}= \langle \nu^{\rho, \delta}, \xi\rangle$ has vanishing second component.
\item If $\nu^{\rho, \delta}$ is atomic, then it is stationary. 
\item For each fixed $\omega \in \Omega$, the realizations $X_{\rho,\delta}(\omega)$ are smooth solutions to the Euler equations with $X_{\rho,\delta}^0(\omega)$ smooth initial data, such that $X_{\rho,\delta}^0(\omega) \to v^0$ in $L^2$ as $\rho, \delta \to 0$. Moreover, we have a uniform bound on the vorticity in the $H^{-1}$ norm. 
\end{itemize}

All of these properties -- except for the smoothness of the random fields $X_{\rho,\delta}$ -- also hold for any limiting measure $\nu^{\rho,\delta}_{x,t} \wsto \nu_{x,t} $, obtained in the limit $\rho,\delta \to 0$, i.e. we are allowed to formally set $\delta = \rho = 0$.
\end{theorem}
\begin{proof}
We start with the proof of property (1). The statistics of the perturbation ansatz for each interface 
$$p_\delta(x) = \sum_{k=1}^{N} \alpha_k \sin(kx_1-\beta_k)$$
 is invariant with respect to translation in the $x_1$-direction. For any $h\in \bbR$, the values $\beta_1-h/1,\; \dots,\; \beta_N-h/N$ have the same probability of occuring as $\beta_1,\; \dots,\; \beta_N$. Hence 
 $$\mathrm{Prob}[p_\delta(x+he_1) \in A] =\mathrm{Prob}[p_\delta(x) \in A]$$
  for any measurable set $A$ and any $h\in \bbR$. We obtain equality of the law 
  $$\cL(v^0_{\rho}(x_1,x_2-p_\delta(x_1))) = \cL(v^0_{\rho}(x_1+h,x_2-p_\delta(x_1+h)))$$
   and hence of the initial data 
   $$\cL(X^0_{\rho,\delta}(x_1,x_2)) = \cL(X^0_{\rho,\delta}(x_1+h,x_2)).$$
   Finally, because the Euler equations are translation-invariant, it follows that 
$$\nu^{\rho,\delta}_{x_1,x_2,t} = \cL(X_{\rho,\delta}(x_1,x_2)) = \cL(X_{\rho,\delta}(x_1+h,x_2)) = \nu^{\rho,\delta}_{x_1+h,x_2,t}.$$

To prove (2), we proceed as follows. Let $\eta^{\rho,\delta}$ be the vorticity corresponding to the random field $X_{\rho,\delta}$. Taking the mean and interchanging integration and differentiation, we see that $\overline \eta^{\rho,\delta}$ is the vorticity corresponding to the mean $\overline \nu^{\rho,\delta}$. By property (1), the mean is independent of $x_1$.  The same must be true of the mean vorticity  $\overline \eta^{\rho,\delta}$, i.e. we have $\partial_{x_1} \overline \eta^{\rho,\delta} = 0$. It follows that also for the second component of $\overline \nu^{\rho,\delta}$, we have 
$$\overline \nu^{\rho,\delta}_2 = \Delta^{-1} \partial_{x_1} \overline \eta^{\rho,\delta} = 0.$$

We come to property (3). Assume that $\nu^{\rho,\delta}_{x,t} = \delta_{v(x,t)}$. By property (2), the second component of the mean $\overline \nu^{\rho,\delta}_{x,t} = v(x,t)$ vanishes, i.e. $v_2= 0$. Furthermore, $v$ is independent of $x_1$ by property (1). It is straightforward to check that these two observations imply that $v$ is a stationary solution.

We recall that the Leray projection is an orthogonal $L^2$ projection $(\ast)$ and that $X_{\rho,\delta}^0(\omega)$ is obtained in three steps $(\ast\ast)$: In the first step, the intital datum $v^0$ is mollified to obtain a smooth field $v^0_\rho$. In a second step, we determine a random perturbation of the interfaces $p_\delta(x)$, which yields a field $v^0_{\rho,\delta}(x) = v^0_\rho(x-p_\delta(x))$. In the last step, we project this field to the space of divergence-free vector fields using the Leray projection to obtain $X_{\rho,\delta}^0(\omega)$. In particular, we find that
\[\Vert v^0 - X^0_{\rho,\delta}\Vert_{L^2} \overset{(\ast)}{\le} \Vert v^0 - v^0_{\rho,\delta}\Vert_{L^2}  \overset{(\ast\ast)}{\le} \Vert v^0 - v^0_{\rho}\Vert_{L^2}+ \Vert v^0_\rho - v^0_{\rho,\delta}\Vert_{L^2}  \]
Next, note that we have uniform $L^\infty$ bounds $\Vert v^0\Vert_{L^\infty}, \, \Vert v^0_\rho\Vert_{L^\infty}, \, \Vert v^0_\delta\Vert_{L^\infty}\le 1$ and that all of these fields are pointwise $= \pm e_1$, except in a region with width of order $O(\delta + \rho)$. We conclude that $\Vert v^0 - X^0_{\rho,\delta}\Vert_{L^2}  \le O(\rho+\delta)$. Finally, a uniform $L^2$ bound on a vector field implies a uniform $H^{-1}$ bound on its vorticity. This concludes the proof.
\end{proof}
\subsection{Non-atomicity of the limit measure valued solution}
\label{sec:nonatomic}
\subsubsection{Non-zero variance}
One of the most important questions concerning the measure valued solution realized as a limit of the approximations computed using Algorithm \ref{alg1} applied to the flat vortex sheet initial data \eqref{initdata} is whether this measure is atomic or not, i.e. whether the limit measure valued solution is a weak solution of the Euler equations \eqref{EQ}? To answer this question, we focus on the variance of the computed approximations. By property (2) of theorem \ref{atomicuniqueness}, we see that for a fixed $\rho,\delta$, the computed Young measures will be invariant in the $x_1$-direction. We fix $N=512$, and present a $x_1 = \mathrm{const}$ slice of the mean and the variance of the velocity field $v_1$ in the $x_2$ direction for different values of $\delta$. The results shown in figure \ref{fig:10} demonstrate that there is convergence as $\delta \rightarrow 0$. Furthermore, the mean, as $\delta$ is reduced, does not coincide with the initial velocity discontinuity. The variance is also very different from zero, at least along two patches (symmetric with respect to $x_2 = \pi$). We denote these two patches as the \emph{turbulence zone}. This is the first indication that the computed measure valued solution is \emph{not atomic}. 
\subsubsection{Spread of the turbulence zone in time}
To further test the issue of atomicity of the limit measure, we use property (3) of theorem \ref{atomicuniqueness}. This property provides a clear criterion for atomicity i.e, if the limit measure is atomic, then it must be stationary  (coincide with the initial flat vortex sheet \eqref{initdata}). We investigate the stationarity of the limit measure by considering the time dependent map for (the spatial average of) the variance,
\begin{equation}
t \mapsto \int_{\bbT^2} {\rm Var} \left(\nu^{\rho,\delta}_{x,t}\right)\dx,
\label{eq:vart}
\end{equation}
as $\delta \to 0$. Given the fact that the variance is non-zero only in the turbulence zone, we can interpret the above quantity as the mean spreading rate (in time) of the turbulence zone. In figure \ref{fig:12}, we show how the zone spreads in time with respect to different values of $\delta$. We observe that
 \begin{itemize}
 \item The spread rate of the turbulence zone converges as $\delta \rightarrow 0$.
 \item The limiting spread rate is \emph{non-zero}, implying that the turbulence zone spreads out at a linear rate in time.
 \end{itemize}
 Thus, the limit Young measure is not stationary and hence, non-atomic.
 
 Although, we are unable to provide a rigorous proof for the observed linear spread rate of the turbulence zone and of the consequent non-atomicity of the limit measure, we present a rigorous upper bound on the rate at which variance can increase. To see this, we let  $(\nu, \nu^\infty,\lambda)$ be an (admissible) measure valued solution (MVS) with atomic initial data, concentrated on $v_0$. Then $(\nu, \nu^\infty,\lambda)$ satisfies
\begin{align}\label{weak}
\int_0^T\int_{\bbT^2} \langle \nu_{x,t}, \xi \rangle \chi' (t)\phi(x) + \langle \nu_{x,t},\xi \otimes \xi \rangle : \nabla \phi (x)\chi(t) \,& dx \, dt  \\ + \int_0^T\left( \int_{\bbT^2} \langle \nu_{x,t}^\infty, \theta\otimes\theta\rangle : \nabla \phi(x) \chi(t) \, \lambda_t(dx)\right)\, dt &= -\int_{\bbT^2} v_0(x)\cdot \phi(x) \, dx, \nonumber
\end{align}
for all $\phi \in C^\infty(\bbT^2;\bbR^2)$, and $\chi \in C^\infty_c([0,T))$ with $\chi(0)=1$. 

If we take $\phi \ast \rho_\epsilon$ as a test function, where $\rho_\epsilon = \epsilon^{-2} \rho(x/\epsilon)$, $\rho\in C^\infty_c(\bbT^2)$ is a standard mollifier on $\mathbb T^2$ and $\ast$ denotes convolution, then we obtain 
\begin{align}\label{mollfied}
\begin{aligned}
\int_0^T\int_{\bbT^2} \langle (\rho_\epsilon\ast \nu)_{x,t}, \xi \rangle \chi' (t)\phi(x) + \langle (\rho_\epsilon\ast \nu)_{x,t}&,\xi \otimes \xi \rangle : \nabla \phi(x)\chi(t) \, dx \, dt  \\
 + \int_0^T \int_{\bbT^2} \Big( \int_{\bbT^2} \rho_\epsilon(x-y) \langle \nu_{x,t}^\infty, \theta\otimes\theta\rangle &\, \lambda_t(dx)\Big) : \nabla \phi(y)\chi(t)\, dy\, dt  \\
  &= -\int_{\bbT^2} (v_0\ast \rho_\epsilon)(x)\cdot \phi(x) \, dx
\end{aligned}
\end{align}
 after an application of Fubini's theorem in \eqref{weak}. In the following, we will denote 
 \begin{align*}
\langle \nu^\epsilon_{x,t}, f(\xi)\rangle := \langle (\rho_\epsilon \ast \nu)_{x,t}, f(\xi)\rangle,  \quad
\langle \lambda^\epsilon_{y,t},f(\theta)\rangle := \int_{\bbT^2} \rho_\epsilon(x-y) \langle \nu_{x,t}^\infty, f(\theta)\rangle \, \lambda_t(dx).
\end{align*}
Similarly,  we will write $v_0^\epsilon(x) := (\rho_\epsilon \ast v_0)(x)$ for the mollified initial data. With this notation, equation \eqref{mollfied} takes the form
\begin{align}
\begin{aligned}
\int_0^T\int_{\bbT^2} \langle \nu^\epsilon_{x,t}, \xi \rangle \chi'(t) \phi(x) + \langle \nu^\epsilon_{x,t},\xi \otimes \xi \rangle : \nabla \phi &(x)\chi(t) \, dx \, dt  \\
+ \int_0^T \int_{\bbT^2} \langle \lambda^\epsilon_{x,t},\theta \otimes \theta \rangle : \nabla  \phi(x)\chi(t)\, dx\, dt  &= -\int_{\bbT^2} v_0^\epsilon(x)\cdot \phi(x) \, dx
\end{aligned}
\end{align}
for all $\phi \in C^\infty(\bbT^2;\bbR^2)$, and $\chi \in C^\infty_c([0,T))$ with $\chi(0)=1$. Thus, $(\nu^\epsilon, \lambda^\epsilon\, \mathrm d x)$ is seen to be a MVS with mollified initial data given by $v_0^\epsilon$. \footnote{We will have no need to bring the concentration measure $\langle \lambda_{x,t}^\epsilon, \cdot \rangle \, dx$ into the sliced form $\langle \tilde \nu^\infty, \cdot \rangle \tilde \lambda_t(dx)$. Though, this could certainly be done.}

At this point, let us observe that for any suitable function $f$, we have 
\begin{align}\label{observation}
\begin{aligned}
\int_{\bbT^2} \langle \nu^\epsilon_{x,t}, f(\xi) \rangle \, dx &= \int_{\bbT^2} \langle \nu_{x,t}, f(\xi) \rangle \, dx, \\
\int_{\bbT^2} \langle \lambda^\epsilon_{y,t},f^\infty(\theta)\rangle \, dx &=\int_{\bbT^2} \langle \nu_{x,t}^\infty, f^\infty(\theta)\rangle \, \lambda_t(dx),
\end{aligned}
\end{align}
as follows from an application of Fubini's theorem.

Fix $\epsilon > 0$ for the moment. In the spirit of \cite{BDS1}, we define 
\begin{align}
F(t) &= \int_{\bbT^2} \langle \nu_{x,t}, \frac 12 \left| \xi - v_0^\epsilon \right|^2 \rangle + \lambda_t(\bbT^2) \\
F^\epsilon(t) &=  \int_{\bbT^2} \langle \nu^\epsilon_{x,t}, \frac 12 \left| \xi - v_0^\epsilon \right|^2 \rangle +  \left(|\lambda^\epsilon_{x,t}| \, dx\right)(\bbT^2) , \\
E(t) &=  \int_{\bbT^2} \langle \nu_{x,t}, \frac 12 \left| \xi \right|^2 \rangle + \frac 12 \lambda_t(\bbT^2)  \\
E^\epsilon(t) &=  \int_{\bbT^2} \langle \nu^\epsilon_{x,t}, \frac 12 \left| \xi \right|^2 \rangle +  \frac{1}{2}\left(|\lambda^\epsilon_{x,t}| \, dx\right)(\bbT^2).
\end{align}
By our observation \eqref{observation}, we have $F(t) = F^\epsilon(t)$ and $E(t) = E^\epsilon(t)$ for all $t\ge 0$.

It has been shown in the proof of \cite[Theorem 2]{BDS1} that for any MVS (admissible or not) with sufficiently smooth initial data $v_0^\epsilon(x)$, and with corresponding (strong) solution $v^\epsilon(x,t)$, the following inequality holds:
\begin{align}
F^\epsilon(t) \le E^\epsilon(t) - \frac 12 \int_{\bbT^2} |v_0^\epsilon|^2 \, dx + \frac 12 \int_0^t \Vert \nabla v_0^\epsilon +  (\nabla v_0^\epsilon)^T\Vert_\infty F^\epsilon(\tau) \, d\tau.
\end{align}
Assume now that $(\nu,\nu^\infty,\lambda)$ is in fact an \emph{admissible} solution, so that $E(t) \le \frac 12 \int_{\bbT^2} |v_0|^2 \, dx$ for all $t$. Then
\begin{align*}
F(t) &= F^\epsilon(t) \\
& \le E^\epsilon(t) - \frac 12 \int_{\bbT^2} |v_0^\epsilon|^2 \, dx + \frac 12 \int_0^t \Vert \nabla v_0^\epsilon +  (\nabla v_0^\epsilon)^T\Vert_\infty F^\epsilon(\tau) \, d\tau \\
&\le \frac 12 \int_{\bbT^2} |v_0|^2 \, dx - \frac 12 \int_{\bbT^2} |v_0^\epsilon|^2\, dx + \frac 12 \int_0^t \Vert \nabla v_0^\epsilon +  (\nabla v_0^\epsilon)^T\Vert_\infty F(\tau) \, d\tau.
\end{align*}
Note that the first difference is of order $\epsilon$, while (for a suitable mollifier) $ \Vert \nabla v_0^\epsilon +  (\nabla v_0^\epsilon)^T\Vert_\infty$ can be bounded by $\epsilon^{-1}$. Gronwall's inequality thus implies that
\begin{align*}
F(t) 
&\le \frac 12 \left(\int_{\bbT^2} |v_0|^2- |v_0^\epsilon|^2\, dx \right)\, \mathrm{e}^{\frac 12 \int_0^t \Vert \nabla v_0^\epsilon +  (\nabla v_0^\epsilon)^T\Vert_\infty \, d\tau  }\leq C\epsilon\,  \mathrm{e}^{\frac{t}{2\epsilon}},
\end{align*}
where $C\ge 0$ satisfies $\frac 12 \int_{\bbT^2} |v_0|^2- |v_0^\epsilon|^2\, dx \le C\epsilon$. 

The particular choice $\epsilon = t/2$ now gives the bound
\begin{align}\label{bound}
\int_{\bbT^2} \langle \nu_{x,t}, \frac 12 | \xi - v_0^{(t/2)} |^2 \rangle \dx + \lambda_t(\bbT^2) \le \frac{C\mathrm{e}}2  t,
\end{align}
for $t>0$.

\begin{corollary}
The mean $\bar\nu_{x,t} \overset{t\to 0}\longrightarrow v_0(x)$ converges \emph{strongly} in $L^2(\bbT^2;\bbR^2)$ for any admissible MVS with vortex sheet initial data $v_0$ \eqref{eq:vs}. Furthermore, the spatially averaged variance \eqref{eq:vart} cannot grow more than linearly for such solutions.
\end{corollary}
\begin{proof}
This is an immediate corollary of estimate \eqref{bound}. We have
\begin{align*}
\int_{\bbT^2}  | \overline \nu - v_0 |^2  \dx 
&\le 2\int_{\bbT^2}  | \overline \nu - v_0^{(t/2)} |^2  \dx + 2\int_{\bbT^2}  |v_0 - v_0^{(t/2)} |^2  \dx \\
&\le   2\int_{\bbT^2} \langle \nu_{x,t}, | \xi - v_0^{(t/2)} |^2 \rangle \dx+ 2\int_{\bbT^2}  | v_0 - v_0^{(t/2)} |^2  \dx \\
&\to 0,
\end{align*}
as $t\to 0$, and
\begin{align*}
\int_{\bbT^2} \mathrm{Var}(\nu_{x,t}) \dx 
&= \int_{\bbT^2} \langle \nu_{x,t}, \frac 12 | \xi - \overline \nu|^2 \rangle \dx 
\le  \int_{\bbT^2} \langle \nu_{x,t}, \frac 12 | \xi - v_0^{(t/2)} |^2 \rangle \dx  
\le \frac{C\mathrm{e}}2  t,
\end{align*}
\end{proof}

For the particular choice of a piecewise linear function $v_0^\epsilon$
\[v_0^\epsilon(x) = 
\begin{cases} 
+e_1, 	& x_2 < \pi/2-\epsilon \text{  or  } x_2 > 3\pi/2+\epsilon, \\ 
\dfrac{\pi/2-x}{\epsilon}\; e_1, & \vert \pi/2-x_2 \vert \le \epsilon, \\
-e_1, & \pi/2+\epsilon \le x_2 < 3\pi/2-\epsilon, \\
\dfrac{x-3\pi/2}{\epsilon} \; e_1, & \vert 3\pi/2-x_2 \vert \le \epsilon,
\end{cases}
\]
we obtain a value of $C = 4\pi/3$, and the argument implies a bound on the spreading with constant $\frac{Ce}2 = \frac{2\pi e}3\approx 5.7$.
\begin{figure}
	\begin{subfigure}[b]{.45\textwidth}
    \includegraphics[width = \textwidth]{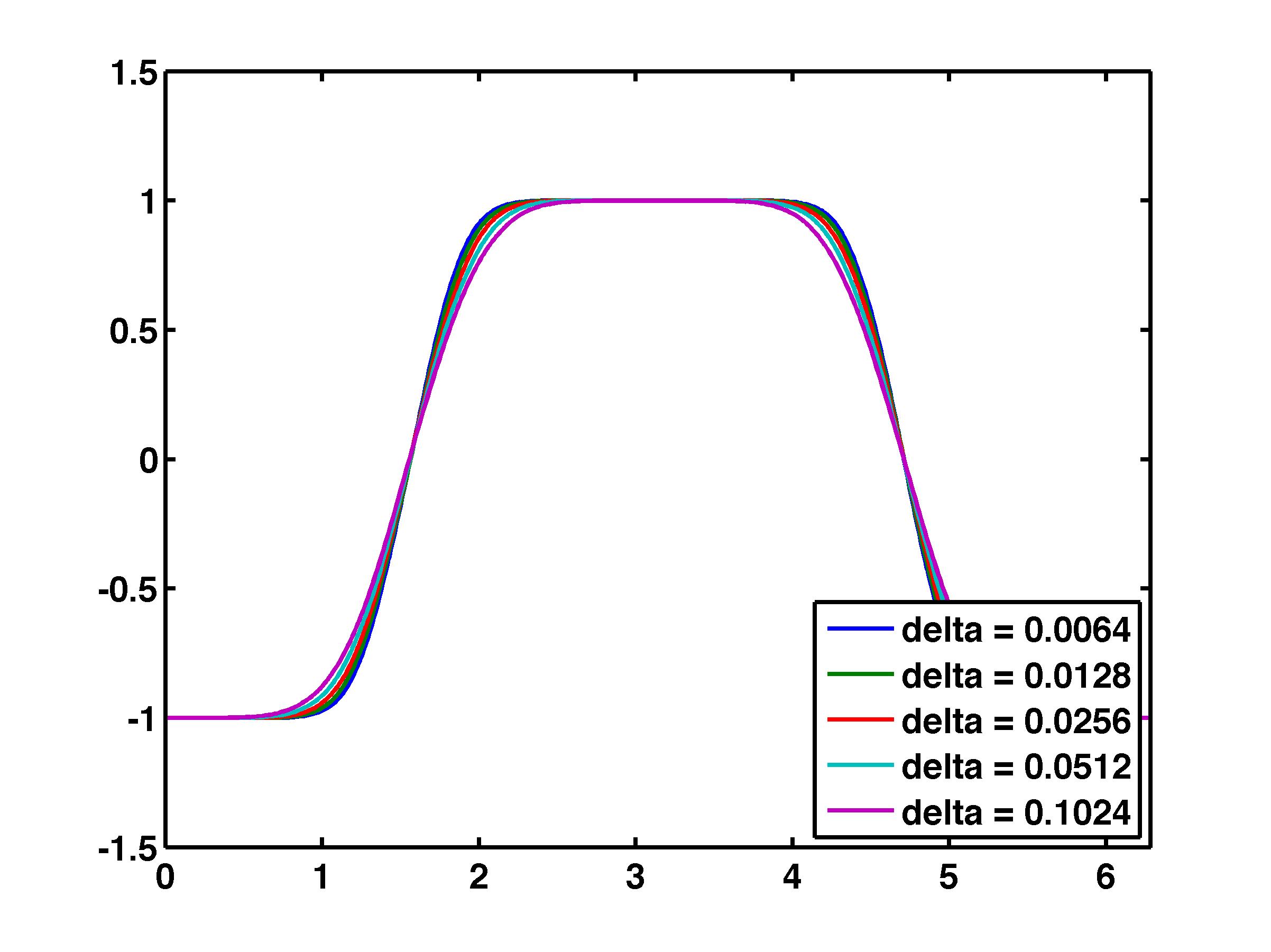}
    \caption{mean}
    \end{subfigure}
    \begin{subfigure}[b]{.45\textwidth}
    \includegraphics[width = \textwidth]{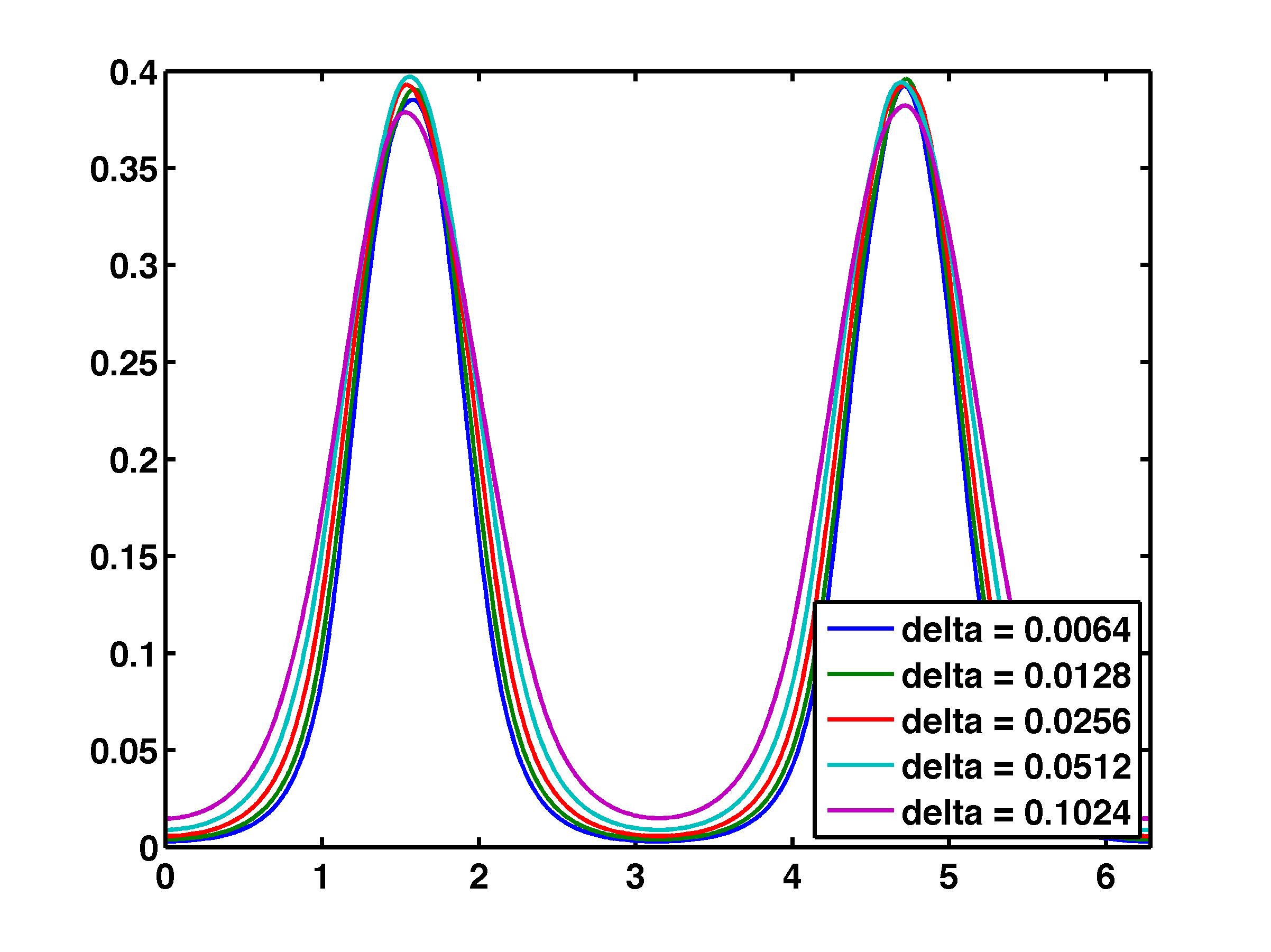}
    \caption{variance}
    \end{subfigure}

	\caption{Flat vortex sheet: 1-D slices of the mean and the variance of the first component computed with different values of $\delta$}
	\label{fig:10}
\end{figure}

\begin{figure}
	\centering
	\begin{subfigure}[b]{.49\textwidth}
		\includegraphics[width=\textwidth]{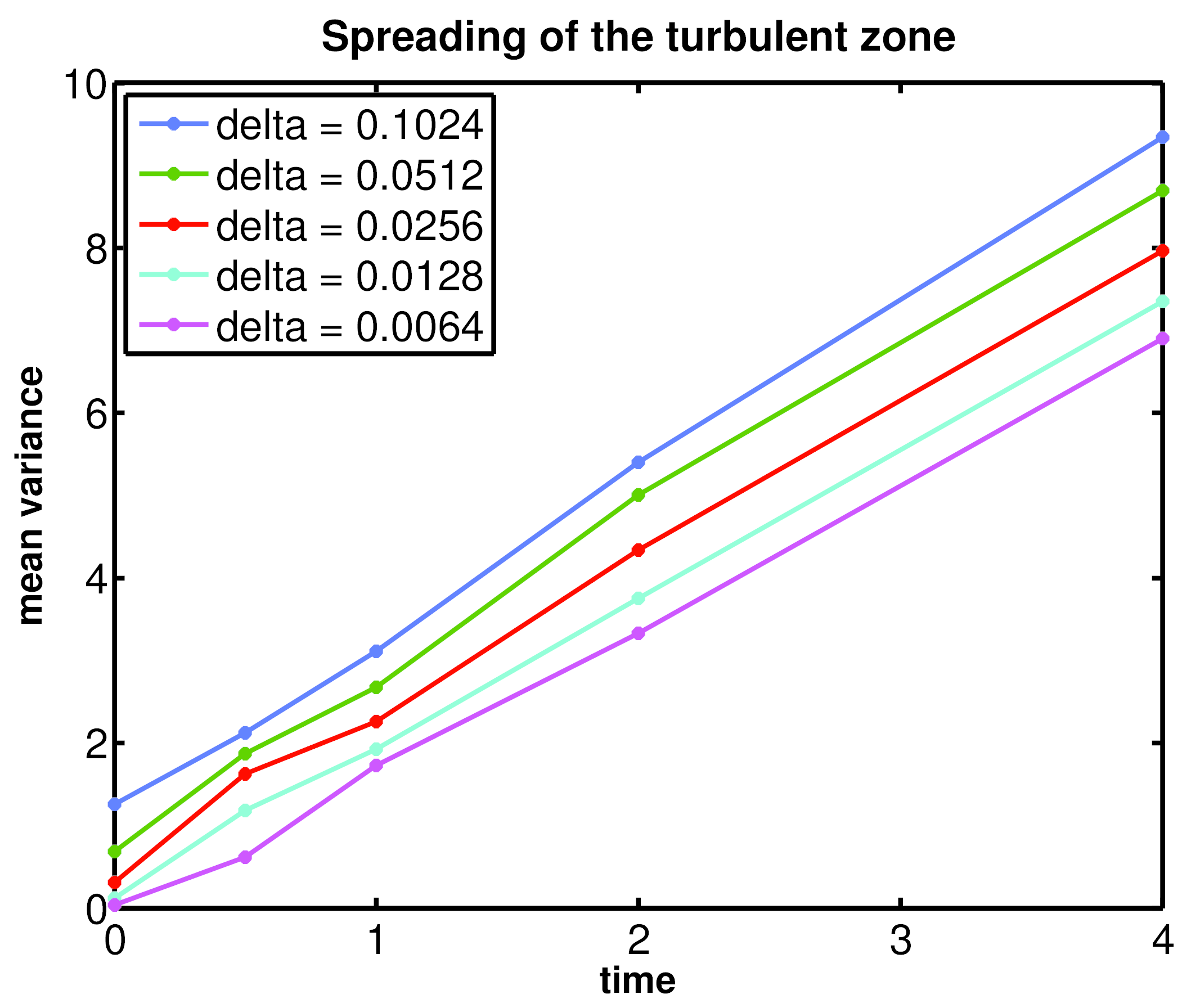}
		\caption{Spreading at different values of time for varying values of $\delta$.}
	\end{subfigure}
	\begin{subfigure}[b]{.49\textwidth}
		\includegraphics[width=\textwidth]{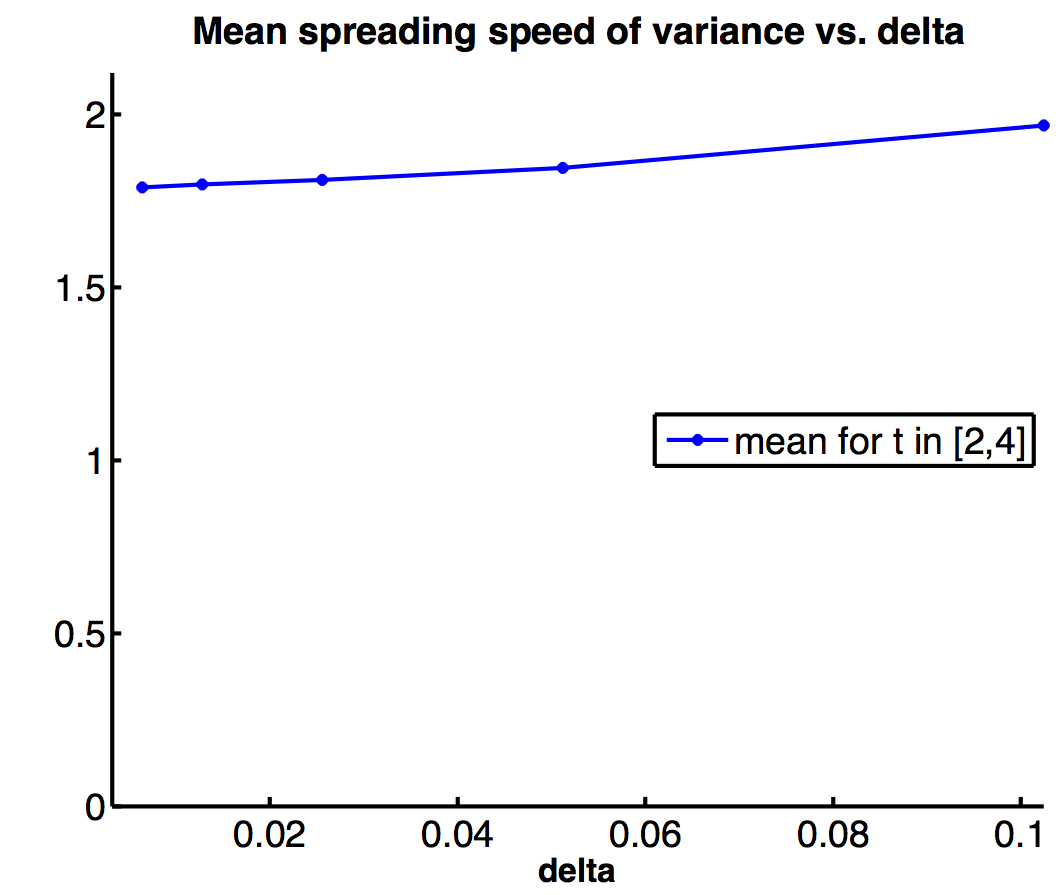}
		\caption{Mean spreading rate against $\delta$ over the time interval $[2,4]$.}
	\end{subfigure}
	\caption{Spreading of the turbulence zone in time.}
	\label{fig:12}
\end{figure}
 
The results of our computation as presented in figure \ref{fig:12} are consistent with the above corollary in establishing that averaged variance (concentrated in the turbulence zone for the flat vortex sheet initial data) does spread out at a rate that is linear in time but at a value of approximately $1.8$ (or about a third of the rigorous upper bound). 
\subsubsection{Probability distribution functions}
As a final test of the non-atomicity of the computed limit measure, we plot the empirical histogram at a point in space and different values of the perturbation parameter $\delta$ over time, in figure \ref{fig:14}. The histograms serve as approximation of the probability density function (pdf), corresponding to the measure valued solution \cite{FKMT1}. The figure shows that the pdfs converge as $\delta \rightarrow 0$. Furthermore, we observe that even if the initial measure is atomic (for small values of the perturbation parameter $\delta$), the resulting pdf is non-atomic at points in the turbulent zone. Thus, we provide considerable evidence that the limiting measure valued solution is non-atomic.
\begin{figure}
	\centering
	 \includegraphics[width =1\textwidth]{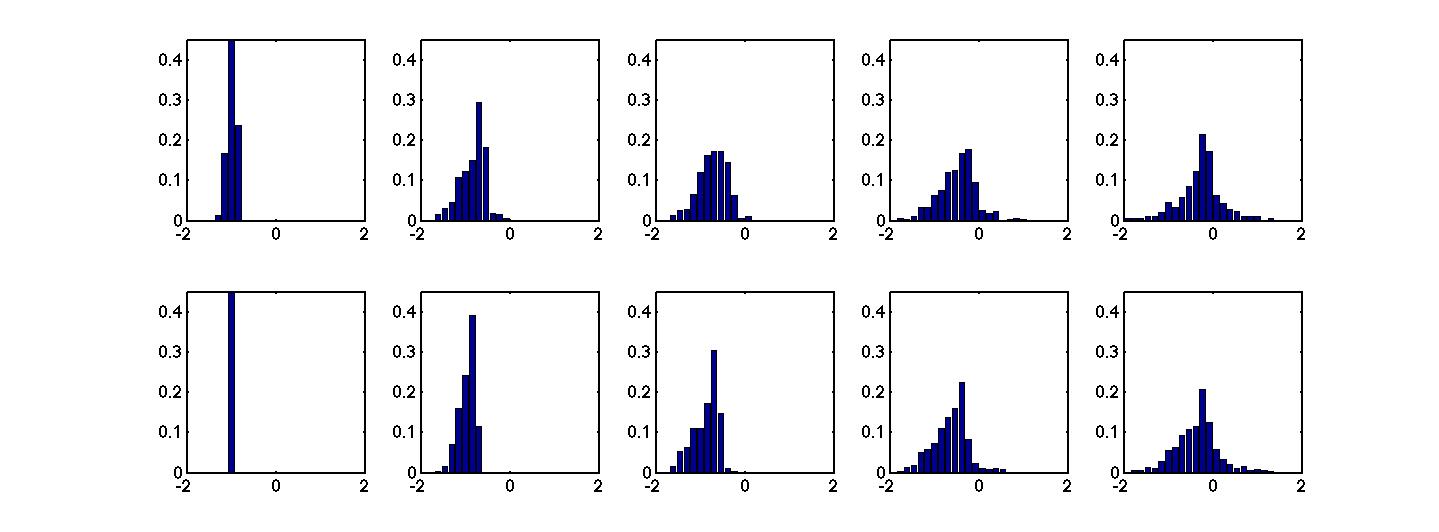}
	\caption{PDFs at a point $x = 2\pi\cdot (0.25,0.77)$ and different times $t=0,0.5,1,2,4$, for different values of $\delta = 0.0512$ (top), $\delta = 0.0064$ (bottom).}
	\label{fig:14}
\end{figure}
\begin{figure}
	 \begin{subfigure}[b]{.49\textwidth}
\includegraphics[width=\textwidth]{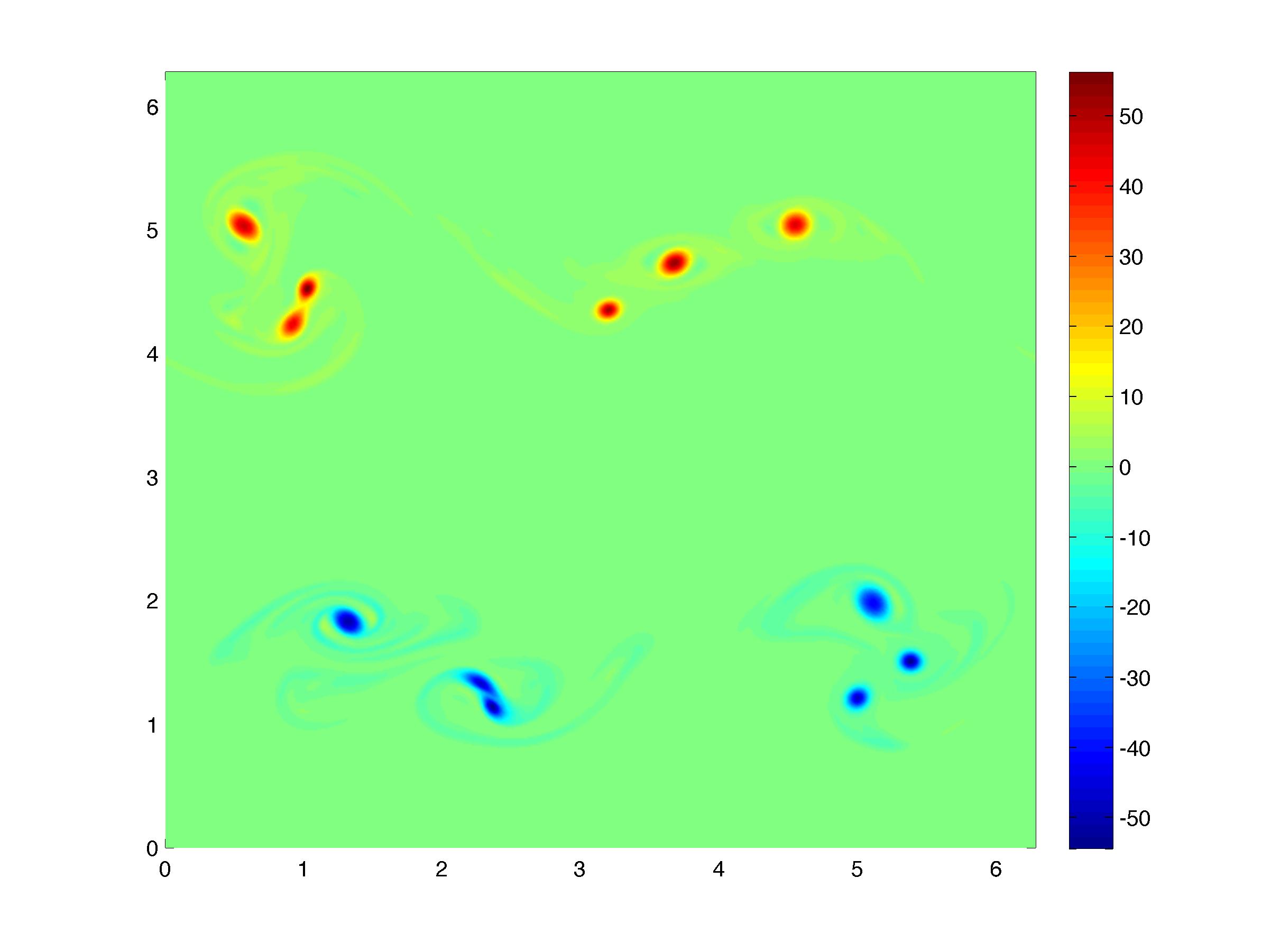}
\caption{$\delta = .0512$}
\end{subfigure}
\begin{subfigure}[b]{.49\textwidth}
\includegraphics[width=\textwidth]{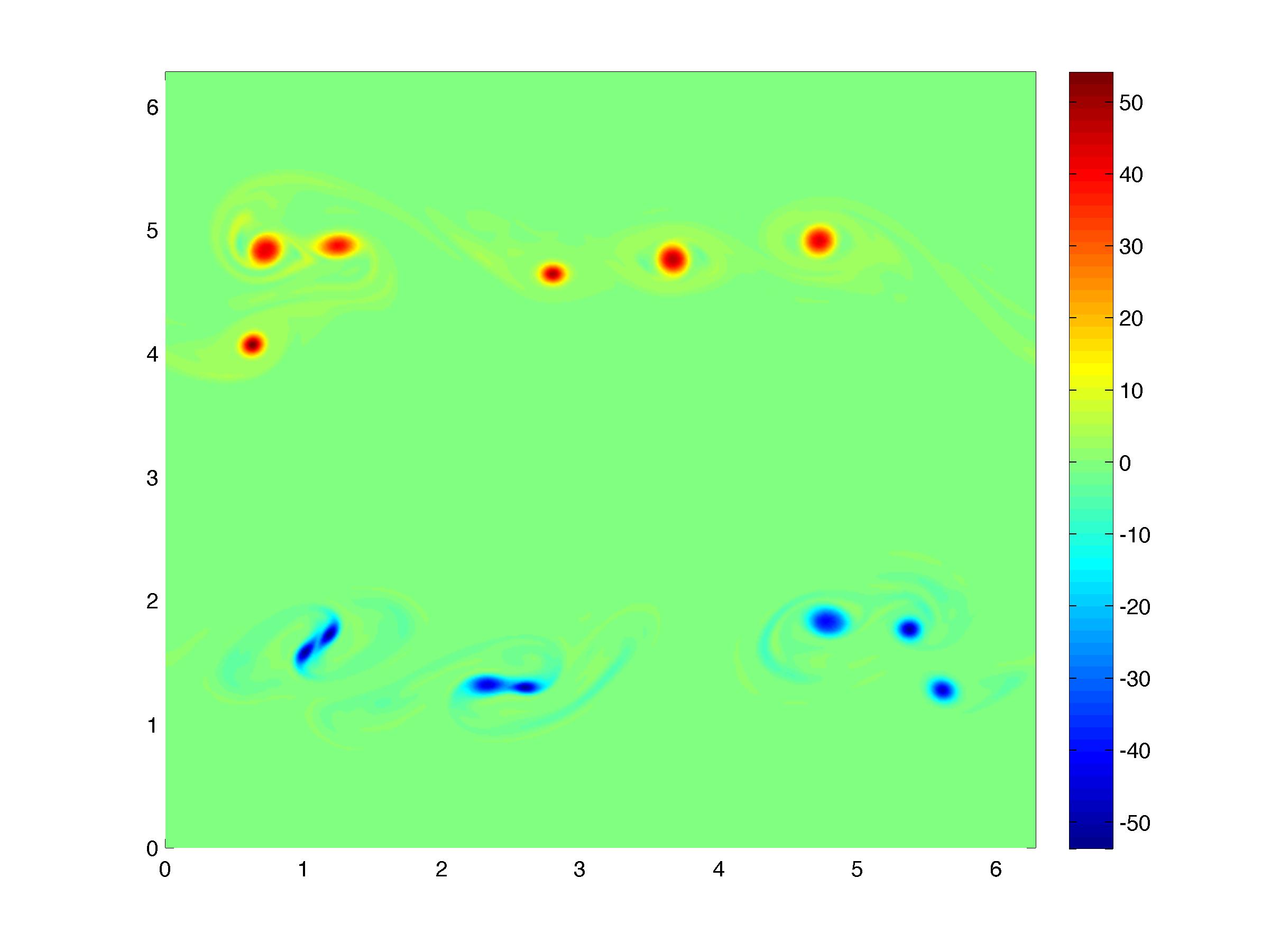}
\caption{$\delta = .0256$}
\end{subfigure} \\
\begin{subfigure}[b]{.49\textwidth}
\includegraphics[width=\textwidth]{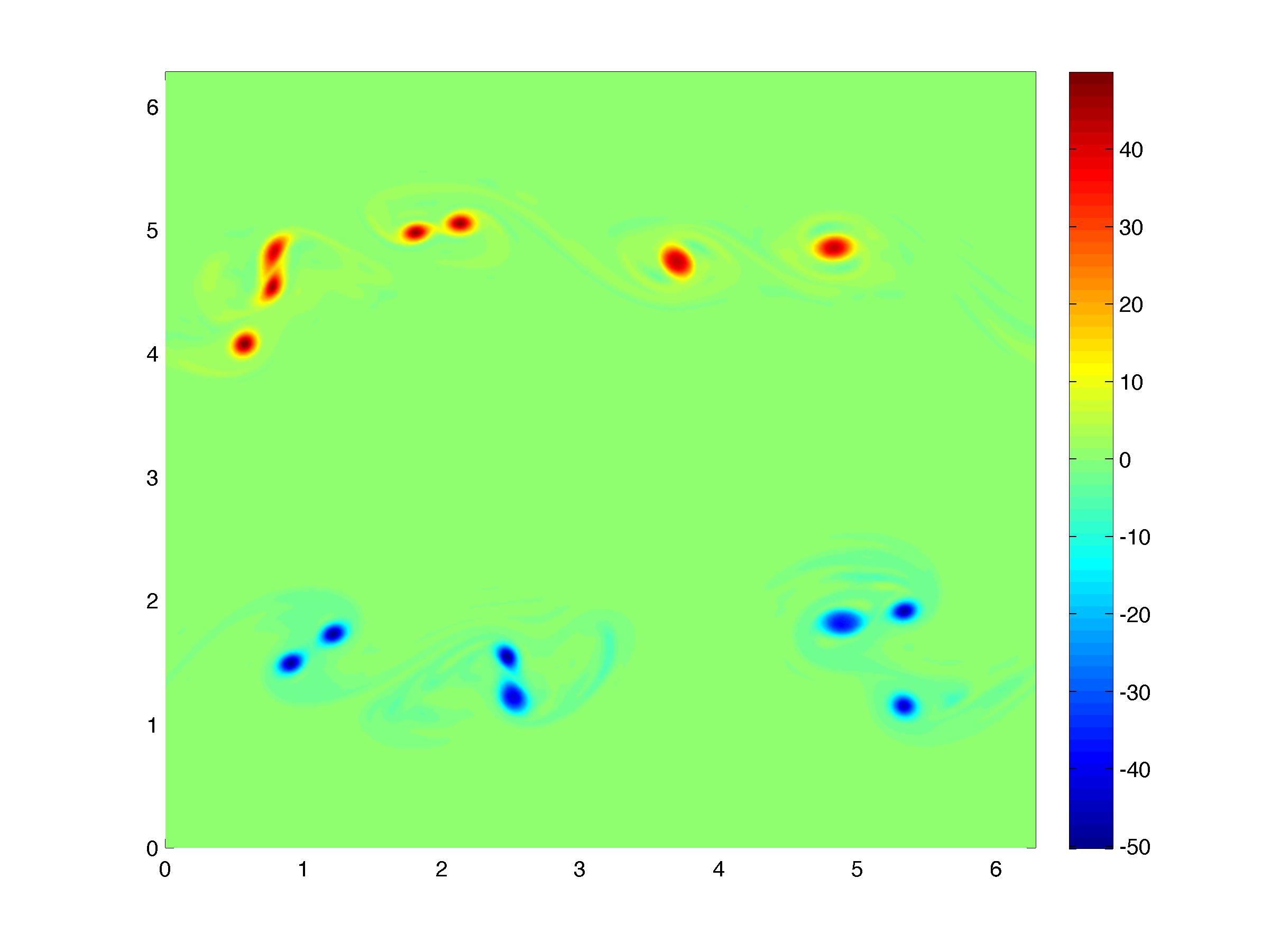}
\caption{$\delta = .0128$}
\end{subfigure}
\begin{subfigure}[b]{.49\textwidth}
\includegraphics[width=\textwidth]{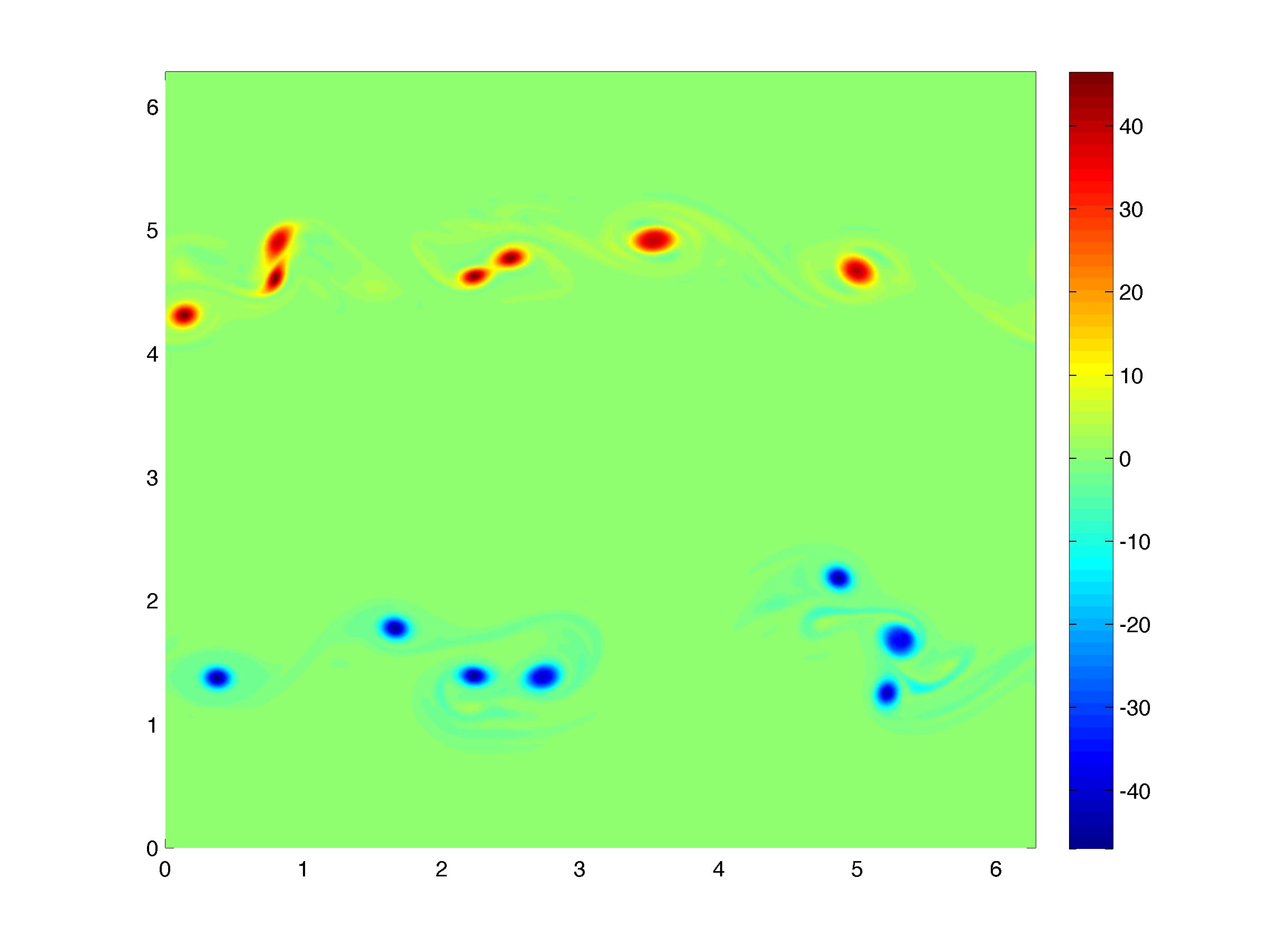}
\caption{$\delta = .0064$}
\end{subfigure}
	\caption{Separation of vortices of different signs}
	\label{fig:15}
\end{figure}
\begin{figure}
	\centering
	\begin{subfigure}[b]{.49\textwidth}
		\includegraphics[width=\textwidth]{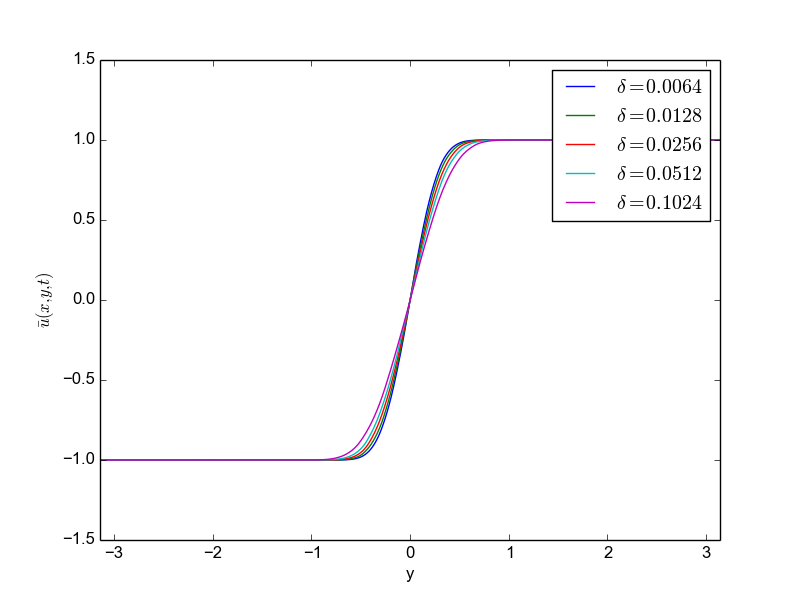}
		\caption{Mean}
	\end{subfigure}
	\begin{subfigure}[b]{.49\textwidth}
		\includegraphics[width=\textwidth]{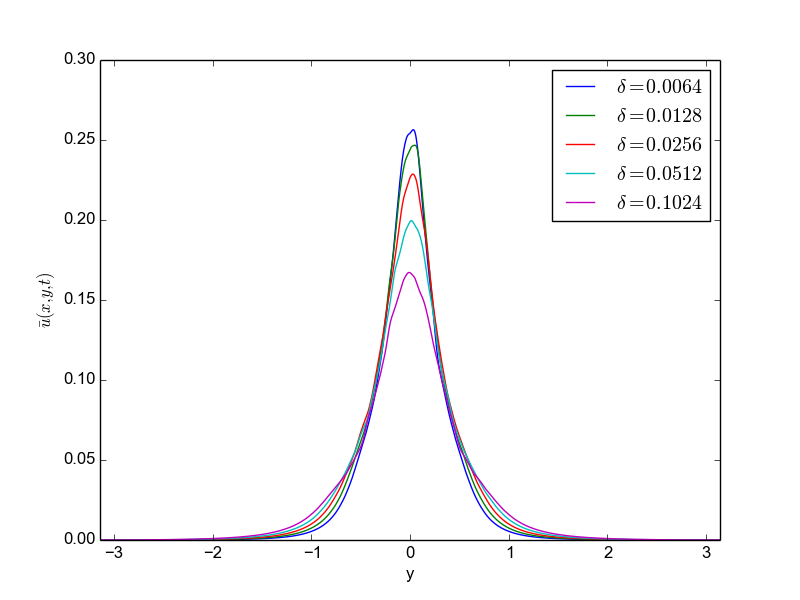}
		\caption{Variance}
	\end{subfigure}
	\caption{Mean and second moment (1-D slices along $x_1 = \pi$) for different values of $\delta$ from \cite{LeoSid1}}
	\label{fig:16}
	\end{figure}

\subsection{Possible non-uniqueness of Delort solutions}
As mentioned in the introduction, Delort in \cite{Del1} showed the \emph{first rigorous existence results} for vortex sheets in two space dimensions. In particular, the result pertained to the following class of velocity fields,
\begin{definition}
\label{def:v1}
A vector field $v \in L^\infty([0,T];L^2(\bbT^2;\bbR^2))$ will be said to belong to the Delort class, if the vorticity $\eta = \curl v$ is a bounded measure of distinguished sign i.e, $\eta \in H^{-1}(\bbT^2)\cap \mathcal{BM}_+$.
\end{definition}
Delort proved the following celebrated result,
\begin{theorem}
\label{thm:del1}
\emph{\cite{Del1}}: Under the assumption that the initial data is in the Delort class, as defined above. 
There exists a weak solution $v$ of the 2-D incompressible Euler equations \eqref{EQ}, that also belongs to the Delort class.
\end{theorem}

The proof is based on mollifying the initial data, resulting in the generation of a sequence of approximate (smooth) solutions to the Euler equations. The resulting vorticity will be of a definite sign as it satisfies a maximum principle. The strong compactness of the approximating sequence is based on a localized $L^1$ control of the vorticity and uses the definite sign of the vorticity in a crucial manner, see also \cite{schochet}.

The uniqueness of the solution constructed by Delort is still open. It turns out that we can use property (4) of theorem \ref{atomicuniqueness} to numerically investigate this interesting question of uniqueness. However, as we consider the Euler equations with periodic boundary conditions in this article, we \emph{cannot restrict ourselves to the Delort class i.e, vorticity being a bounded measure with a distinguished sign}. We need to define the following class of solutions,
\begin{definition}\label{def:del2}
A vector field $v \in L^\infty([0,T];L^2(\bbT^2;\bbR^2))$ will be said to belong to the \emph{extended Delort} class, if the vorticity $\eta = \curl v$ is a bounded measure i.e, $\eta \in H^{-1}(\bbT^2)\cap \mathcal{BM}$.
\end{definition}

The existence proof of Delort in \cite{Del1} can be readily extended to the case of extended Delort class initial data in the sense of definition \ref{def:del2} provided that vortices of opposite sign do not interact with each other. We formalize this argument in the following theorem,
\begin{theorem}
\label{thm:del2}
Let the initial velocity field $v_0$ belong to the extended Delort class as defined above. Further, assume that there exists a constant $c>0$ and a terminal time $T > 0$, such that the time-dependent regions 
 \begin{align*}
 \mathcal D_+ (t) &= \{x\in \bbT^n; \; \exists N\in \bbN, \; \eta_N(x,t) >  0 \}, \\
 D_{-} (t) &= \{x\in \bbT^n; \; \exists N\in \bbN, \; \eta_N(x,t) < 0 \},
 \end{align*}
 satisfy 
 \begin{equation}
 \label{eq:sep}
 \mathrm{dist}({\mathcal D}_+(t), {\mathcal D}_-(t)) \geq c, \quad \forall t \in [0,T],
 \end{equation}
 then there exists a weak solution $v$ of the incompressible Euler equations \eqref{EQ} that belongs to the extended Delort class \ref{def:del2}. 
 \end{theorem} 
The proof follows from a straightforward repetition of the arguments of the proof of theorem \ref{thm:del1} in \cite{Del1} and \cite{schochet}, while replacing the distinguished sign of the resulting vorticity field with assumption \eqref{eq:sep}.  

Next, we will investigate the uniqueness of weak solutions of \eqref{EQ} that belong to the extended Delort class.  To this end, we return to the flat vortex sheet initial data \eqref{initdata} and consider the perturbed random field initial data $X^{0}_{\rho,\delta}$ and the resulting solutions $X_{\rho,\delta}$. We collect some properties of this set of solutions below,
\begin{lemma}
\label{lem:10}
The solutions $X_{\rho,\delta}$ of the 2-D Euler equations \eqref{EQ} with randomly perturbed flat vortex sheet data $X^{0}_{\rho,\delta}$ satisfy for every realization $\omega \in \Omega$: there exist $\rho_k, \delta_k \to 0$ such that
\begin{itemize}
\item $X_{\rho_k,\delta_k}(\omega) \to X(\omega)$ in $C([0,T]; L^2_w(\bbT^2;\bbR^2))$, 
\item $\int \vert \eta_{\rho_k,\delta_k} (\omega) \vert \dx \le C$ uniformly for some constant $C$ with $\eta_{\rho,\delta} = \curl{X_{\rho,\delta}}$
\item  Under the further assumption that vortices of distinguished sign are separated i.e, $\eta_{\rho,\delta}(\omega)$ satisfies \eqref{eq:sep} (uniformly) for all $\omega$, we have a uniform lack of concentration of vorticity, in the sense that
\[\lim_{r\to 0} \sup_{0\le t\le T} \sup_N \int_{B_r(x)} \vert \eta(\omega)_{\rho_k,\delta_k}  \vert \dx = 0, \qquad \forall \; x \in \bbT^2,\]
\end{itemize}
Then for each $\omega \in \Omega$, $X(\omega)$ is a weak solution of the Euler equations that belongs to the \emph{extended Delort class} \ref{def:del2}. Furthermore, 
$$
\lim _{t \rightarrow 0} X(t,\omega) = v_0, \quad {\rm in} \quad L^2(\bbT^2;\bbR^2),
$$
\end{lemma}
The first and second assertions of the above lemma are straight forward consequences of energy conservation and the maximum principle on vorticity for the smooth solutions $X_{\rho,\delta}$. Once, we assume \eqref{eq:sep}, the compactness of the approximating sequence is established by repeating the arguments of the proof of Theorem \ref{thm:del1} as presented in \cite{Del1} and \cite{schochet}. 

We are unable to provide a rigorous proof for the assumption \eqref{eq:sep} in the case of perturbed flat vortex sheet initial data. However, this assumption can be readily verified \emph{a posteriori} in our numerical computations. As an example, we fix a single sample (realization) and present the vorticity, obtained with the spectral method with $N=512$ nodes and $\rho = 0.008$ for different values of $\delta$ at time $T=4$ in figure \ref{fig:15}. The figure clearly shows that the vortices of positive and negative sign for any value of the perturbation parameter $\delta$ are well separated even at this relatively late time $T=4$. In fact, we observe that the time of separation as required by the assumption \eqref{eq:sep} is $T \geq 4$ for all tested $\omega \in \Omega$. Hence, we can assert that each of our realizations (samples) converges (upto a subsequence) to a weak solution of \eqref{EQ} that belongs to the extended Delort class.  This in turn, results in the following statement about the mean of the (admissible) measure valued solution $\nu_{\rho,\delta}$ constructed by the ensemble based algorithm \ref{alg1} applied to the vortex sheet initial data \eqref{initdata},
\begin{lemma}
\label{lem:11}
Let $v_0$ be the flat vortex sheet initial data \eqref{initdata} and $\nu = \lim\limits_{\rho,\delta \rightarrow 0} \nu_{\rho,\delta}$ (in the narrow sense) be an (admissible) measure valued solution of the Euler equations \eqref{EQ}, corresponding to this atomic initial data $\delta_{v_0}$. Further, if we assume that the solutions of the Euler equations that belong to the \emph{extended Delort class} \ref{def:del2} are unique, then 
$$
\langle \nu_{x,t}, \xi \rangle = v_0(x) ,\quad {\rm in} \quad L^2(\bbT^2;\bbR^2).
$$
\end{lemma}
\begin{proof}
Clearly, the flat vortex sheet $v_0$ is a stationary weak solution of the Euler equations that belongs to the extended Delort class \ref{def:del2} for all time $T \in [0,\infty)$. Under our assumption of uniqueness, $v_0$ is the unique weak solution in this class. Therefore, for every $\omega \in \Omega$, we can extract a further subsequence of $X_{\rho_k,\delta_k}(\omega)$converging weakly to the unique solution $v_0$. The uniqueness of the weak limit in turn implies that we in fact must have $\lim_{\rho, \delta \to 0} X_{\rho,\delta}(\omega) = v_0$ in the weak $L^2$ sense. From this, and the fact that the $X_{\rho,\delta}$ are uniformly bounded in the $L^2$ norm, we obtain that for any test function $\phi$, we have
\begin{align*}
\int_{\bbT^2\times [0,\infty)} \langle \nu_{x,t}, \xi \rangle \phi(x,t) \dx\dt  
&= \lim_{\rho,\delta \to 0} \int _{\bbT^2\times [0,\infty)} \langle \nu_{x,t}^{\rho,\delta}, \xi \rangle \cdot \phi \dx\dt \\
&= \lim_{\rho,\delta \to 0}  \int_\Omega  \left( \int _{\bbT^2\times [0,\infty)}  X_{\rho,\delta}(\omega) \cdot \phi  \dx\dt \right) \; dP(\omega) \\
&=  \int_\Omega  \lim_{\rho,\delta \to 0} \left(\int _{\bbT^2\times [0,\infty)}  X_{\rho,\delta}(\omega) \cdot \phi  \dx\dt \right) \; dP(\omega) \\
&=  \int_\Omega \int _{\bbT^2\times [0,\infty)}  v^0 \cdot \phi  \dx\dt \; dP(\omega) \\
&= \int _{\bbT^2\times [0,\infty)}  v^0 \cdot \phi  \dx\dt.
\end{align*}
We have used the uniform bound on 
$$\int _{\bbT^2\times [0,\infty)}  \vert X_{\rho,\delta}(\omega) \cdot \phi \vert  \dx\dt \le \Vert X_{\rho,\delta} \Vert \Vert \phi \Vert \le C\Vert \phi \Vert ,$$
to justify passing to the limit inside of the $dP$-integral.
Hence $\langle \nu_{x,t}, \xi \rangle = v^0(x)$ for any possible limiting measure-valued solution. 
\end{proof}
We use the admissibility of measure valued solutions to show the following,
\begin{lemma} \label{MVSvar}
Let $\nu$ be an admissible measure-valued solution to the Euler equations \eqref{MVS} with atomic initial data. If the barycenter $\overline \nu(x,t) = \langle \nu_{x,t}, \xi\rangle$ is an energy conserving weak solution to \eqref{EQ}, then $\nu_{x,t} = \delta_{\overline \nu(x,t)}$ is atomic.
\end{lemma}
\begin{proof}
We have the following decomposition of the energy $E(t)$ at time $t$:
\begin{align*}
E(t) 
&= \frac12 \int_{\bbT^n}   \vert \overline \nu(x,t) \vert^2 \dx + \frac12 \int_{\bbT^n} \langle \nu_{x,t},  \vert \xi - \overline \nu(x,t) \vert^2 \rangle \dx + \lambda_t(\bbT^n) \\
&= \frac12 \int_{\bbT^n}   \vert \overline \nu(x,t) \vert^2 \dx + \frac 12 \mathrm{Var}_t(\nu) + \lambda_t(\bbT^n).
\end{align*}
The admissibility assumption $E(t) \le E(0)$ combined with the assumption of energy conservation for $\overline \nu(x,t)$ now yields
\begin{align*}
E(0) = \frac12 \int_{\bbT^n}   \vert \overline \nu(x,0) \vert^2 \dx =  \frac12 \int_{\bbT^n}   \vert \overline \nu(x,t) \vert^2 \dx \le E(t) \le E(0).
\end{align*}
Thus, all inequalities in these estimates are equalities. In particular, this implies that $\mathrm{Var}(\nu)=0$ and $\lambda = 0$, hence $\nu_{x,t} = \delta_{\overline \nu(x,t)}$ a.e..
\end{proof}
We combine the above two lemmas to obtain the following theorem about the measure valued solutions corresponding to the flat vortex sheet initial data,
\begin{theorem} \label{delortuniqueness}
If the stationary solution $v_0$ is unique in the extended Delort class of flows with vorticity $\omega \in H^{-1}(\bbT^2)\cap \mathcal{BM}$ and  the (admissible) measure valued solutions $\nu_{\rho,\delta}$ are constructed by applying algorithm \ref{alg1}, then we have $\nu^{\rho,\delta} \wto \delta_{v^0}$ (narrowly sub-sequentially) as $\rho,\delta \to 0$.
\end{theorem}

The main conclusion of all the above arguments is that if the week solutions of the Euler equations were unique in the extended Delort class, then the measure valued solution, computed using algorithm \ref{alg1} would be \emph{an atomic measure concentrated on the initial flat vortex sheet}. However, we provided considerable numerical evidence in sub-section \ref{sec:nonatomic} that the computed solutions are \emph{non-atomic}. In fact, the turbulence zone (region where the variance is non-zero) increases linearly in time. Thus, we conclude that the {\bf weak solutions in the extended Delort class are not unique}. 
\begin{remark}
We have to consider the extended Delort class in this paper due to the restriction of periodic boundary conditions.  However, in a forthcoming article \cite{LeoSid1}, the authors consider a projection-finite difference approach to compute measure valued solutions in non-periodic domains. Consequently, we can explore the uniqueness in the Delort class (of vorticities with definite sign) itself. As an example, we can consider the box $[-\pi,\pi]^2$ and initial data,
\begin{equation}
\label{eq:vs10}
v_0(x) = \begin{cases}
(-1,0), &{\rm if}~ x_2 \leq 0, \\
(1,0),  &{\rm if} ~ x_2 > 0.
\end{cases}
\end{equation}
The algorithm \ref{alg1} is applied to this initial data, together with a projection-finite difference method replacing the spectral method in Step 2 of the algorithm \ref{alg1}. The above initial data is clearly a stationary solution of \eqref{EQ} that belongs to the Delort class. The arguments of Theorem \ref{delortuniqueness} can be repeated to show that the computed measure valued solution should be atomic and concentrated on the initial data. We present a figure from \cite{LeoSid1} as figure \ref{fig:16}, where 1-D slices (in $x_2$ direction) for the mean and the variance are presented. Clearly, the variance converges to a non-zero value with a well-defined turbulence zone. Thus, the figure shows that weak solutions, even in the more restricted Delort class, may be non-unique.
\end{remark}
	
\subsubsection{Comparison with the admissible weak solutions of Szekelyhidi.} In \cite{Sz1}, Szekelyhidi was able to construct infinitely many admissible (finite kinetic energy) weak solutions to the 2-D Euler equations for the flat vortex sheet \eqref{initdata}. Although admissible, these weak solutions are highly oscillatory. Hence, they may not belong to the (extended) Delort class as the resulting vorticity is no longer a bounded measure. 

The single samples that we consider lie in the (extended) Delort class but converge to non-unique weak solutions. Furthermore, the computed measure valued solution has a turbulence zone (patches of non-zero variance) that spreads linearly in time. This is remarkably analogous to the construction of Szekelyhidi in \cite{Sz1} where a well defined turbulence zone is also defined and spreads linearly in time. Moreover, the empirical spread rate obtained by us is within the bounds provided by \cite{Sz1}. 
\subsection{Stability (uniqueness) of the computed measure valued solution.}
Admissible (weak) solutions of the Euler equations are not necessarily unique \cite{CDL1,Sz1}. Furthermore, the numerical evidence in the last subsection suggests that even weak solutions, restricted to the considerably narrower Delort class, may not be unique. Since, every weak solution is also a measure valued solution, we cannot expect any uniqueness in the wider class of (admissible) measure valued solutions. However, the measure valued solution that we compute by application of algorithm \ref{alg1} is not a generic measure valued solution but is one that is obtained with a very specific construction. Is this solution unique in a suitable sense? Is it stable? We explore these questions in the following.
\subsubsection{Stability with respect to different numerical methods} A key step (Step $2$) of algorithm \ref{alg1} requires evolving the initial (perturbed) random field with a spectral (viscosity) method. We can replace the spectral method used here with some other consistent numerical method. To do so, we use a projection-finite difference method, constructed in a forthcoming paper \cite{LeoSid1}. This method is very similar to the classical projection method of Chorin \cite{Cho2}, \cite{BCG1}. We show the mean and the variance computed by this method, as compared to the spectral (viscosity) method, for a fixed $(\rho,\delta) = (0.001,0.01)$ and with $N=1024$ Fourier modes. The corresponding results with the projection-finite difference method are obtained on a $1024 \times 1024$ grid (to obtain similar resolution). The mean (of the first component) and the second moment (of the second component) of the approximate Young measure are shown in figure \ref{fig:17}. Comparing these results to the corresponding results obtained with spectral method (see figures \ref{fig:18}, \ref{fig:19}), we observe that there is very little difference in the statistical quantities computed with two very different numerical methods. Similar agreement was observed for different values of the regularization parameters, indicating that the computed measure valued solution is not sensitive to the choice of the underlying numerical method.
\begin{figure}
	\centering
	\begin{subfigure}[b]{.49\textwidth}
		\includegraphics[width=\textwidth]{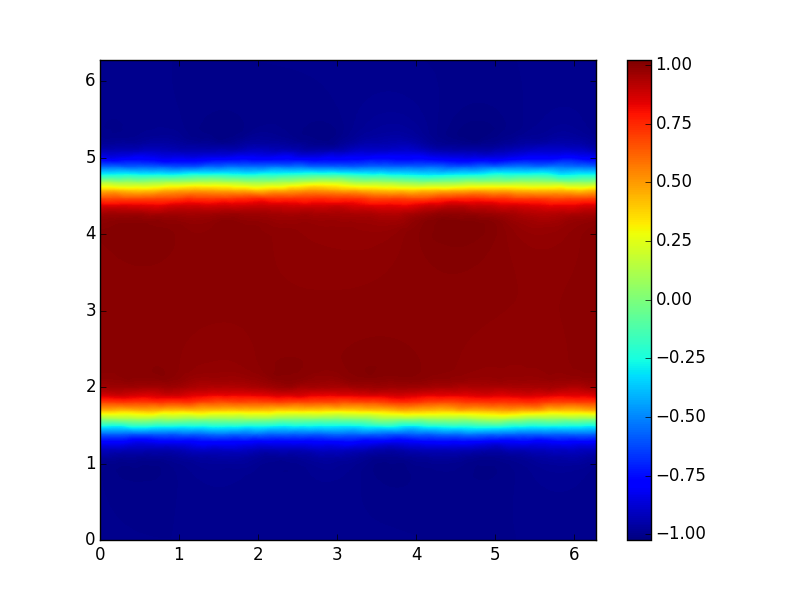}
		\caption{Mean}
	\end{subfigure}
	\begin{subfigure}[b]{.49\textwidth}
		\includegraphics[width=\textwidth]{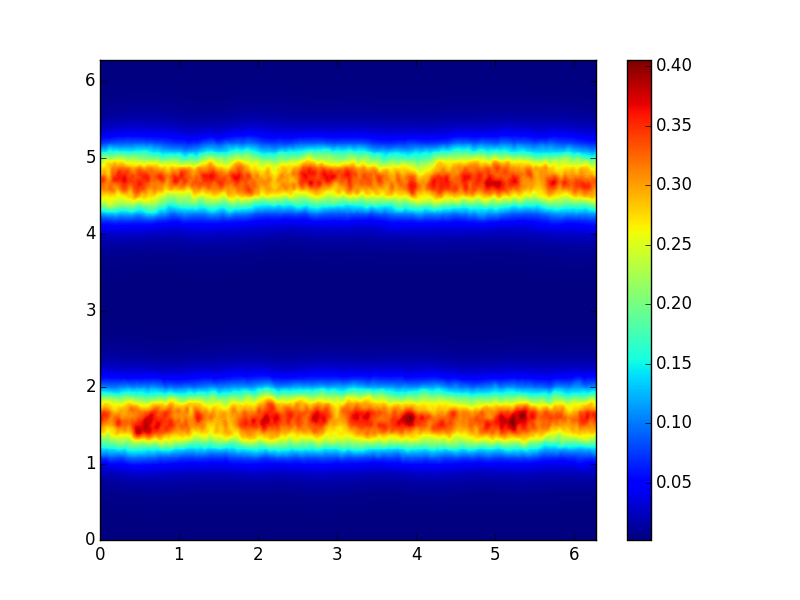}
		\caption{Second moment}
	\end{subfigure}
	\caption{Flat vortex sheet: Left (mean) Right (second moment) from \cite{LeoSid1} at two different times to compare with the spectral method}
	\label{fig:17}
\end{figure}

\begin{figure}
	\begin{subfigure}[b]{.24\textwidth}
\includegraphics[width=\textwidth]{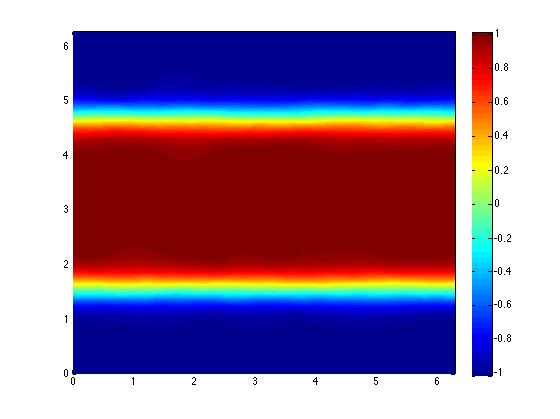}
\caption{interfaces}
\end{subfigure}
\begin{subfigure}[b]{.24\textwidth}
\includegraphics[width=\textwidth]{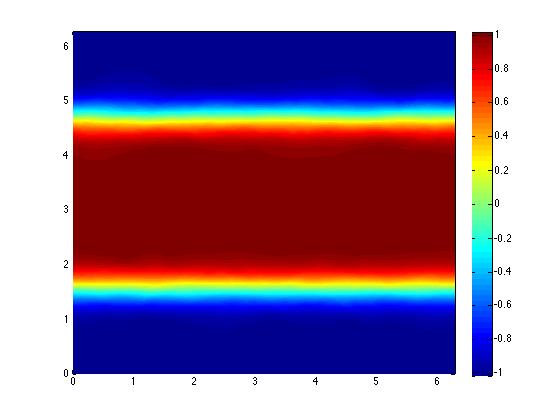}
\caption{uniform}
\end{subfigure} 
\begin{subfigure}[b]{.24\textwidth}
\includegraphics[width=\textwidth]{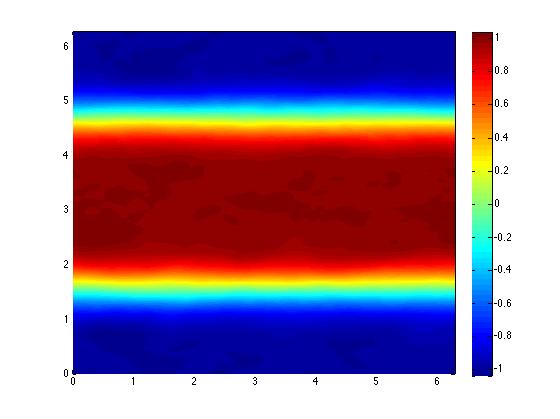}
\caption{uncorrelated}
\end{subfigure}
\begin{subfigure}[b]{.24\textwidth}
\includegraphics[width=\textwidth]{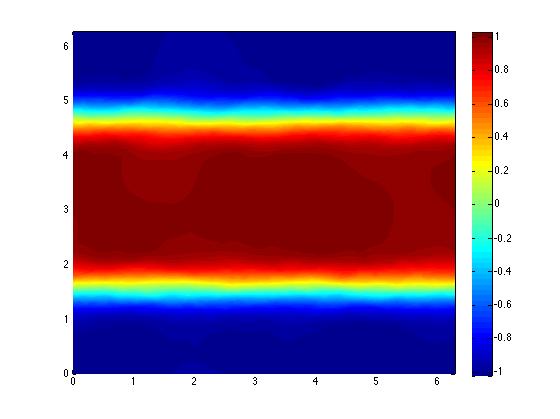}
\caption{Gaussian}
\end{subfigure}

	\caption{Mean for different types of perturbations}
	\label{fig:18}
\end{figure} 
\begin{figure}
\begin{subfigure}[b]{.24\textwidth}
\includegraphics[width=\textwidth]{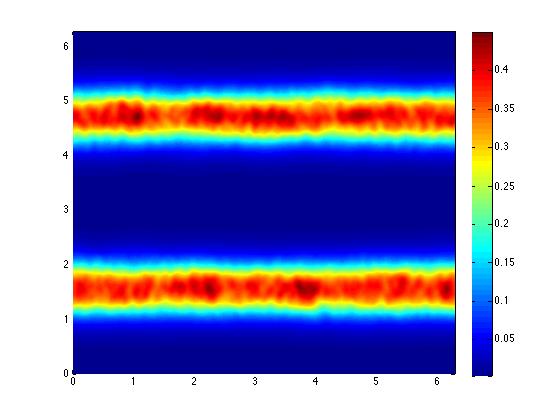}
\caption{interfaces}
\end{subfigure}
\begin{subfigure}[b]{.24\textwidth}
\includegraphics[width=\textwidth]{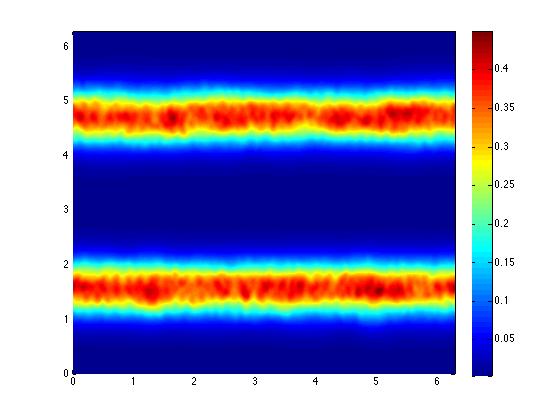}
\caption{uniform}
\end{subfigure} 
\begin{subfigure}[b]{.24\textwidth}
\includegraphics[width=\textwidth]{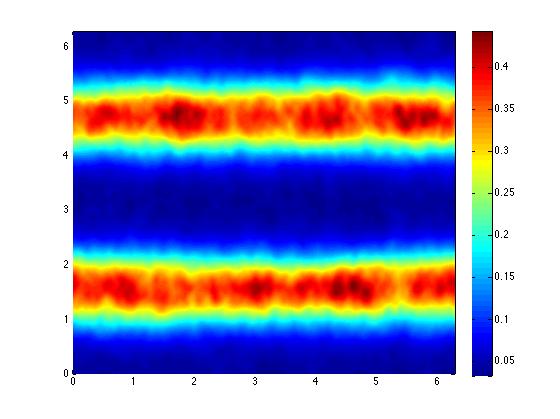}
\caption{uncorrelated}
\end{subfigure}
\begin{subfigure}[b]{.24\textwidth}
\includegraphics[width=\textwidth]{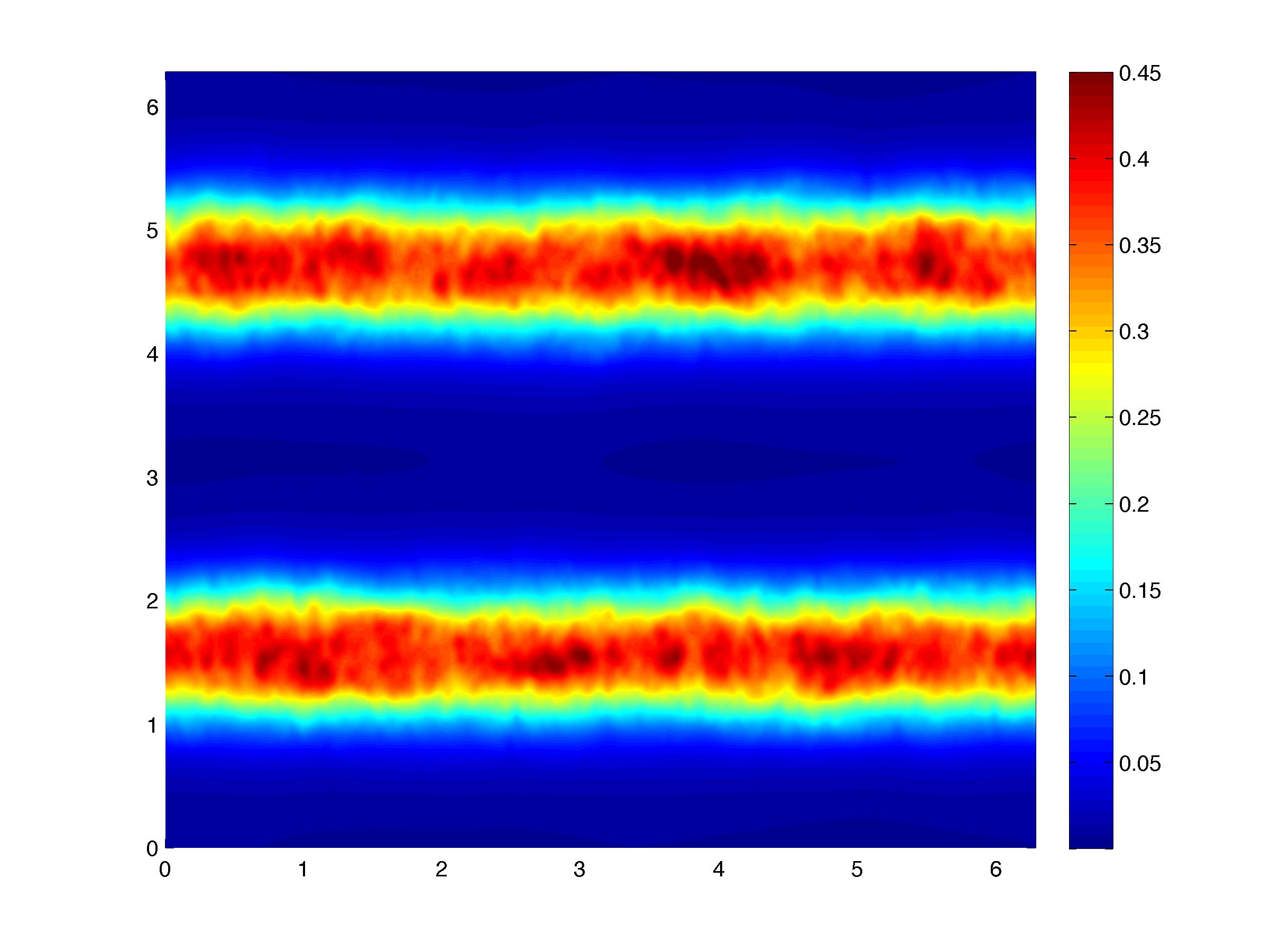}
\caption{Gaussian}
\end{subfigure}
\caption{Second moments for different types of perturbations}
	\label{fig:19}
\end{figure}

\begin{figure}
\begin{subfigure}[b]{.24\textwidth}
\includegraphics[width=\textwidth]{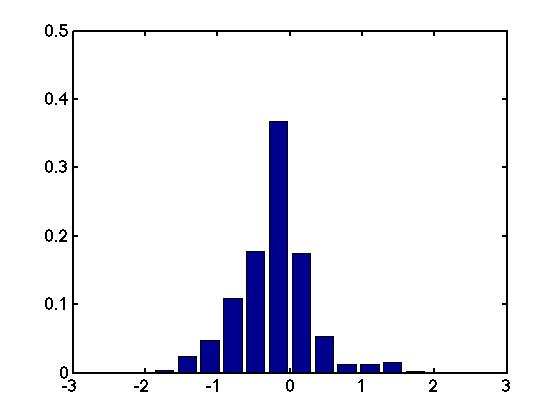}
\caption{interfaces}
\end{subfigure}
\begin{subfigure}[b]{.24\textwidth}
\includegraphics[width=\textwidth]{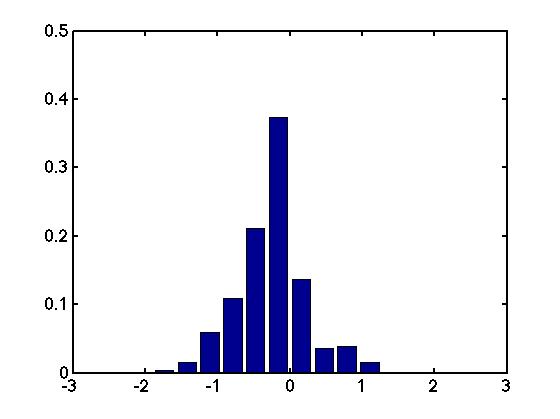}
\caption{uniform}
\end{subfigure} 
\begin{subfigure}[b]{.24\textwidth}
\includegraphics[width=\textwidth]{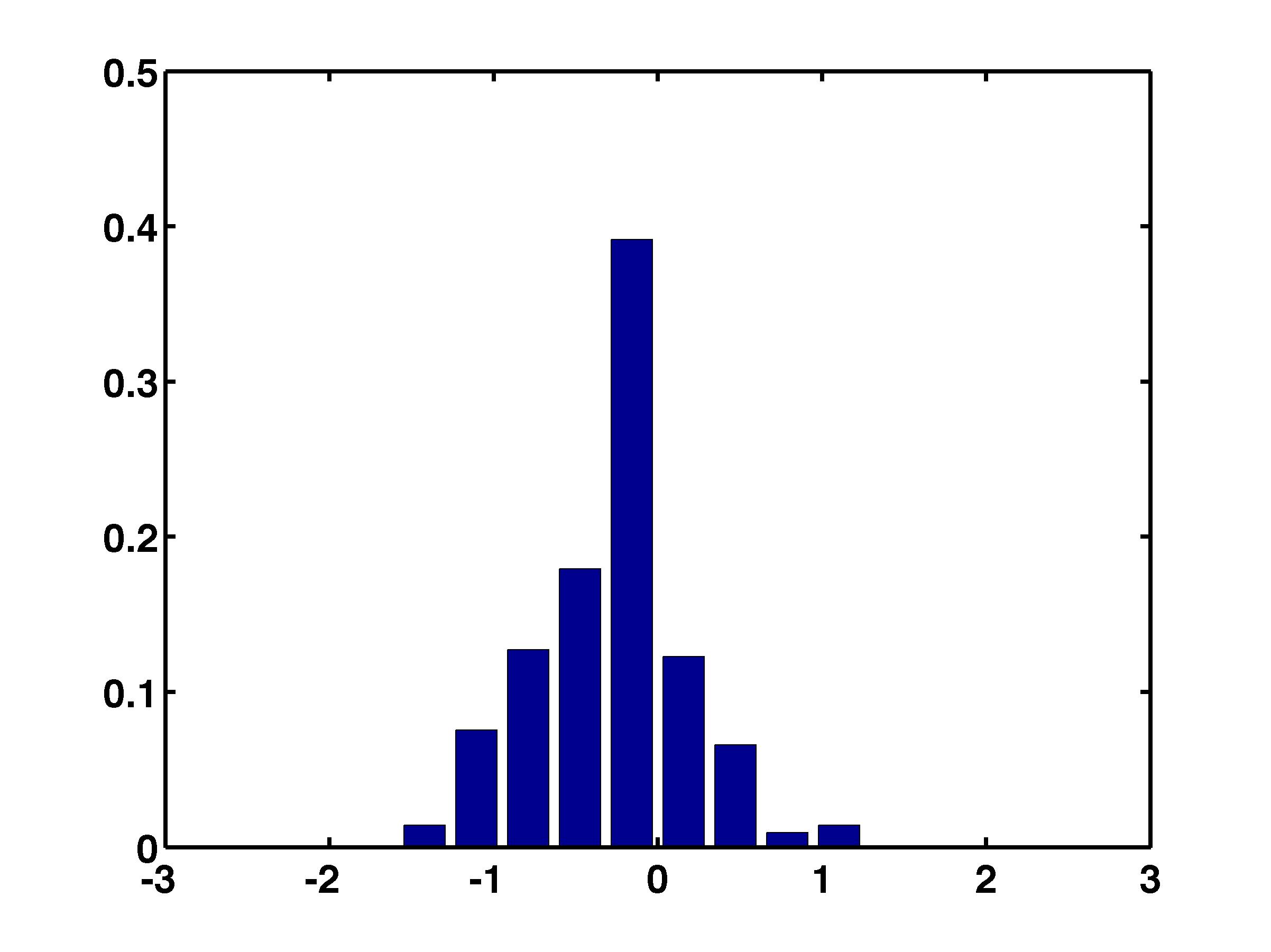}
\caption{uncorrelated}
\end{subfigure}
\begin{subfigure}[b]{.24\textwidth}
\includegraphics[width=\textwidth]{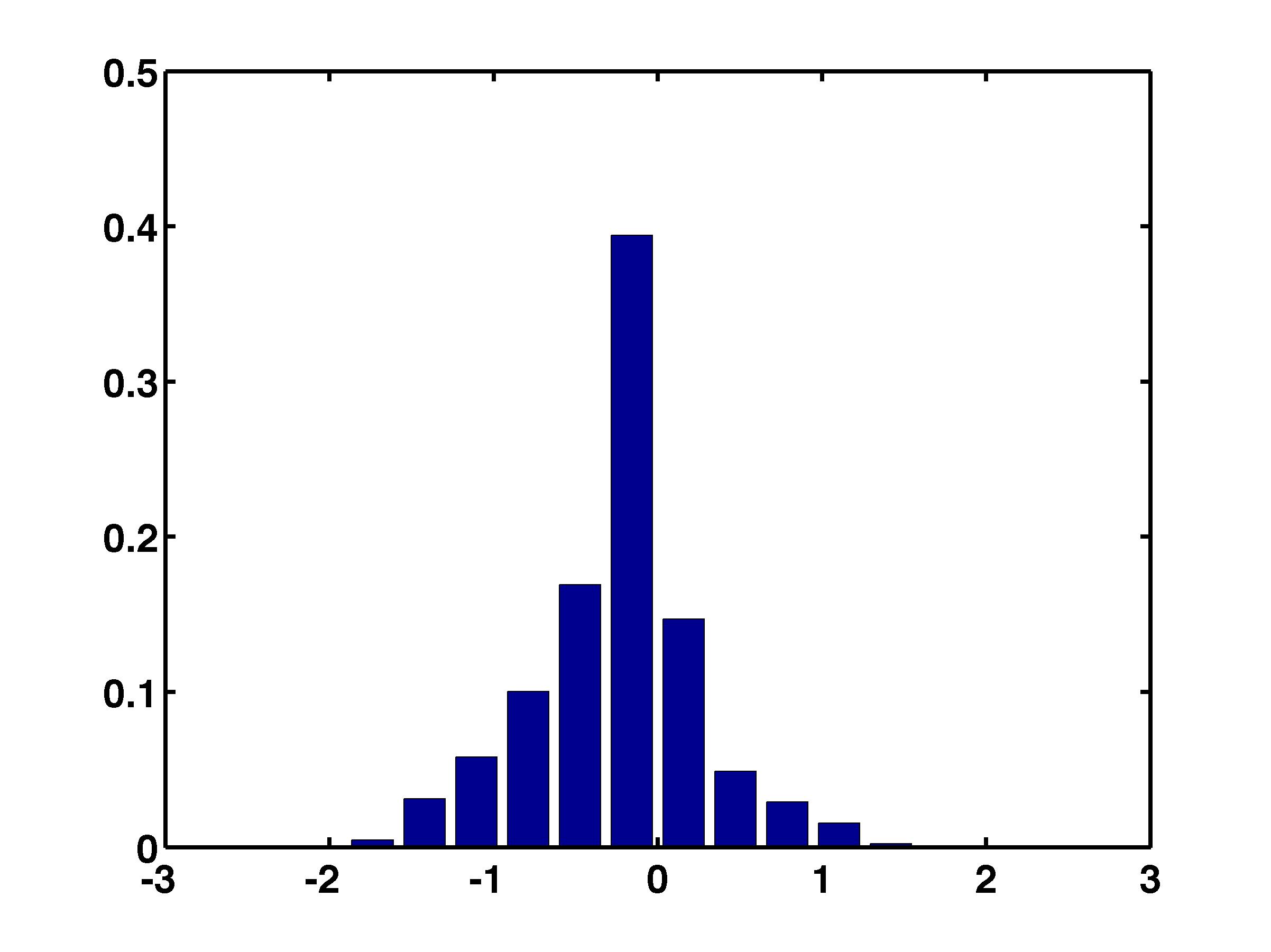}
\caption{Gaussian}
\end{subfigure}
\caption{Different perturbations -- distribution of $x_1$-velocity at a point near the interface, $t = 4$}
\label{fig:20}
\end{figure}

\subsubsection{Stability with respect to different perturbations}
After having demonstrate the robustness of algorithm \ref{alg1} with respect to choice of the numerical method in Step $2$, we investigate if the algorithm is sensitive with respect to the type of perturbations in step $1$. To do this, we consider the most general perturbation to the initial data \eqref{initdata}  by adding a random field that is constant on local patches, and which exhibits uncorrelated fluctuations of equal strength in all of space. More precisely, we considered random fields of the form $X^0 = \sum_{i,j} X^0_{i,j} 1_{{\mathcal C}_{i,j}}$, where the patches are
$${\mathcal C}_{i,j} = \{(x,y) \in \bbT^2 : ik\Delta x \le x < (i+1)k\Delta x, \,  jk\Delta y \le y < (j+1)k\Delta y \}$$
with $k=16$ comprise $16\times 16$ mesh cells, and the $X^0_{i,j}$ are independent, identically distributed $[-1,1]^2$-valued random variables. We obtain our initial perturbations $Z^0_\delta$ as the projection of $v^0 + \delta X^0$ to the space of divergence free vector fields. We refer to the results obtained from this perturbation procedure as `uncorrelated', below.

Note that we can rewrite the evolution equation for the mean $\overline{\nu}$ of the MVS as,
\[\partial_t \overline \nu + \overline \nu \cdot \nabla \overline \nu + \nabla p = -\div \langle \nu, (\xi - \overline \nu) \otimes (\xi - \overline \nu) \rangle.\]
If the fluctuations of the mean $\overline \nu$ in the neighborhood of any given point are an indication of the fluctuations of $\nu$, then we should expect the relevant contributions to the evolution of $\overline \nu$ to originate at the two interfaces, where $\overline \nu$ has a large jump. Hence, we localize the above uncorrelated perturbation to the initial data by multiplying it with cutoff functions that 
are supported around the two interfaces. We refer to the results from these localizations as `uniform' or `Gaussian' according to the corresponding distribution the values of the $X_{i,j}^0$ were chosen from. The results of applying algorithm \ref{alg1} with these perturbations, with amplitude $\delta = 0.05$ and at time $T=4$ are shown in figures \ref{fig:18} (mean), \ref{fig:19} (second moment) and \ref{fig:20} (pdfs). Clearly the computed solutions are very similar to those computed with the sinusoidal perturbations. Also, the nature of underlying distribution does not seem to affect the computed measure valued solution. 

Summarizing the results of this subsection, we remark that the measure valued solution of the two-dimensional Euler equations, computed with the algorithm \ref{alg1}, are stable with respect to perturbations as well as robust vis a vis the choice of numerical method used to approximate them. This indicates that the computed measures may have MV stability, a weaker stability concept introduced in \cite{FKMT1}. Although stability (uniqueness) does not hold for generic (admissible) measure valued solutions, the solutions computed by algorithm \ref{alg1} do belong to a subset of admissible MVS, within which a suitable notion of stability (uniqueness) may hold. Further elaboration of these ideas is envisaged to be the subject of forthcoming articles.

\section{Discussion}
We consider the incompressible Euler equations \eqref{EQ} governing the motion of inviscid incompressible fluids. No global existence and uniqueness results are available in three dimensions. Wellposedness results in two space dimensions are restricted to smooth initial data and exclude such physically interesting flows like vortex sheets. Although Delort \cite{Del1} was able to show the existence of weak solutions for vortex initial data in 2D, uniqueness of such solutions is still open. Similarly, many different types of numerical schemes are available but rigorous convergence results exist only for special cases.

The starting point of the current article was the observation that even a well established numerical method, like the spectral (viscosity) method may not converge (at least for realistic resolutions), even in 2D. Finer and finer scale oscillations are observed as the resolution is increased.

Given the appearance of structures at smaller and smaller scales on increasing the numerical resolution, we follow a recent paper \cite{FKMT1} and postulate that (admissible) measure valued solutions, introduced by DiPerna and Majda \cite{DM1} are an appropriate solution framework for the incompressible Euler equations, particularly with regard to the stability of initial data and the convergence of approximation schemes. 

Our main aim was to design an algorithm to compute measure valued solutions of the Euler equations in a robust and efficient manner. To this end, we modified the ensemble based algorithm proposed in a recent paper \cite{FKMT1} by combining it with the spectral (viscosity) method. We prove that the resulting approximate Young measures converge to an admissible measure valued solution of the Euler equations as the number of Fourier modes increases and the perturbation parameters converge to zero. 

We present a wide variety of numerical experiments to illustrate the theoretical results on the proposed algorithm. In particular, we focus on an extensive case study for the two-dimensional flat vortex sheet. The numerical experiments reveal that
\begin{itemize}
\item Single realizations (samples) may not converge as the number of Fourier modes is increased. Furthermore, there is no convergence as the perturbation amplitude is reduced indicating instability of the flat vortex sheet, at least at realistic numerical resolutions.
\item Statistical quantities of interest such as the mean and variance do converge as the predicted by the theory.
\item Furthermore, the approximate Young measure is observed to converge with respect to the Wasserstein metric on the space of probability measures, indicating a considerably stronger form of convergence of the approximate Young measures than the predicted narrow convergence.
\item The computed measure valued solution is robust with respect to the choice of numerical method as well as the nature of the initial perturbations, suggesting stability of the computed measure valued solution in a suitable sense, for instance in the sense of MV stability of \cite{FKMT1}
\item The computed measure valued solution is non-atomic. The variance is concentrated (spatially) into two patches, symmetric with respect to the line $x_2 = \pi$. This \emph{turbulence zone} spreads in time at a linear rate and is consistent with a theoretical upper bound. 
\end{itemize}

We show analytically that if the weak solutions of the Euler equations are unique in the (extended) Delort class, i.e, the vorticity is a bounded measure, then the resulting measure valued solution, corresponding to the flat vortex sheet, will be atomic and concentrate on the initial data. However, given the observed non-atomicity of the measure, we conclude that the \emph{weak solutions belonging to the Delort class may not be unique}.  This numerical evidence provides a new perspective on an interesting open question and calls for further theoretical investigation. 

It is educative to reevaluate the question of uniqueness (consequently regularity and stability) of solutions of the two-dimensional Euler equations in light of our computations. Smooth solutions (those with smooth initial data) are clearly unique. Whereas, admissible weak solutions (those with finite kinetic energy) are not unique by the results of \cite{CDL1,Sz1}. Delort \cite{Del1} constructed weak solutions whose regularity is between the space of admissible weak solutions ($L^2$) and smooth solutions ($W^{1,\infty}$), namely velocity fields, whose vorticity is a bounded measure (of distinguished sign). Our numerical evidence suggests that these solutions may not be unique. If one starts with initial data where vorticity is a bounded measure, our computations show that numerical approximations converge to a \emph{non-atomic measure valued solution}. The time dynamics are also rich as the initial atomic measure has to burst out (at a linear rate as the average variance grows linearly) into a measure that is non-atomic. Heuristically, it appears as if the non-unique Delort solutions can be put together (weighted) into a measure and this measure could well be the physically relevant solution of the two-dimensional Euler equations with singular initial data. Its stability with respect to perturbations is indicated by our numerical results. It would be interesting to see if this measure can be characterized by some principles of statistical mechanics, such as those explored in \cite{Rob1}.

The results of this paper suggest further exploration on the following aspects,
\begin{itemize}
\item Although our algorithm and convergence results are applicable to three dimensional flows, we only presented two-dimensional results. Employing the ensemble based algorithm to three dimensional flows is being undertaken currently and the relevance  of the concept of measure valued solutions to turbulence will be explored in a forthcoming article.
\item Our algorithm for evaluating averages with respect to the measure valued solution (algorithm \ref{alg:montecarlo}) uses Monte-Carlo approximation of the phase space integrals. Monte Carlo methods are notoriously slow to converge (with respect to the number of samples). Alternative strategies such as the recently developed Multi-level Monte Carlo (MLMC) algorithm \cite{SSSid1} need to extended to compute measure valued solutions, particularly in three space dimensions. The issue of Uncertainty Quantification (UQ) is also germane to this discussion as our formulation automatically allows us to model initial data as probability distributions. In particular, a Multi-level Monte Carlo version of algorithm \ref{alg:montecarlo} can be employed for efficient UQ for incompressible flows, in the framework of measure valued solutions.
\item The characterization of the computed measure valued solution stills need to be performed.  The related question of a rigorous proof of stability (uniqueness) will also be taken up in future articles.
\end{itemize}

\end{document}